\documentclass{amsart}
  \usepackage[fontsize=11pt]{fontsize} 
 \usepackage[a4paper, left=1.4cm, right=1.3cm, top=1cm,bottom=1cm]{geometry}

\usepackage{latexsym}
\usepackage{indentfirst}
\usepackage{amssymb,amsfonts,amsmath,amsthm}
\usepackage{graphicx}
\usepackage{mathrsfs}
\usepackage{array}
\usepackage{color}
\usepackage{epstopdf}
\usepackage[all]{xy}
\usepackage{url}
\usepackage{enumitem}
 \usepackage[normalem]{ulem} 

\usepackage{tikz-cd}

\usepackage{color}
\definecolor{ghcolor}{RGB}{0, 150, 200} 
\definecolor{winestain}{rgb}{0.5,0,0}
\usepackage[linktocpage=true, hidelinks, colorlinks, linkcolor=blue, citecolor=magenta, bookmarksopen=true, urlcolor=brown]{hyperref}

 \newtheorem{thm}{Theorem}[section]
\newtheorem{theorem}[thm]{Theorem}
\newtheorem{prop}[thm]{Proposition}

\newtheorem{lem}[thm]{Lemma}

\newtheorem{lemma}[thm]{Lemma}

\newtheorem{cor}[thm]{Corollary}

\theoremstyle{definition}

\newtheorem{rmk}[thm]{Remark}
\newtheorem{rem}[thm]{Remark}
\newtheorem{remark}[thm]{Remark}
\newtheorem{dfn}[thm]{Definition}
\newtheorem{convention}[thm]{Convention}

\newtheorem{observation}[thm]{Observation}
\newtheorem{defn}[thm]{Definition}
\newtheorem{notation}[thm]{Notation}
\newtheorem{Notation}[thm]{Notation}
\newtheorem{construction}[thm]{Construction}

\numberwithin{equation}{section}

\newcommand{\frakB}{{\mathfrak B}}

\newcommand{\frakS}{{\mathfrak S}}

\newcommand{\bA}{{\mathbb A}}

\newcommand{\bM}{{\mathbb M}}
\newcommand{\bN}{{\mathbb N}}

\newcommand{\bQ}{{\mathbb Q}}

\newcommand{\bZ}{{\mathbb Z}}

\newcommand{\calA}{{\mathcal A}}

\newcommand{\calI}{{\mathcal I}}

\newcommand{\calM}{{\mathcal M}}

\newcommand{\calO}{{\mathcal O}}

\newcommand{\calS}{{\mathcal S}}

\newcommand{\rC}{{\mathrm C}}

\newcommand{\rd}{{\mathrm d}}

\newcommand{\rH}{{\mathrm H}}

\newcommand{\rR}{{\mathrm R}}

\newcommand{\rW}{{\mathrm W}}

\newcommand{\Zp}{{\bZ_p}}

\newcommand{\Qp}{{\bQ_p}}

\newcommand{\Ainf}{{\mathrm{A_{inf}}}}

\newcommand{\End}{{\mathrm{End}}}           
\newcommand{\id}{{\mathrm{id}}}             
\newcommand{\Ima}{{\mathrm{Im}}}            
\newcommand{\Ker}{{\mathrm{Ker}}}           
\newcommand{\Rep}{{\mathrm{Rep}}}           
\newcommand{\RGamma}{{\mathrm{R\Gamma}}}    
\newcommand{\Rlim}{{\mathrm{R}\underleftarrow{\lim}}} 
\newcommand{\Spa}{{\mathrm{Spa}}}           
\newcommand{\Vect}{{\mathrm{Vect}}}          

\newcommand{\GL}{{\mathrm{GL}}}             

\newcommand{\an}{{\mathrm{an}}}             
\newcommand{\cris}{{\mathrm{cris}}}         
\newcommand{\pd} {{\mathrm{pd}}}                           
\newcommand{\perf}{\mathrm{perf}}           

\usepackage{relsize}
\usepackage[bbgreekl]{mathbbol}
\DeclareSymbolFontAlphabet{\mathbb}{AMSb} 
\DeclareSymbolFontAlphabet{\mathbbl}{bbold}
\newcommand{\Prism}{{\mathlarger{\mathbbl{\Delta}}}} 

\newcommand{\ya}{{\rangle}}
\newcommand{\za}{{\langle}}

\newcommand{\okprism}{{(\ok)_\prism}}
\newcommand{\okpris}{{(\mathcal{O}_K)_\prism}}

\newcommand{\okprislog}{{(\mathcal{O}_K)_{\prism, \mathrm{log}}}}
\newcommand{\okprisast}{{(\mathcal{O}_K)_{\prism, \ast}}}
\newcommand{\okprisperfast}{{(\mathcal{O}_K)^{\mathrm{perf}}_{\prism, \ast}}}

\newcommand{\opris}{{\mathcal{O}_\prism}}

\newcommand{\baroprism}{{\overline{\O}_\prism}}

 \newcommand{\Strat}{\mathrm{Strat}}

\label{to move around}
 \renewcommand{\O}{{\mathcal{O}}} 
 \renewcommand{\o}{{{\mathcal{O}}}}

\label{Topology, THH, TC, TP}

 \label{colors}

\label{arrows}

\newcommand \into {\hookrightarrow }
\renewcommand \to {\rightarrow}

\renewcommand{\projlim}{\varprojlim}

\label{•}
 
 \label{matrix, groups}

\newcommand{\vect}{\mathrm{Vect}}

\def\Mat{\mathrm{Mat}}

\label{module, homological alg}
 
\newcommand{\rep}{{\mathrm{Rep}}}

\def\inf{{\mathrm{inf}}}
\def\sup{\mathrm{sup}}

\newcommand{\gal}{{\mathrm{Gal}}}

\newcommand{\spf}{{\mathrm{Spf}}}

\def\an{\mathrm{an}}
\def\perf{\mathrm{perf}}

\newcommand{\Lie}{\mathrm{Lie}}

\label{condesend, solid math}

\label{!}
\label{ALL: p-adic Hodge}

\label{LAV}

\newcommand{\la}{{\mathrm{la}}}
\newcommand{\dan}{\text{$\mbox{-}\mathrm{an}$}}

\newcommand{\dla}{\text{$\mbox{-}\mathrm{la}$}}

\newcommand{\dpa}{\text{$\mbox{-}\mathrm{pa}$}}
\renewcommand{\log}{\mathrm{log}}

\newcommand{\Kpinfty}{{K_{p^\infty}}}
\newcommand{\kpinfty}{{K_{p^\infty}}}
\newcommand{\hatkpinfty}{{\widehat{K_{p^\infty}}}}
\newcommand{\Kinfty}{{K_{\infty}}}
\newcommand{\kinfty}{{K_{\infty}}}

\label{LAV: Lie gp Lie alg}

\newcommand{\gammak}{{\Gamma_K}}

\newcommand{\gk}{{G_K}}

\label{period ring: Fontaine}

\newcommand\HT{{\mathrm{HT}}}
\newcommand\Sen{{\mathrm{Sen}}}

\newcommand\sen{{\mathrm{Sen}}}

\newcommand\rig{{\mathrm{rig}}}

\newcommand{\ainf}{{\mathbf{A}_{\mathrm{inf}}}}

\newcommand{\bcrisplus}{{\mathbf{B}^+_{\mathrm{cris}}}}

\newcommand{\bdrplus}{{\mathbf{B}^+_{\mathrm{dR}}}}

\label{period ring: relative}

\label{period ring: Gao-Min-Wang}
 
\newcommand{\nht}{{\mathrm{nHT}}}

\newcommand{\bm}{\mathbb{M}}

\newcommand{\smat}[1]{\left( \begin{smallmatrix} #1 \end{smallmatrix} \right)}
\newcommand{\dacc}[1]{\{\!\{ #1 \}\!\}}

\label{period ring: bb font}

\newcommand{\bbdrplus}{{\mathbb{B}_{\mathrm{dR}}^{+}}}

\newcommand*{\wt}[1]{\widetilde{#1}}
\newcommand*{\wh}[1]{\widehat{#1}}

\label{period ring: AB ring}

\newcommand{\B}{  {\mathbf{B}}  }

\newcommand{\wtb}{   {\widetilde{{\mathbf{B}}}}  }

\label{period ring: integral Hodge}

\def \ok {{\mathcal{O}_K}}

\def \oc {{\mathcal{O}_C}}

\newcommand{\ocflat}{{\mathcal{O}_C^\flat}}

\label{peirod module: D}

\label{Cats: Fontaine}

\label{Cats: Rep and Higgs}

\label{*}
\label{ALL: site, sheaves}
  
\newcommand{\rgamma}{\mathrm{R}\Gamma}
 
\newcommand{\rg}{\mathrm{R}\Gamma}

 \newcommand{\Shv}{\mathrm{Shv}}

\label{site: etale proetale}

\label{site: prismatic}

\usepackage{relsize}
\usepackage[bbgreekl]{mathbbol}
\usepackage{amsfonts}

\DeclareSymbolFontAlphabet{\mathbb}{AMSb}
\DeclareSymbolFontAlphabet{\mathbbl}{bbold}
\newcommand{\prism}{{\mathlarger{\mathbbl{\Delta}}}}
 
\newcommand{\xpris}{{X_\prism}}

\newcommand{\baropris}{{\overline{\O}_\prism}}

\label{syntomic coho}

\label{font: roman rm}

\label{font: mathbb}

\newcommand{\barK}{{\overline{K}}}

\newcommand{\zp}{{\mathbb{Z}_p}}
\newcommand{\qp}{{\mathbb{Q}_p}}

\label{font: mathcal}

\label{font: mathfrak}

\newcommand{\gs}{{\mathfrak{S}}}

\newcommand{\fkt}{{\mathfrak{t}}}

 \label{font: hat}

\label{font: bar}

\label{font: bold}

\label{font: tilde}

 \label{font: mathbf}

\newcommand{\cbf}{\mathbf{c}}

\newcommand{\kbf}{\mathbf{k}}

\label{;'}

\label{stack proj}

\label{words}

\label{Gao added to dR paper}

\label{address}
\author[]{Hui Gao}   \address{Department of Mathematics, Southern University of Science and Technology, Shenzhen 518055, China}   \email{gaoh@sustech.edu.cn}

\author[]{Yu Min}
\address{Department of Mathematics, Imperial College London, London SW7 2RH}
\email{y.min@imperial.ac.uk}

\author[]{Yupeng Wang}
\address{Beijing International Center of Mathematics research, Peking University, Yiheyuan 5, Beijing, 100190, China.}
\email{2306393435@pku.edu.cn}

\begin{document}
\title[]{Hodge--Tate prismatic crystals  and Sen theory} 
 \subjclass[2010]{Primary  14F30,11S25}
 
\begin{abstract}  
\normalsize{
We study Hodge--Tate crystals on the absolute  (log-) prismatic site of $\ok$, where $\ok$ is a mixed characteristic complete discrete valuation ring with perfect residue field. 
We  first classify  Hodge--Tate crystals   by  $\ok$-modules equipped with certain small  endomorphisms. 
We then construct   Sen theory over a non-Galois Kummer tower, and use it to classify rational Hodge--Tate crystals by (log-) nearly Hodge--Tate representations. 
Various cohomology comparison and vanishing results are proved along the way.
} 
\end{abstract}

\date{\today}
\maketitle
\setcounter{tocdepth}{1}
\tableofcontents

\section{Introduction }\label{Sec-Introduction}
  
\subsection{Overview}

In the  ground-breaking work \cite{BS22}, Bhatt and Scholze construct  the prismatic cohomology, which is of ``motivic nature" in the sense that it can specialize to most existing $p$-adic cohomology theories such as \'etale, de Rham and crystalline cohomologies.
 Similar to the theory of crystalline cohomology, one can also consider crystals on the relative/absolute prismatic site. It turns out that these coefficient objects contain important geometric and arithmetic information. 
 Let $\ok$ be a mixed characteristic complete discrete valuation ring with perfect residue field, and let $K=\ok[1/p]$.
In \cite{BS23}, Bhatt and Scholze prove an equivalence between the category   of   $F$-crystals over the absolute prismatic site of $\ok$ and the category of crystalline $\bZ_p$-representations of  $G_K=\gal(\barK/K)$.

Since the work of Bhatt--Scholze \cite{BS23}, the study of various types of prismatic crystals   has been very fast developing. 
In this paper, we focus on the study of \emph{(absolute) Hodge--Tate crystals}, which are   vector bundles over the Hodge--Tate  structure  sheaf on the absolute (log-) prismatic site. 
It turns out these   crystals are closely related with classical Sen theory \cite{Sen80} developed some 40 years ago. 
In particular, the study of these crystals leads to   definition of a new category of objects called \emph{(log-) nearly Hodge--Tate representations}. These representations can be regarded as the ``(log-) crystalline objects" in Sen theory, and give  interesting updates to Sen's original work.
Here, we should mention right away that there have been   independent works by Drinfeld \cite{Dri20}, Bhatt--Lurie \cite{BL-a, BL-b} and Ansch\"utz--Heuer--Le Bras \cite{AHLB1, AHLB2} on   Hodge--Tate crystals, using a stacky approach. We choose to give a more detailed account/comparison of the literature only near the end of the introduction, cf. \S \ref{subsecliterature}, particularly after we introduce the necessary notations. 
In addition, we also mention some of our computations are inspired by Tian \cite{Tia23} on \emph{relative} Hodge--Tate crystals.

  We now quickly set up some notations and terminologies for our discussions.
Let  $(\calO_K)_{\Prism}$  be the absolute prismatic site (cf. \cite{BS22}), and let $(\calO_K)_{\Prism,\log}$ be the absolute log-prismatic site constructed by Koshikawa \cite{Kos21}.
Let  $(\calO_K)_{\Prism}^{\perf}$ (resp. $(\calO_K)_{\Prism,\log}^{\perf}$) be the subsite consisting of perfect prisms (resp. perfect log prisms).  
We use $\opris$ resp. $\mathcal{I}_\prism$ to denote the structure sheaf resp. the structure ideal sheaf on these sites. For example, on  $\okprism$,  we have
\[ \mathcal{O}_\Prism((A, I))=A, \quad \text{resp. }\mathcal{I}_\prism((A, I)) =  I.\]
From these basic sheaves, one can form other variant sheaves. Let $\baroprism: =\opris/\mathcal{I}_\prism  $ be the Hodge--Tate structure sheaf, and let $\baroprism[1/p]$ be the rational  Hodge--Tate  structure sheaf with $p$ inverted.
We    set up  notations of various crystals.

\begin{dfn}\label{Dfn-A crystal}
   Let $\calS$ be one of the sites $\{(\calO_K)_{\Prism},(\calO_K)_{\Prism,\log},(\calO_K)_{\Prism}^{\perf},(\calO_K)_{\Prism,\log}^{\perf}\}$.
    Let $\bA$ be one of the sheaves $\{\calO_{\Prism},\overline\calO_{\Prism},\overline\calO_{\Prism}[1/p]\}$.  
   Define the category of $\bA$-crystals on $\calS$ by 
   \[ \Vect(\calS,\bA):= \projlim_{\calA\in \calS} \vect(\bA(\calA))\]
    In concrete terms, a such crystal $\bM$ (of rank $l$)   is a sheaf of $\bA$-modules such that 
   \begin{itemize}
   \item for any object $\calA\in\calS$, $\bM(\calA)$ is a finite projective $\bA(\calA)$-module of rank $l$, and that

   \item for any morphism $\calA\to\calA'$ in $\calS$, the natural map
   \[\bM(\calA) \otimes_{\bA(\calA)}\bA(\calA')\to \bM(\calA')\]
   is an isomorphism.
   \end{itemize}
 \end{dfn}

The following commutative diagram summarizes the   main theorem of Bhatt--Scholze \cite{BS23}, which serves as a proto-type for the results in this paper.  (Here and throughout the paper, we use $\into$ resp. $\simeq$ in commutative diagrams to signify a functor is fully faithful  resp. an equivalence.) 
\begin{equation} \label{eqBS23}
 \begin{tikzcd}
{\mathrm{Vect}^{\varphi}((\calO_K)_{\Prism}, \mathcal{O}_\Prism)} \arrow[d, "\simeq"] \arrow[rr, hook] &  & {\mathrm{Vect}^{\varphi}(\okprism, \mathcal{O}_\Prism[1/\mathcal{I}_\Prism]^{\wedge}_p)} \arrow[d, "\simeq"] \\
\rep^{\cris}_{\gk}(\zp) \arrow[rr, hook]                                                               &  & \rep_\gk(\zp)                                     
\end{tikzcd}  \end{equation}
Here,   an object in $\mathrm{Vect}^{\varphi}((\calO_K)_{\Prism}, \mathcal{O}_\Prism)$ is called an $F$-crystals: it is a crystal in $\vect(\okpris, \opris)$ (per Def. \ref{Dfn-A crystal})  equipped with a ``$\varphi$-isogeny". Objects in $\mathrm{Vect}^{\varphi}(\okprism, \mathcal{O}_\Prism[1/\mathcal{I}_\Prism]^{\wedge}_p)$ are similarly defined, and are called Laurent $F$-crystals. The category $\rep_\gk(\zp)$ consists of all   $\zp$-linear representations of  $\gk$, and $\rep^{\cris}_{\gk}(\zp) $ is the subcategory consisting of integral crystalline representations.
Let us mention that the  right   vertical equivalence in \eqref{eqBS23} is also independently obtained by \cite{Wu21}. In addition, the log-prismatic version of \eqref{eqBS23}  is proved by \cite{DL23}.

\subsection{Hodge--Tate crystals and (log-) nearly Hodge--Tate representations} \label{subsecHTintro}

Let $\ast \in \{\emptyset, \log  \}$.
We call an object in $$\mathrm{Vect}(\okprisast, \baropris) \text{ resp. } \mathrm{Vect}(\okprisast, \baropris[1/p])$$
an (integral) Hodge--Tate crystal resp.  a rational Hodge--Tate crystal. 
In this subsection, we state our main theorem  concerning the relation between rational Hodge--Tate crystals and $C$-representations. 


We first define our representation theoretic objects in our main theorem.
Let $\rep_\gk(C)$ be the usual category of $C$-representations; an object is a finite dimensional $C$-vector space equipped with a \emph{semi-linear} $\gk$-action. For $W \in \rep_\gk(C)$, Sen  \cite{Sen80} canonically associates a \emph{linear operator} $\phi: W \to W$ which nowadays is called the Sen operator; its eigenvalues are in $\barK$ and are called the (Hodge--Tate--)Sen weights of $W$. 
(In our set-up, the Hodge--Tate--Sen weight  of the the cyclotomic character is $1$.)

\begin{notation}\label{notaeu}
Let $k$ be the residue field of $\ok$.
Let $W(k)$ be the ring of Witt vectors, and let $K_0=W(k)[1/p]$.
Let $\pi \in K$ be a \emph{fixed} uniformizer, and let $E(u)=\mathrm{Irr}(\pi, K_0) \in W(k)[u]$ be the minimal polynomial over $K_0$. Let $E'(u) =\frac{d}{du}E(u)$, and let $E'(\pi)$ be the evaluation at $\pi$. Recall by \cite[Chap. III, \S 6, Cor. 2]{Serrelocal}, the ideal $E'(\pi)\cdot\ok$ is precisely the different $\mathfrak{D}_{\ok/W(k)}$ and hence is independent of choices of $\pi$.
\end{notation}

\begin{defn}\label{defnnht}
Say $W\in \rep_\gk(C)$ is \emph{nearly Hodge--Tate} (resp. \emph{log-nearly Hodge--Tate})  if all of its  Sen weights are in the subset
\[\mathbb{Z} + (E'(\pi))^{-1}\cdot \mathfrak{m}_{\O_{\overline{K}}}, \quad \text{ resp. } \mathbb{Z} +  (\pi\cdot E'(\pi))^{-1}\cdot \mathfrak{m}_{\O_{\overline{K}}}.  \]
 where $\O_{\overline{K}}$ is the ring of integers of $\overline{K}$ with $\mathfrak{m}_{\O_{\overline{K}}}$  its maximal ideal.
 That is: the Sen weights are near to being an integer up to a bounded distance. (By discussions in Notation \ref{notaeu}, the definition is   independent of choices of $\pi$).
Write
\[\rep_\gk^{\mathrm{nHT}}(C), \quad \text{resp. }  \rep_C^{\mathrm{lnHT}}(\gk)=\rep_C^{\mathrm{log-nHT}}(\gk) \]
 for the (tensor) subcategory of $\rep_\gk(C)$ consisting of these objects.
\end{defn}

We recall the Fontaine prism which will be   used in the statement of our main theorem.

\begin{notation}\label{exampleAinfprism}
Fix an algebraic closure $\overline{K}$ and denote $\gk=\gal(\overline{K}/K)$. Let $C$ be the $p$-adic completion of $\overline K$.
Let $C^\flat$ be the tilt of the perfectoid field $C$, and let $\O_{C^\flat}$ be its ring of integers. Let $\ainf=W(\O_{C^\flat})$ be the ring of Witt vectors, equipped with the absolute Frobenius. There is a usual Fontaine's map $\theta: \ainf \to \O_C$.
Let $\mu_1$ be a primitive $p$-root of unity, and inductively, for each $n \geq 2$, choose $\mu_n$ a $p$-th root of $\mu_{n-1}$.  The sequence $(1, \mu_1, \cdots, \mu_n, \cdots)$ defines an element $\epsilon \in \ocflat$. Let $[\epsilon] \in \ainf$ be its Teichm\"uller lift. Define $\xi := \frac{[\epsilon]-1}{[\epsilon]^{{1}/{p}}-1}$, then it is a generator of the kernel of $\theta$.
 Then $(\ainf, (\xi))$ is an object in $(\calO_K)_{\Prism}^{\perf}$ and hence in  $\okprism$, and is called the Fontaine prism. One can further equip $\Ainf$ with a log structure $M_{\ainf}$, giving rise to the Fontaine log prism  $(\ainf, (\xi), M_{\ainf})$ on the log-prismatic site, cf. Notation \ref{notaBKlogprism}.
\end{notation}
The following is the main theorem of this paper, which relates Hodge--Tate crystals with its ``(pro-)\'etale realizations".

\begin{theorem} \label{thmintroHT} 
\begin{enumerate}
\item \emph{(Theorems \ref{thmnht}  and \ref{rational crystal as representation})}. We have a commutative diagram of tensor functors:
\begin{footnotesize}
\begin{equation*}
\begin{tikzcd}
{\Vect( \okpris,\overline \calO_{\Prism}[\frac{1}{p}])} \arrow[d, "\simeq"] \arrow[r, hook] & {\Vect( \okprislog,\overline \calO_{\Prism}[\frac{1}{p}])} \arrow[d, "\simeq"] \arrow[r, hook] & {\Vect((\calO_K)^{\perf}_{\Prism, \log},\overline \calO_{\Prism}[\frac{1}{p}])} \arrow[d, "\simeq"] & {\Vect((\calO_K)^{\perf}_{\Prism},\overline \calO_{\Prism}[\frac{1}{p}])} \arrow[d, "\simeq"] \arrow[l, "\simeq"'] \\
\rep_\gk^{\mathrm{nHT}}(C) \arrow[r, hook]                                                  & \rep_\gk^{\mathrm{lnHT}}(C) \arrow[r, hook]                                                    & \rep_\gk(C)                                                                                         & \rep_\gk(C) \arrow[l, "="']
\end{tikzcd}
\end{equation*}
\end{footnotesize}
Here,   the first and fourth (resp. second and third) vertical equivalences are induced by evaluating the crystals on the Fontaine prism (resp. Fontaine log prism).  All the vertical equivalences are bi-exact; namely, the equivalences and  their quasi-inverses are exact.

\item  \emph{(Thm. \ref{thm-coho-pris-proetale})}. Let $\ast \in \{\emptyset, \log\}$. Let $\bm \in \Vect( \okprisast,\overline \calO_{\Prism}[1/p] )$ be a Hodge--Tate crystal, and let  $W \in \rep_\gk^{\ast-\nht}(C)$ be the associated representation. There exists a natural $K$-linear quasi-isomorphism 
\[   \rg( \okprisast, \bm) \simeq \rg(\gk, W) \]
where the right hand side is   Galois cohomology.
\end{enumerate} 
\end{theorem}

We shall explain the proof of this theorem in two steps. In \S \ref{subsecstrat}, we explain how to use the Breuil--Kisin prism and stratifications to classify Hodge--Tate crystals by certain ``coherent data".
In \S \ref{subsecKummerSen}, we develop a new Sen theory to read off Sen weights from these coherent data, which leads to proof of Thm. \ref{thmintroHT}.

 \begin{rem} \label{remintro115}
 We make some   general remarks about the main theorem.
 \begin{enumerate}
 \item 
Recall $W\in \rep_\gk(C)$ is called \emph{almost Hodge--Tate} (cf. \cite{Fon04}) if its Sen weights are integers, and is  called \emph{Hodge--Tate} if  additionally the Sen operator is    semi-simple. 
Thus, the (log-) nearly Hodge--Tate representations are those \emph{near} to an almost Hodge--Tate representation. 
The similarities between the diagrams in Thm. \ref{thmintroHT}  and in \eqref{eqBS23} seem to suggest that the  (log-) nearly Hodge--Tate representations could be regarded as the ``(log-) crystalline objects"  in the category of $C$-representations; this analogy deserves further studies.

\item  Recall the notion of \emph{almost  de Rham representations} is introduced in \cite{Fon04}. Given a $\bdrplus$-representation of $\gk$, it is almost de Rham if its mod $t$ reduction (which is a $C$-representation) is almost Hodge--Tate. Inspired by Fontaine's result, we could introduce the notion of \emph{(log-) nearly de Rham representations}, which are precisely those $\bdrplus$-representations whose mod $t$ reductions are (log-) nearly Hodge--Tate $C$-representations.
Indeed, if one changes the rational reduced structure sheaf $\overline \calO_{\Prism}[{1}/{p}]$ in Thm. \ref{thmintroHT}  to a ``prismatic de Rham period sheaf", i.e., 
\[\bbdrplus:= \projlim_{m \geq 1}  (\mathcal{O}_\Prism[1/p])/\mathcal{I}^m_\prism,\]
then one can define the notion of \emph{$\bbdrplus$-crystals} and prove they are   classified by the (log-) nearly de Rham representations. This will be discussed in our sequel \cite{GMWdR}.
 \end{enumerate}
 \end{rem}

\subsection{Breuil--Kisin prisms and  stratifications}
\label{subsecstrat}
In the statement of the main Theorem \ref{thmintroHT}, the vertical functors are all induced by Fontaine (log-) prisms. But to prove the theorem, one needs to carry out \emph{explicit} computations, where one has to make use of the ``smaller" Breuil--Kisin (log-) prisms.

\begin{notation}\label{notaBKprism}
\begin{enumerate}
\item 
Recall in Notation \ref{notaeu}, we defined $\pi$ and $E(u)$.
Let $\gs=W(k)[[u]]$, and equip it with a Frobenius $\varphi$ extending  the absolute Frobenius on $W(k)$ and such that $\varphi(u)=u^p$. Then $(\gs, (E(u)) \in \okprism$, and is called the Breuil-Kisin prism (associated to $\pi$).
One can equip it with a log structure and obtain a log prism $(\gs, (E), M_\gs) \in \okprislog$, cf. Notation \ref{notaBKlogprism}.

\item Let $\pi_0=\pi$, and for each $n \geq 1$, inductive fix some $\pi_n$ so that $\pi_n^p=\pi_{n-1}$.
This compatible sequence defines an element $\pi^\flat \in \ocflat$. We can define a morphism of prisms (resp. log prisms)
\begin{equation}\label{eqmorphbkainf}
(\gs, (E)) \to (\ainf, (\xi)), \quad \text{ resp. } (\gs, (E), M_\gs) \to (\ainf, (\xi), M_\ainf)
\end{equation} 
which is a $W(k)$-linear map and sends $u$   to the Teichm\"uller lift $[\pi^\flat]$.
\end{enumerate}
\end{notation}

As can be seen in the statement of our main Theorem \ref{thmintroHT}, the results in prismatic case resp. log-prismatic case are similar in shape. We introduce the following convention in order to discuss the two cases in parallel.

\begin{convention} \label{convnotalogprism}
 We   use
\[ (\gs, (E), \log), \text{ resp. } (\ainf, (\xi), \log) \]
to denote the log prism
\[ (\frakS,(E),M_{\frakS}), \text{ resp. } (\ainf, (\xi), M_\ainf). \]
Thus, if we let $\ast \in \{\emptyset, \log  \}$, then we  could use
\[ (\gs, (E), \ast), \text{ resp. } (\ainf, (\xi), \ast) \]
to denote objects in both prismatic and log-prismatic sites.
\end{convention}

Let   $\ast \in \{\emptyset, \log  \}$.
It turns out that the Breuil--Kisin prism $(\frakS,(E), \ast)$ is a cover of the final object of ${\rm Shv}((\calO_K)_{\Prism, \ast})$. Thus we can interpret Hodge--Tate crystals as stratifications with respect to the \v Cech nerve
$(\frakS^{\bullet},(E), \ast)$ associated with $(\frakS,(E), \ast)$. It turns out that the cosimplicial ring $\frakS_\ast^{\bullet}$ is extremely complicated; fortunately, the structure of $\frakS_\ast^{\bullet}/E$ turns out to be relatively easy. Indeed, in Prop. \ref{pd polynomial}, we will show
\begin{equation*}
\gs^n_\ast/E \simeq \ok\{X_1, \cdots, X_n \}^{\wedge}_{\pd}
\end{equation*}
 where the right hand side is the $p$-completed pd-polynomial ring with certain explicit variables   $X_i$.
 
 With the \emph{explict} cosimplicial rings $\frakS_\ast^{\bullet}/E$ in mind, we can \emph{explicitly} compute the \emph{stratifications} (cf. Def. \ref{stratification-dfn} for a recall of definition) associated to  Hodge--Tate crystals. Indeed, we can classify (integral and rational) Hodge--Tate crystals by the following simple linear algebra data.

\begin{defn} \label{defnnhtendo}
 Let $*\in\{\emptyset,\log\}$. 
Let \begin{equation*}
a=
\begin{cases}
  -E'(\pi), &  \text{if } \ast=\emptyset \\
 -\pi E'(\pi), &  \text{if } \ast=\log
\end{cases}
\end{equation*}  
Let  $\End_{\calO_K}^{\ast-\nht}$ (resp. $\End_{K}^{\ast-\nht}$)  be the category consisting of pairs $(M,\phi_M)$ which we call a \emph{module equipped with $a$-small endomorphism}, where
    \begin{enumerate}
        \item[(1)] $M$ is a finite free $\calO_K$-module (resp. $K$-vector space), and
        \item[(2)] $\phi_M$ is an $\calO_K$-linear (resp. $K$-linear) endomorphism of $M$ such that 
        \begin{equation}
            \lim_{n\to+\infty}\prod_{i=0}^{n-1}(\phi_M-ai) = 0.
        \end{equation} 
            \end{enumerate}
\end{defn}

 \begin{theorem} \label{thmintrostrat}
 \begin{enumerate}
 \item \emph{(Thm. \ref{Thm-HTCrystal}).}  Evaluation on the Breuil--Kisin (log-) prisms induce bi-exact equivalences of categories  
 \[ \Vect((\calO_K)_{\Prism, \ast},\overline \calO_{\Prism}) \simeq  \End_{\calO_K}^{\ast-\nht}                                                    \]
  \[ \Vect((\calO_K)_{\Prism, \ast},\overline \calO_{\Prism}[1/p]) \simeq  \End_{K}^{\ast-\nht}                                                    \]
 
  \item  \emph{(Thm. \ref{HTcohom})}. Let $\bM$ be an object in $\Vect((\calO_K)_{\Prism,*},\overline \calO_{\Prism})$ resp.  $\Vect((\calO_K)_{\Prism,*},\overline \calO_{\Prism}[1/p])$. Let $(M,\phi_M)$ be the corresponding object in $\End_{\calO_K}^{\ast-\nht}$ resp.  $\End_{K}^{\ast-\nht}$.
There is a natural $\ok$-linear (resp. $K$-linear) quasi-isomorphism
\[[M\xrightarrow{\phi_M}M]\simeq \rR\Gamma((\calO_K)_{\Prism,*},\bM).\]
 \end{enumerate}
 \end{theorem}
The proof of the above theorem is parallel for both the prismatic and log-prismatic cases. Thus, we will develop in Section \ref{Strat} an axiomatic approach which covers both cases. 
Via a devissage argument, we also obtain the following cohomology vanishing theorem for crystals.

\begin{thm}[Theorem \ref{non-reduced coh}]\label{Intro-coho dim}
 Let $\bM\in \Vect(\okprisast,\calO_{\Prism})$ (a such object is called  a  crystal). Then   \[ \rH^i((\calO_K)_{\Prism},\bM)=0, \forall i \geq 2\]
\end{thm}

 \subsection{Sen theory over  Kummer tower} \label{subsecKummerSen}
In this subsection, we explain a   Sen theory over a Kummer tower, and show how to use it to conclude the proof of Thm. \ref{thmintroHT}.

\begin{construction} \label{constfinal}
Let us explain the basic ideas to prove  Thm. \ref{thmintroHT}, and introduce some notations for later discussions.
Let $ \bm \in \vect(\okprisast, \baropris[1/p])$  be a rational Hodge--Tate crystal, and let $ W=\bm((\ainf, (\xi), \ast))$  be the associated $C$-representation. 
  The most difficult step for Thm. \ref{thmintroHT} is to show that $W$ is (log-) nearly Hodge--Tate; that is, we need to show the Sen weights of $W$ satisfy the condition in Def. \ref{defnnht}.
  
  We hint that this condition on Sen weights is already reflected in the second condition in Def. \ref{defnnhtendo}.
Indeed, let 
$ M=\bm((\gs, (E), \ast)) $ 
and let $(M, \phi_M)$ be the corresponding endomorphism associated to $\bm$. By linear algebra, the eigenvalues  of $\phi_M$ must be contained in the subset
\[ a\cdot \mathbb{Z} + \mathfrak{m}_{\barK}. \]
That is to say, the eigenvalues of $\frac{\phi_M}{a}$ precisely satisfy the (log-) nearly Hodge--Tate condition in Def. \ref{defnnht}! Thus for our purpose, it would suffice to \emph{relate $\frac{\phi_M}{a}$ with Sen theory}.
\end{construction}

 In this subsection, we explain how to achieve the final goal in Construction \ref{constfinal} using Sen theory over a \emph{Kummer tower}. 
Let us mention that this new  Sen theory is indeed a ``shadow" of the \emph{overconvergent $(\varphi, \tau)$-modules} constructed in \cite{GL20, GP21}; to save space here, we refer to \S \ref{subsecreviewOC} for more comments on the origin of this Sen theory.
  
  Let us firstly briefly recall classical Sen theory defined over the \emph{cyclotomic tower}, and see why it is insufficient for our purpose. For   convenience of ensuing discussions, we shall adopt a slightly modern approach using the notion of \emph{locally analytic vectors}; their deep relation with Sen theory (in a broader sense) is first discovered by Berger--Colmez \cite{BC16}.  Indeed, for $W\in \rep_\gk(C)$, the classical Sen module, mentioned above Def. \ref{defnnht},  can be obtained via
\begin{equation} \label{eqsenkpinftyintro}
D_{\Sen, \kpinfty}(W)=(W^{G_\kpinfty})^{\Gamma_K\dla}.
\end{equation}
Here,  $\kpinfty$ it the cyclotomic extension of $K$ by adjoining all $p$-power roots of unity, $G_\kpinfty=\gal(\barK/\kpinfty)$, and $\Gamma_K= \gal(\kpinfty/K)$; the notation ``$\Gamma_K\dla$" denotes the subset of locally analytic vectors under the action of  the $p$-adic Lie group $\gammak$.
 Using this framework, the Sen operator is precisely the \emph{Lie algebra operator}. See \S \ref{subseccycSen} for more details.
  Note the Eqn \eqref{eqsenkpinftyintro} informs us that Sen operators are related with (infinitesimal information of) Galois actions. 

\begin{notation} \label{notafields}
We introduce several other fields and Galois groups for further discussions.
 We already introduced $\mu_n$ resp. $\pi_n$ in Notations \ref{exampleAinfprism} resp. \ref{notaBKprism}, which are $p$-power roots of unity resp. of $\pi$.  
 Define the fields
$$K_{\infty}   = \cup _{n = 1} ^{\infty} K(\pi_n), \quad K_{p^\infty}=  \cup _{n=1}^\infty
K(\mu_n), \quad L =  \cup_{n = 1} ^{\infty} K(\pi_n, \mu_n).$$
Let $$G_{\kinfty}:= \gal (\overline K / K_{\infty}), \quad G_{\kpinfty}:= \gal (\overline K / K_{p^\infty}), \quad G_L: =\gal(\overline K/L).$$
Further define $\Gamma_K, \hat{G}$ as in the following diagram, where we let $\tau$ be a topological generator of $\gal(L/\kpinfty) \simeq \zp$, cf. Notation \ref{nota hatG} for more details.
\[
\begin{tikzcd}
                                       & L                                                                                              &                             \\
\kpinfty \arrow[ru, "<\tau>", no head] &                                                                                               & \kinfty \arrow[lu, no head] \\
                                       & K \arrow[lu, "\Gamma_K", no head] \arrow[ru, no head] \arrow[uu, "\hat{G}"', no head, dashed] &
\end{tikzcd}
\]
\end{notation}

\begin{construction}\label{contrcontinue}
In continuation using notations in Construction \ref{constfinal}, we hope to connect $\frac{\phi_M}{a}$ with Galois actions. Let $\underline{e}$ be a basis of $M$, which also serves as a basis of $W$ via the   identification 
\[ M \otimes_{K} C \simeq W \]
induced by the morphism of prisms
\[ (\gs, (E), \ast)\to (\ainf, (\xi), \ast)\]
Note via the embedding $\gs \into \ainf$, $\gs$ if fixed by $G_\kinfty$ but is \emph{not} fixed by $G_\kpinfty$. That is to say, elements of $M$ (in general) can not land inside the classical Sen module $D_{\sen, \kpinfty}(W)$. 
Nonetheless, it turns out we can compute \emph{explicitly} (using results from the proof of Thm. \ref{thmintrostrat}) the $\tau$-action by the following formula:
$$\tau^i (\underline e) = (\underline e) (1- i  \cdot c)^{\frac{A_1}{a}}, \quad \forall i \in \zp,$$
where $c$ is some constant that we do not specify here, and $A_1$ is the matrix of $\phi_M$ with respect to the basis $\underline{e}$; cf. \S \ref{subsec_fullcrystal_rep}.
Thus, one observes that if we ``take the log of $\tau$", we obtain (up to scaling) our target matrix $\frac{A_1}{a}$! In other words, to prove Thm. \ref{thmintroHT}, it is equivalent to show that: \emph{``taking the log of $\tau$" gives rise to Sen operator}; this is achieved in Thm. \ref{thmintrosen}.
\end{construction}

The ideas in Construction \ref{contrcontinue} prompts us to define a Sen module over the Kummer tower. The following definition is inspired by similar formula in study of overconvergent $(\varphi, \tau)$-modules, cf. \S \ref{subsecreviewOC}.

\begin{defn}\label{defnintrokummersen}
For $W\in \rep_\gk(C)$, define
 \begin{equation}
 D_{\Sen, \kinfty}(W):= (W^{G_L})^{\gamma=1,\tau\dla}.
 \end{equation}
 Here $\gamma=1$ denotes the invariant under $\gal(L/\kinfty)$-action, and  $\tau\dla$ denotes the locally analytic vectors under the  $\gal(L/\kpinfty)$-action, cf. Notation \ref{notataula}.
\end{defn}

\begin{theorem} \label{thmintrosen} (cf. \S\ref{subsecKS}.)
Let $W\in \rep_\gk(C)$, then $D_{\Sen, \kinfty}(W)$ is a finite dimensional $\kinfty$-vector space   such that the natural map
\[ D_{\Sen, \kinfty}(W)\otimes_\kinfty C \to W\]
is an isomorphism.
In addition,  we can construct a $\kinfty$-linear operator
\begin{equation*}
 \frac{\log \tau}{\mathfrak{c}} : D_{\Sen, \kinfty}(W) \to D_{\Sen, \kinfty}(W),
\end{equation*}
where $\log \tau$ denotes the Lie algebra operator associated to $\tau$, and  $\mathfrak{c} \in (\hat{L})^{\hat G \dla}$ is some ``normalizing constant" that we do not specify here. We call this operator the  \emph{Sen operator over the Kummer tower}.
This operator, after $C$-linearly extending to \[ \frac{\log \tau}{\mathfrak{c}}: W \to W,\]
 is the \emph{same} as the $C$-linear extension of classical Sen operator. In particular, eigenvalues of $\frac{\log \tau}{\mathfrak{c}}$ are precisely the Sen weights of $W$.
\end{theorem}
 
 As argued in Construction \ref{contrcontinue}, Thm. \ref{thmintrosen} allows us to read off   Sen weights of $W$ from the data $(M, \phi_M)$, hence concluding the proof of Thm. \ref{thmintroHT}.

\subsection{Discussion of the literature}\label{subsecliterature}
We give several extensive remarks about the literature on Hodge--Tate crystals, and make some comparisons with our work.

\begin{remark}\label{remstacky}
In this remark, we discuss another independent approach, the \emph{stacky approach}, to study the Hodge--Tate crystals. 
\begin{enumerate}
\item In \cite{Dri20} and \cite{BL-a, BL-b}, Drinfeld and  Bhatt--Lurie define a stack  $X^\prism$ for a $p$-adic formal scheme $X$; vector bundles over this stack classify prismatic crystals on $X_\prism$.
A closed substack, the \emph{Hodge--Tate locus}  $X^{\HT}$ is also constructed; vector bundles over $X^\HT$ classify $\vect(\xpris, \baropris)$, whence the   name \emph{Hodge--Tate crystals} for the later. 

\item When $X=\spf \ok$, it turns out $\spf(\ok)^{\HT}$ admits a cover by $\spf \ok$, which is indeed a $G$-torsor for some explicit group $G$.  
Using this torsor structure, Bhatt--Lurie obtains our Theorem \ref{thmintrostrat} in the prismatic setting independently (cf. \cite[Theorem 3.5.8]{BL-a},\cite[Example 9.6]{BL-b}). Using the stacky approach, their results also  hold in a derived fashion; in particular, the cohomology comparison in Thm. \ref{thmintrostrat}(2) is automatic in their work. We also remark that this stacky approach (particularly in the ramified case) is studied in more detail in \cite{AHLB1}. 



\item The paper   \cite{AHLB1} further constructs a certain ``Sen theory" using a period ring $B_{\mathrm{en}}$, and deduces our Theorem \ref{thmintroHT} in the prismatic setting. We also mention when $K$ is unramified (hence there is a $q$-de Rham prism), Hodge--Tate crystals can also be easily related with classical Sen theory (over the cyclotomic tower); thus, our  Theorem \ref{thmintroHT} in the unramified case can essentially be deduced by  \cite{BL-a}.
\end{enumerate}
\end{remark}
 
\begin{remark}\label{remintrocompare}
We compare our \emph{site-theoretic} approach to the \emph{stacky} approach in  Rem. \ref{remstacky}.
\begin{enumerate}
\item Firstly, our paper can also treat the log-prismatic case in a parallel fashion. The log-prismatic cases are not yet covered in \cite{Dri20, BL-a, BL-b, AHLB1}. With the recent development of  log-prismatic cohomology, cf. \cite{Kos21, KY23}, the stacky approach in the log setting should also be within reach.

\item As mentioned in Rem. \ref{remstacky}, in the stacky approach, one can even obtain many results in a derived setting. 
It seems possible that we can also ``derive" most of our constructions in the current paper, cf. e.g. the axiomatic computations in \S \ref{subsecaxiomstrat} and particularly Thm. \ref{Thm-AxiomHT}, where we can also treat  general $p$-complete objects (that are not necessarily projective).
We have chosen not to carry out all the derived computations as our main focus is ``arithmetic", and we prefer to carry out  computations in a very \emph{explicit} fashion.

\item Indeed, our explicit computations make it possible for us to study $\bbdrplus$-crystals in the sequel \cite{GMWdR}; in fact, a key difficulty there is to construct certain \emph{explicit} ``variables". In comparison, as far as we are aware, it is not clear how to construct a ``$\bdrplus$-stack" (perhaps certain ``infinitesimal thickening" of the ``generic fiber" of the Hodge--Tate stack) to classify $\bbdrplus$-crystals. This stacky question seems very intriguing.

\item Finally, we point out both the stacky approach and the site-theoretic approach   work in the relative setting as well, ---i.e., for a formal scheme over $\spf \ok$---, cf. \cite{BL-a, BL-b, AHLB2} for the former, and cf. \cite{MW22, MW22log} for the later. We remark these results are related with non-abelian $p$-adic Hodge theory. Since the current paper only considers the point case, we refer the readers to these papers for more comparisons.
\end{enumerate}
\end{remark}

 \begin{rem}
In continuation of Rem. \ref{remintrocompare}, we comment on another distinctive feature in our  approach.
\begin{enumerate}
\item  As explained in Constructions \ref{constfinal} and \ref{contrcontinue}, the \emph{explicit} computations lead us to the construction of Sen theory over the Kummer tower Thm. \ref{thmintrosen}. This new Sen theory works for \emph{all} $C$-representations, and in particular has close relation to the general locally analytic Sen theory initiated by \cite{BC16}.  In contrast, the ``Sen theory" of \cite{AHLB1}, using the ring $B_{\mathrm{en}}$, can only be used to treat nearly Hodge--Tate representations. 

\item For  readers interested in our sequel paper \cite{GMWdR}, we point out that the Sen theory Thm. \ref{thmintrosen} plays a significant role in guiding our study of $\bbdrplus$-crystals. Indeed in \cite{GMWdR}, the cosimplicial rings and stratifications related with $\bbdrplus$-crystals are \emph{substantially} more involved, and initially it was not clear what our goal should be. However, inspired by \cite{Fon04}, Thm. \ref{thmintrosen} has a rather straightforward generalization to study all $\bdrplus$-representations, in this case leading to certain ``$K_\infty[[\lambda]]$-connections". The structure of this Sen-Fontaine theory (together with other explicit computations in the current paper) leads us to guess and conduct the \emph{correct} computations in \cite{GMWdR}.

\item As mentioned earlier, the Sen theory Thm. \ref{thmintrosen} (as well as the variant in \cite{GMWdR}) are ``shadows" of overconvergent $(\varphi, \tau)$-modules constructed in \cite{GL20,GP21}. However, there are also other  ``shadows", including the overconvergent modules themselves, as well as their application to Breuil--Kisin modules in \cite{Gao23}. We expect to see these shadows in other variants of prismatic crystals as well. We hope these variations on a  theme---namely, Sen theory--- help to clarify   the relation between prismatic crystals and (classical) $p$-adic Hodge theory.

\end{enumerate}
 \end{rem}

\subsection{Structure of the paper}
In \S \ref{secbkstrat}, we quickly review the notion of stratifications, which motivates us to study in detail the Breuil--Kisin cosimplicial ring and its reductions.
In \S \ref{Strat}, we compute, in an axiomatic way, stratifications (and their cohomology) over the Breuil--Kisin cosimplicial ring.
In \S \ref{sectionHT}, these axiomatic computations specialize  to classification theorem and cohomology comparison of  Hodge--Tate crystals.
In \S \ref{secHTperf}, which is the begining where we start to study Galois actions, we classify Hodge--Tate crystals on the perfect prismatic site by $C$-representations.
Some explicit formulae in \S \ref{secHTperf} prompts us to study Sen theory over the Kummer tower. We give a brief review of locally analytic vectors in \S \ref{seclav}, and then use the tool to define Sen theory over the Kummer tower in \S \ref{seckummersen}. In \S \ref{subsecHTsen}, we finally relate Hodge--Tate crystals with our new Sen theory, and  classify them by (log-) nearly Hodge--Tate representations.
The final \S \ref{higher vanishing} contains a technical proof of a cohomology vanishing theorem which is stated in \S \ref{Strat}.

 \subsection*{Acknowledgment}
  We thank Ruochuan Liu, Matthew Morrow and Takeshi Tsuji for useful correspondences during early stages of   this work.
We thank Bhargav Bhatt for informing us of his joint work with Jacob Lurie on absolute prismatic cohomology and suggesting us to use the name ``Hodge--Tate crystals" to replace our original one. 
We thank
Heng Du,
Juan Esteban Rodr\'{\i}guez Camargo, 
Tong Liu,
 and Zeyu Liu 
for useful discussions and correspondences. 
 Special thanks go to Shizhang Li for his careful reading of an early draft  and sharing many valuable comments. 
 Hui Gao  is partially supported by the National Natural Science Foundation of China under agreement No. NSFC-12071201; in addition, he gratefully acknowledges the support of the Infosys Member Fund at the Institute for Advanced Study for providing an excellent working condition where part of the work is carried out.
 Yu Min is partially supported by China Postdoctoral Science Foundation E1900503. 
 Yupeng Wang is partially supported by CAS Project for Young Scientists in Basic Research, Grant No. YSBR-032.
 All authors further thank Beijing International Center for Mathematical Research and Morningside Center of Mathematics where part of the work is carried out.

\section{Cosimplicial Breuil--Kisin prisms and stratifications} \label{secbkstrat}
 
In this section, we recall the notion of stratifications, and relate them to Hodge--Tate crystals. Then we give a detailed study of the cosimplicial Breuil-Kisin (log-) prism, particularly after reduction modulo the structural ideal. 
 
 \subsection{Breuil--Kisin prism  and Fontaine prism}
 \begin{notation}\label{notaBKlogprism}
 We introduce  two log prisms that will be repeatedly used, cf. \cite{Kos21} for foundations of log prismatic site and these examples.
 \begin{enumerate}
 \item Let $(\gs, (E))$ be the Breuil-Kisin prism in Notation \ref{notaBKprism} (with respect to the chosen uniformizer $\pi$).
  Let $M_{\frakS}\to \frakS$ be the log structure associated to the pre-log structure $\bN\xrightarrow{1\mapsto u}\frakS$. Then $(\frakS,(E),M_{\frakS})$  a log prism in  $(\calO_K)_{\Prism,\log}$, and is called the Breuil--Kisin log prism.
  
  \item Recall in Notation \ref{exampleAinfprism},    we defined the Fontaine prism $(\ainf, (\xi))$.
  We can equip $\Ainf$ with the log structure $M_{\ainf}$ induced by the pre-log structure $\bN\xrightarrow{1\mapsto[\pi^{\flat}]}\ainf$. Then $(\ainf,(\xi),M_{\ainf})$ is a log prism in $(\calO_K)_{\Prism,\log}$. We call it the Fontaine log prism. 
  
  \item 
 There is a morphism of prisms (resp. log prisms)
\begin{equation} 
(\gs, (E)) \to (\ainf, (\xi)), \quad \text{ resp. } (\gs, (E), M_\gs) \to (\ainf, (\xi), M_\ainf)
\end{equation} 
which is a $W(k)$-linear map and sends $u$   to the Teichm\"uller lift $[\pi^\flat]$.
 \end{enumerate}
 \end{notation}
 

 Recall in Convention \ref{convnotalogprism}, we introduced the notations $(\gs, (E), \ast) $ and $(\ainf, (\xi), \ast)$ for  $\ast \in \{ \emptyset, \log \}$.

 \begin{lemma}\label{Lem-BKCover}
 Let $\ast \in \{ \emptyset, \log \}$. Then the prisms   
\[ (\gs, (E), \ast) \text{ and } (\ainf, (\xi), \ast) \]
 are covers of the final object of the topos $\Shv(\okprisast)$.
 \end{lemma}
 
   \begin{proof}
 Since the morphism $\iota:(\frakS,(E), \ast)\to(\ainf,(\xi), \ast)$   is a cover on $\okprisast$, it suffices to consider $(\frakS,(E), \ast)$. This follows from \cite[Example 2.6]{BS23} and \cite[Proposition 5.0.16]{DL23} for $\ast=\emptyset$ and $\ast=\log$ respectively.
\end{proof}

\subsection{Crystals and stratifications}
We recall the definition of a stratification.
\begin{convention}
    Let $A^{\bullet}$ be a cosimplicial ring.
     For any $0\leq i\leq n+1$, let $p_i:A^n\to A^{n+1}$
     be the $i$-th face map   induced by the order-preserving map $[n]\to[n+1]$ whose image does not contain $i$. For any $0\leq i\leq n$, let $\sigma_i:A^{n+1}\to A^n$ be the $i$-th degeneracy map induced by the order-preserving map $[n+1]\to[n]$ such that the preimage of $i$ is $\{i,i+1\}$. For any $0\leq i\leq n$, let $q_i:A^0\to A^n$ be the structure map induced by the map $[0]\to [n]$ sending $0$ to $i$. 
\end{convention}
\begin{defn}\label{stratification-dfn}
For a cosimplicial ring $A^{\bullet}$, a \emph{stratification with respect to $A^{\bullet}$} is a pair $(M,\varepsilon)$ consisting of a finite projective $A^0$-module $M$ and an $A^1$-linear isomorphism
 \[\varepsilon: M\otimes_{A^0,p_0}A^1\to M\otimes_{A^0,p_1}A^1,\]
 such that the \emph{cocycle condition} is satisfied:
 \begin{enumerate}
 \item  $p_2^*(\varepsilon)\circ p_0^*(\varepsilon) = p_1^*(\varepsilon): M\otimes_{A^0,q_2}A^2\to M\otimes_{A^0,q_0}A^2$;
\item  $\sigma_0^*(\varepsilon) = \id_M$.
 \end{enumerate}
Write $\mathrm{Strat}(A^\bullet)$ for the category of stratifications with respect to $A^\bullet$.
\end{defn}

\begin{notation}
 Let $(\frakS^{\bullet},(E))$ (resp. $(\frakS^{\bullet}_{\log},(E),M^{\bullet})$) be the cosimplicial prism corresponding to the \v Cech nerve of the Breuil--Kisin prism (resp. the Breuil--Kisin log prism) in the prismatic site (resp. log-prismatic site). For their existence, see \cite[Example 2.6(1)]{BS23} (resp. \cite[Lemma 5.0.12]{DL23}).
\end{notation}

 Using the language of stratification, we get the following. Recall we use the phrase \emph{bi-exact equivalence} to mean an equivalence and its quasi-inverse are both exact.
 
 \begin{prop}\label{Prop-dRStratification}    
   Evaluation at the  Breuil-Kisin (log-) prism induces bi-exact  equivalences of categories:
  \[\vect(\okprisast, \baropris) \simeq \Strat(\gs^{\bullet}_\ast/E)\]
  \[\vect(\okprisast, \baropris[1/p]) \simeq \Strat(\gs^{\bullet}_\ast/E[1/p]).\]
 \end{prop}
  \begin{proof} These equivalences are standard consequences of Lemma \ref{Lem-BKCover}, using the fact that the sheaves $\baropris$ and $\baropris[1/p]$ satisfy $p$-complete faithfully flat descent \cite{Mat22}. We point out there is  another way to define $\vect(\okprisast, \baropris)$ and $\vect(\okprisast, \baropris[1/p])$ following \cite[Notation 2.1]{BS23} using ringed topoi; these definitions coincide with our Def. \ref{Dfn-A crystal} following a similar argument as \cite[Prop. 2.7]{BS23}.

  It remains to prove the above equivalences are bi-exact. We only treat the case for $\vect(\okpris, \baropris)$ since other cases are similar. For $i\in\{1,2,3\}$, let $\bM_i$ be a Hodge--Tate crystals with   induced stratification $(M_i,\varepsilon_i)$.
If there is a short exact sequence of Hodge--Tate crystals
\begin{equation}\label{biexact:I}      0\to\bM_1\xrightarrow{}\bM_2\xrightarrow{}\bM_3\to 0.
  \end{equation}
  Then by evaluations on $(\frakS^{\bullet},(E))$, we obtain a short exact sequence of stratifications
\begin{equation}\label{biexact:II}       0\to(M_1,\varepsilon_1)\to(M_2,\varepsilon_2)\to(M_3,\varepsilon_3)\to0
  \end{equation}
  Conversely, suppose we are given a short exact sequence \eqref{biexact:II}; to show the induced sequence \eqref{biexact:I}  is short exact, we need to show its evaluation at \emph{any} prism $(A, I)$ is short exact. By Lemma \ref{Lem-BKCover}, one can find a cover    $(A,I) \to (B,J)$ which admits a morphism $(\frakS,(E))\to(B,J)$. Let $(B^{\bullet},JB^{\bullet})$ be the \v Cech nerve associated to the covering $(A,I)\to (B,J)$. Then we have a canonical morphism $(\frakS^{\bullet},(E))\to (B^{\bullet},JB^{\bullet})$. So the exactness of (\ref{biexact:II}) implies that the evaluation of (\ref{biexact:I}) at $(B^{m},JB^{m})$ is exact for $m=0,1,2$. Hence we have the following exact sequence
  \[0\to\bM_1((A,I))\to\bM_2((A,I))\to\bM_3((A,I))\to\rH^1(((A,I)),\bM_1).\]
  Finally, we can conclude by using $\rH^1((A,I),\bM_1) = 0$, by fpqc descent (cf. Lemma \ref{vansihing lemma}).
  \end{proof}

\begin{remark} 
Many   equivalences of categories constructed in this paper are induced by \emph{evaluations} (or variants thereof) of crystals on either the Breuil--Kisin prism or the Fontaine prism, and thus are all bi-exact following similar argument as in Prop. \ref{Prop-dRStratification}.   
 We shall omit proofs of these bi-exactness in the paper.
\end{remark}

\subsection{Breuil--Kisin cosimplicial rings}
 
In this subsection, motivated by Proposition \ref{Prop-dRStratification}, we study the structure of the cosimplicial ring $\gs_\ast^\bullet$  with $*\in\{\emptyset,\log\}$; in particular, we give a detailed description of  $\frakS_*^n/(E)$ in Prop. \ref{pd polynomial}.
(Some of our computations are partly inspired by  computations in the relative prismatic case in \cite{Tia23}).


\begin{notation} \label{notebksimp}
\begin{enumerate}
\item  Let $n\geq 0$. By \cite[Example 2.6]{BS23} (for the prismatic case) and \cite[Lemma 5.0.2]{DL23} (for the   log-prismatic case), we have
 \begin{equation}\label{Equ-CechBK}
     \frakS^n = \rW(k)[[u_0,\dots,u_n]]\{\frac{u_0-u_1}{E(u_0)},\dots,\frac{u_0-u_n}{E(u_0)}\}^{\wedge}_{\delta}
 \end{equation}
 and
 \begin{equation}\label{Equ-CechBKlog}
      \frakS^n_{\log} = \rW(k)[[u_0,1-\frac{u_1}{u_0},\dots,1-\frac{u_n}{u_0}]]\{\frac{1-\frac{u_1}{u_0}}{E(u_0)},\dots,\frac{1-\frac{u_1}{u_n}}{E(u_0)}\}_{\delta}^{\wedge}.
 \end{equation}
 There is a natural morphism of cosimplicial rings $\frakS^{\bullet}\to\frakS_{\log}^{\bullet}$  given by identifying $u_i$'s. 
 
 \item 
 For any $n\geq 0$, we define $\rW(k)$-algebras
 \[A^n = \rW(k)[[u_0,v_1, \dots,v_n]] \text{ where }  v_i:=u_0-u_i  \]
 and 
 \[A^n_{\log} = \rW(k)[[u_0,v_{1,\log},\dots,v_{n,\log} ]] \text{ where }  v_{i, \log}:=1-\frac{u_i}{u_0}   \]
 with the $\delta$-structures given by requiring $\delta(u_i) = 0$ for all $i$.
 Let $*\in\{\emptyset,\log\}$, then we can summarize the equations in Item (1) as
\[ \frakS_{*}^n = A^n_*\{\frac{v_{1,*}}{E(u_0)},\dots,\frac{v_{n,*}}{E(u_0)}\}^{\wedge}_{\delta}. \]
 \end{enumerate}
\end{notation}

To prove Proposition \ref{pd polynomial}, we first give some lemmas.

\begin{lem}[\emph{\cite[Lem. 2.1.7]{ALB}}]\label{ALB}
  Let $(A,I)$ be a prism and let $d\in I$ be distinguished. If $(p,d)$ is a regular sequence in $A$, then for all $r,s\geq 0, r\neq s$, the sequences $(p,\varphi^r(d))$ and $(\varphi^r(d),\varphi^s(d))$ are regular.
\end{lem}

\begin{lem}\label{regular sequence}
  Let $*\in\{\emptyset,\log\}$.
  For any $n\geq j\geq 0$, the structure morphism $p_j:\frakS\rightarrow \frakS_*^n$ (cf. Notation \ref{notebksimp})
  is $(p,E(u_0))$-completely faithfully flat.  
  In particular, $(p,E(u_0))$ is a regular sequence in $\frakS_*^n$. 
\end{lem}
\begin{proof}
  As the sequence
  $E(u_0),v_{1,*}, \dots, v_{n,*}$
  is regular in $A_*^n$, the lemma follows from \cite[Proposition 3.13]{BS22} (and its proof).
\end{proof}

\begin{lem}\label{divisiblity}
  Let $*\in\{\emptyset,\log\}$. For any $i\geq j\geq 1$, $n\geq 0$ and $m\geq 1$, $E(u_0)$ divides $\varphi^m(\delta_n(\frac{v_{j,*}}{E(u_0)}))$ in $\frakS_*^{i}$.
\end{lem}
\begin{proof}
  Without loss of generality, we may assume $i = j = 1$ and put $x_* = \frac{v_{1,*}}{E(u_0)}$. By Taylor's expansion, we see that $v_{1,*}$ divides $\varphi(v_{1,*})$ (even in $A_{*}^1$), and thus divides $\varphi^m(v_{1,*})$ for any $m\geq 1$. As a consequence $E(u_0)$ divides $\varphi^m(v_{1,*})$ in $\frakS_*^1$ for any $m\geq 0$. Now, we are going to prove the result by induction on $n$.

  For $n=0$ and $m\geq 1$, we have $\varphi^m(x_*) = \frac{\varphi^m(v_{1,*})}{\varphi^m(E(u_0))}$.
  As $E(u_0)$ divides $\varphi^m(v_{1,*})$, by Lemma \ref{ALB} and Lemma \ref{regular sequence}, it also divides $\varphi^m(x_*)$ as desired.

  For any $n\geq 0$ and $m\geq 1$, as $p\varphi^m(\delta_{n+1}(x_*)) = \varphi^{m+1}(\delta_n(x_*))-\varphi^m(\delta_n(x_*))^p$, it follows from induction hypothesis that $E(u_0)$ divides $p\varphi^{m}(\delta_{n+1}(x_*))$.
  As $(p,E(u_0))$ is a regular sequence in $\frakS_*^1$, we see that $E(u_0)$ also divides $\varphi^m(\delta_{n+1}(x_*))$ as desired, and then complete the proof.
\end{proof}

\begin{prop}\label{pd polynomial}
  Let $*\in\{\emptyset,\log\}$. 
Let \begin{equation*}
a=
\begin{cases}
  -E'(\pi), &  \text{if } \ast=\emptyset \\
 -\pi E'(\pi), &  \text{if } \ast=\log
\end{cases}
\end{equation*}  
  For any $1\leq i\leq n$, let
  \[ \text{ $X_i=$   image of $\frac{v_{i,*}}{E(u_0)} $ in $\frakS_*^{n}/E$. }\] 
  Then \[\frakS_*^n/(E) \cong \calO_K\{X_1,\dots, X_n\}^{\wedge}_{\rm pd}\]
  where the right hand side is the $p$-completed pd-polynomial ring in the variables $X_1,\dots, X_n$. 
  Via the above isomorphisms for all $n$, the face maps $p_i$'s and degeneracy maps $\sigma_i$'s on $\frakS_*^{\bullet}/(E)$ are given by
  \begin{equation}\label{Equ-pd polynomial-Face}
      p_i(X_j) = \left\{
      \begin{array}{rcl}
           (X_{j+1}-X_1)(1+aX_1)^{-1}, & i=0  \\
           X_j, & j<i \\
           X_{j+1}, & 0<i\leq j,
      \end{array}
      \right.
  \end{equation}
  and 
  \begin{equation*}
      \sigma_i(X_j) = \left\{
            \begin{array}{rcl}
                0, & i=0~ \&~ j=1 \\
                X_{j-1}, & i<j~ \&~ j\neq 1\\
                X_j, & j\leq i.
            \end{array}
            \right.
  \end{equation*}
  Moreover, the canonical map 
  \begin{equation}   \label{eqmoreover211}
  \calO_K\{X_1,\dots, X_n\}^{\wedge}_{\rm pd}\cong\frakS^{\bullet}\to\frakS^{\bullet}_{\log}\cong\calO_K\{X_1,\dots, X_n\}^{\wedge}_{\rm pd} 
  \end{equation}
   carries $X_i$ to $\pi X_i$ for all $i$.
\end{prop}
\begin{proof}
  
  Since $(p,E=E(u_0))$ is a regular sequence in $\frakS_*^n$, we see that $\frakS_*^n/(E)$ is $p$-torsion free.
  For any $j$ and any $m\geq 0$, we denote by $\overline \delta_{m,j}$ the reduction of $\delta_m(\frac{v_{j,*}}{E})$ modulo $E$.
  Since for any $m\geq 1$,
  \[\varphi(\delta_{m-1}(\frac{v_{j,*}}{E})) = (\delta_{m-1}(\frac{v_{j,*}}{E}))^p+p\delta_m(\frac{v_{j,*}}{E}),\]
  By Lemma \ref{divisiblity}, we see that $p\overline \delta_{m+1,j} = -\overline{\delta_{m,j}}^p$ for any $m\geq 0$. By noting that $\overline \delta_{0,j} = X_j$, we conclude that for any $m\geq 0$,
  \[\overline \delta_{m,j} = (-1)^{1+p+\cdots+p^{m-1}}\frac{X_j^{p^m}}{p^{1+p+\cdots+p^{m-1}}}.\]
  Since for any $m\geq 0$, $X_j^{[p^m]}:=\frac{X_j^{p^m}}{p^m!}$\footnote{From now on, we denote by $X^{[n]}$ the $n$-th pd-power of $X$ (that is, $X^{[n]} = \frac{X^n}{n!}$). Similar remark applies to $X_1^{[n]}, X_2^{[n]}$, etc..} differs from $\overline \delta_{m,j}$ by a unit in $\Zp$, it is easy to check that $\frac{X_j^m}{m!}\in (\frakS_*^n/(E))[\frac{1}{p}]$ lies in $\frakS_*^n/(E)$.

  Let $\calO_K\{Y_1,\dots, Y_n\}^{\wedge}_{\pd}$ be the $p$-complete free pd-algebra in the variables $\{Y_j\}_{1\leq j\leq n}$ over $\calO_K$.
  Then there exists a well-defined pd-morphism
  \[\alpha: \calO_K\{Y_1,\dots, Y_n\}^{\wedge}_{\pd}\rightarrow \frakS^n/(E)\]
  sending $Y_j$ to $X_j$ for any $1\leq j\leq n$.
  The construction of $\frakS_*^n$ implies that $\alpha$ is surjective.
  Since both sides are $p$-complete, in order to see $\alpha$ is an isomorphism, we only need to see its reduction $\overline \alpha$ modulo $p$ is an injection. This follows from Lemma \ref{injectivity} below.

  It remains to check the cosimplicial structure on $\frakS^{\bullet}_*/(E)$ as claimed. We only deal with the $p_0$ case while the rests are obvious.
  Write $E(u) = \sum_{i=0}^ea_iu^i$ and then we have
  \begin{equation*}
      \begin{split}
          E(u_0)-E(u_1) & = \sum_{i=0}^ea_i(u_0^i-u_1^i) \\
          & = \sum_{i=1}^ea_i(\sum_{j=0}^i\binom{i}{j}u_1^{i-j}(u_0-u_1)^j-u_1^i) \\
          & = \sum_{i=1}^ea_i(\sum_{j=1}^i\binom{i}{j}u_1^{i-j}(u_0-u_1)^j) \\
          & \equiv \sum_{i=1}^eia_iu_1^{i-1}(u_0-u_1) \mod (u_0-u_1)^2.
      \end{split}
  \end{equation*}
  So we get
  \begin{equation*}
      \frac{E(u_1)}{E(u_0)} = 1-\frac{E(u_0)-E(u_1)}{E(u_0)} \equiv \left\{
      \begin{array}{rcl}
          1 - E'(u_0)X_1 \mod E, & * = \emptyset \\
          1-u_0E'(u_0)X_1 \mod E, & * = \log.
      \end{array}
      \right.
  \end{equation*}
  In other words, $\frac{E(u_0)}{E(u_1)}$ goes to $(1+aX_1)^{-1}$ in $\frakS^{\bullet}_*/(E)$. Now, one can conclude by noting that for any $1\leq j\leq n$,
  \begin{equation*}
  \begin{split}
      p_0(X_j) = p_0(\frac{v_{j,*}}{E(u_0)})= (\frac{v_{j+1,*}}{E(u_0)}-\frac{v_{1,*}}{E(u_0)})\frac{E(u_0)}{E(u_1)}\equiv (X_{j+1}-X_1)(1+aX_1)^{-1} \mod E.
  \end{split}
  \end{equation*}

  The ``moreover'' part follows immediately from the definitions of $\frakS^{\bullet}$ and $\frakS_{\log}^{\bullet}$. 
\end{proof}
\begin{lem}\label{injectivity}
   Keep the notations as above. The morphism $\bar \alpha: \calO_K\{Y_1,\cdots,Y_n\}^{\wedge}_{\pd}/p\to \frakS_*^n/(E,p)$ is injective.
\end{lem}
\begin{proof}
  We prove the case $n=1$. The general case follows from the same argument.

  Note that $\frakS_*^1=A_*^1\{X\}^{\wedge}_{\delta}/(EX-v_{1,*})_{\delta}$, where $(EX-v_{1,*})_{\delta}$ is the closure of the ideal generated by $\{\delta_n(EX-v_{1,*})\}_{n\geq 0}$.

  We claim that for any $n\geq 0$, $\delta_n(EX-v_{1,*})\equiv \delta_n(EX) \mod v_{1,*}$. We prove this by induction. When $n=0$, this is trivial. Assume the claim is true when $n=m$, and thus one can write $\delta_m(EX-v_{1,*})=\delta_m(EX)+v_{1,*}Z_1$ for some $Z_1\in A_*^1\{X\}^{\wedge}_{\delta}$. Then we have
  \begin{equation*}
      \begin{split}
      \delta_{m+1}(EX-v_{1,*}) 
      = &\delta(\delta_m(EX-v_{1,*}))\\
      = &\delta(\delta_m(EX)+v_{1,*}Z_1)\\
      = &\delta_{m+1}(EX)+\delta(v_{1,*}Z_1)-\sum_{i=1}^{p-1}\frac{(p-1)!}{i!(p-i)!}\delta_m(EX)^{p-i}v_{1,*}^iZ_1^i\\
      = & \delta_{m+1}(EX)+\delta(v_{1,*})\varphi(Z_1)+v_{1,*}^p\delta(Z_1)-\sum_{i=1}^{p-1}\frac{(p-1)!}{i!(p-i)!}\delta_m(EX)^{p-i}v_{1,*}^iZ_1^i\\
      \equiv &\delta_{m+1}(EX)+\delta(v_{1,*})\varphi(Z_1) \mod v_{1,*}\\
      \equiv &\delta_{m+1}(EX)\mod v_{1,*},
      \end{split}
  \end{equation*}
  where the last congruence follows from $\delta(v_{1,*})\equiv 0\mod{v_{1,*}}$. So the claim is true.

  Thanks to the above claim, we can see that
  \begin{equation*}
      \begin{split}
           \frakS_*^1/(E,p)= & A_*^1\{X\}^{\wedge}_{\delta}/(p,E,EX-v_{1,*},\delta_1(EX-v_{1,*}),\cdots,\delta_n(EX-v_{1,*}),\cdots)\\
           = & A^1_*\{X\}^{\wedge}_{\delta}/(p,E,v_{1,*},\delta_1(EX),\cdots,\delta_n(EX),\cdots)\\
           = & \frakS\{X\}^{\wedge}_{\delta}/(p,E,\delta_1(EX),\cdots,\delta_n(EX),\cdots).
      \end{split}
  \end{equation*}
  
 Let $(\frakS\langle T\rangle^1, (E))$ be the self-product of $(\frakS\langle T\rangle, (E))$ with $\delta(T)=0$ in the category $(\calO_K\langle T\rangle/(\frakS,(E)))_{\Prism}$. Then we have $\frakS\za T\ya^1=\frakS\za T_0,T_1\ya\{X\}^{\wedge}_{\delta}/(EX-(T_0-T_1))_{\delta}$. Then by similar arguments as above, one can prove that for any $n\geq 0$,
 \[\delta_n(EX-(T_0-T_1))\equiv \delta_n(EX)\mod T_0-T_1,\]
 and furthermore that
 \begin{equation*}
 \begin{split}
         \frakS\langle T\rangle^1/(E,p)=&\frakS\langle T_0,T_1\rangle\{X\}^{\wedge}_{\delta}/(p,E,EX-(T_0-T_1),\delta_1(EX),\cdots,\delta_n(EX),\cdots)\\
         =& \frakS\langle T\rangle\{X\}^{\wedge}_{\delta}/(p,E,\delta_1(EX),\cdots,\delta_n(EX),\cdots).
 \end{split}
 \end{equation*}
 So in particular, we have $\frakS_*^1/(E,p)\otimes_{\calO_K/p}\calO_K/p[T]\cong \frakS\langle T\rangle^1/(E,p)$. 
 
 Now we consider the following commutative diagram
 \begin{equation*}
       \xymatrix@C=0.5cm{
    \calO_K/p\{Y_1\}^{\wedge}_{\pd}\ar[rrrr]^{}\ar[d]^{\otimes_{\calO_K/p}\calO_K/p[T]}&&&& \frakS_*^1/(p,E)\ar[d]^{\otimes_{\calO_K/p}\calO_K/p[T]}\\
    \calO_K/p[T]\{Y_1\}^{\wedge}_{\pd}\ar[rrrr]&&&& \frakS\langle T\rangle^1/(p,E),
  }
  \end{equation*}
 where the bottom map sends $Y_j$ to the reduction of $\frac{T_0-T_1}{E}\in\frakS\za T\ya^1$ is well-defined and injective due to \cite[Proposition 5.7]{Tia23}.
 The vertical maps are injective due to the faithful flatness of the map $\calO_K/p\to \calO_K/p[T]$.  So the top map is also injective as desired. We are done.
\end{proof}

 \section{Stratifications:   axiomatic computations}\label{Strat}

The purpose of this section is to carry out an axiomatic study of stratifications, preparing for the discussions on Hodge--Tate crystals in Section \ref{sectionHT}.
Indeed, in Notation \ref{cosimplicial ring structure}, we define a slightly more general cosimplicial ring which generalizes both the prismatic and log-prismatic Breuil--Kisin cosimplicial ring. In \S \ref{subsecaxiomstrat}, we classify stratifications over this cosimplicial ring; in \S  \ref{axiomcoho}, we study \v{C}ech-Alexander complexes associated to these stratifications.

\begin{notation} \label{cosimplicial ring structure}
Let $R$ be a $p$-complete ring, and for any $n\geq 0$, define 
 \[A^{n}:=R[X_1,\dots,X_n]^{\wedge}_{\pd}\] to be the $p$-completed free pd-algebra over $R$ generated by free variables $X_1,\dots,X_n$ (thus, $X_i$'s are all topologically nilpotent in $A^n$). This mimicks the $p$-complete pd-polynomial ring appearing in Proposition \ref{pd polynomial}.
 
  Fix an element $a\in R$. For any $n\geq 0$, we define $R$-linear morphisms
    \begin{enumerate}
        \item[(a)] $p_i:A^n\to A^{n+1}$ for any $0\leq i\leq n+1$ such that
        \begin{equation*}
            p_i(X_j) = \left\{
            \begin{array}{rcl}
                (X_{j+1}-X_1)(1+aX_1)^{-1}, & i=0 \\
                X_j, & j<i \\
                X_{j+1}, & 0<i\leq j
            \end{array}
            \right.
        \end{equation*}
         for any $1\leq j\leq n$, and
        \item[(b)] $\sigma_i: A^{n+1}\to A^n$ for any $0\leq i\leq n$ such that
        \begin{equation*}
            \sigma_i(X_j) = \left\{
            \begin{array}{rcl}
                0, & i=0  \text{ and } j=1 \\
                X_{j-1}, & i<j \text{ and } j\neq 1\\
                X_j, & j\leq i
            \end{array}
            \right.
        \end{equation*}
        for any $1\leq j\leq n+1$.
    \end{enumerate}
\end{notation}

\begin{lem}\label{Lem-CosimplicialRing}
   
    We have 
    \begin{enumerate}
        \item[(1)] $A^{\bullet}$ is a cosimplicial $R$-algebra with the $i$-th face map $p_i$ and $i$-th degeneracy map $\sigma_i$;
        \item[(2)] the kernel $J$ of $\sigma_0:A^1\to A = R$ is the ($p$-complete) $pd$-ideal generated by $X_1$ and $J/J^{[2]}\cong R$, where $J^{[2]}$ is the $p$-complete ideal generated by $\{X_1^{[i]}\mid i\geq 2\}$.
    \end{enumerate}
\end{lem}
\begin{proof}
Item (1) is just a formal generalisation of the cosimplicial algebra associated with the Breuil--Kisin prism (see Proposition \ref{pd polynomial}). Item (2) is easy to check. \end{proof}

\subsection{Computation of stratifications}\label{subsecaxiomstrat}
  Note that $A^{\bullet}$ in Notation \ref{cosimplicial ring structure} is $p$-complete. So one can define the category of $p$-complete stratifications as follows:

\begin{defn}\label{p-complete stratification-dfn}
Let $A^{\bullet}$ be in Notation \ref{cosimplicial ring structure}, a \emph{$p$-complete stratification with respect to $A^{\bullet}$} is a pair $(M,\varepsilon)$ consisting of a $p$-adically complete $A^0$-module $M$ and an $A^1$-linear isomorphism
 \[\varepsilon: M\widehat \otimes_{A^0,p_0}A^1\to M\widehat \otimes_{A^0,p_1}A^1,\]
 such that the \emph{cocycle condition} is satisfied:
 \begin{enumerate}
 \item  $p_2^*(\varepsilon)\circ p_0^*(\varepsilon) = p_1^*(\varepsilon): M\widehat \otimes_{A^0,q_2}A^2\to M\widehat \otimes_{A^0,q_0}A^2$;
\item  $\sigma_0^*(\varepsilon) = \id_M$.
 \end{enumerate}
Write $\mathrm{Strat}(A^\bullet)^{\wedge}_p$ for the category of $p$-complete stratifications with respect to $A^\bullet$.
\end{defn}
 Note that finite projective $A^0$-modules are automatically $p$-complete, so we have a natural embedding of categories $\mathrm{Strat}(A^\bullet)\subset \mathrm{Strat}(A^\bullet)^{\wedge}_p$. The purpose of this part is to give an explicit description of the stratifications $(M,\varepsilon_M)$ with respect to $A^{\bullet}$ and compute the homotopy groups of $\rC(M,\varepsilon_M)$, the cochain complex attached to $M\widehat \otimes_{A^0,q_0}A^{\bullet}$.
  
\begin{rem}
    In fact, the main scope of this paper only concerns $\mathrm{Strat}(A^\bullet)$, i.e, the finite projective objects, as in our main theorems stated in the introduction. Nonetheless,  all the computations in the following work \emph{verbatim} for general $p$-completed objects in $\mathrm{Strat}(A^\bullet)^{\wedge}_p$, and thus we include them for future references. Note this also points to the possibility to ``derive" all our results in this paper, cf. the various remarks in \S\ref{subsecliterature}.
\end{rem}

\begin{construction} \label{cons34}
    Let $(M,\varepsilon_M)$ be a stratification in $\Strat(A^{\bullet})^{\wedge}_p$. Then the difference 
    \[\varepsilon_M-\id_M: M\to M\widehat \otimes_{A^0,p_1}A^1\] is an $R$-linear map taking values in $M\widehat \otimes_{A^0,p_1}J$. Denote by $\phi_M$ the $R$-linear map induced by 
    \[M\xrightarrow{\varepsilon_M-\id_M} M\widehat \otimes_{A^0,p_1}J\to M\widehat \otimes_{A^0,p_1}J/J^{[2]} \xrightarrow{\cong} M. \]
    Equivalently, if we write 
    \begin{equation}\label{Equ-Varepsilon}
      \varepsilon_M(m) = \sum_{i\geq 0}\phi_i(m)X_1^{[i]}
    \end{equation}
    with $\phi_i\in\End_R(M)$ satisfying $\lim_{i\to+\infty}\phi_i = 0$ (and $\phi_0 = \id_M$ as $\sigma_0^*(\varepsilon_M) = \id_M$), then \[\phi_M=\phi_1.\]
\end{construction}


\begin{thm}\label{Thm-AxiomHT}
\begin{enumerate}
    \item The rule $(M,\varepsilon_M)\mapsto (M,\phi_M)$ in Construction \ref{cons34} induces an equivalence from the category $\Strat(A^{\bullet})^{\wedge}_p$ to the category of pairs $(M,\phi_M)$ consisting of a $p$-complete $R$-module $M$ and a $\phi_M\in\End_R(M)$ satisfying the condition
    \begin{equation}\label{Equ-Nilpotent}
        \lim_{n\to+\infty}\prod_{i=0}^{n-1}(\phi_M-ia) = 0.
    \end{equation}
    The quasi-inverse of the above functor is defined by sending $(M,\phi_M)$ to $(M,\varepsilon_M)$
    with
    \begin{equation}     \label{eqinverse33}
    \varepsilon_M=(1+aX_1)^{\frac{\phi_M}{a}}:=\sum_{n\geq 0}\left(\prod_{i=0}^{n-1}(\phi_M-ia)\right)X_1^{[n]},
    \end{equation} 
    where  the first summand with $n = 0$ is $\id_M$.
    Moreover,   all the   equivalences  are bi-exact, preserves ranks, tensor products and duals. 
    \item Let $(M,\phi_M)$ correspond to $(M,\varepsilon_M)$. Then there exists an $R$-linear quasi-isomorphism
    \[\rho_M:[M\xrightarrow{\phi_M}M]\xrightarrow{\simeq}\rC(M,\varepsilon_M)\]
    which is functorial in $(M,\varepsilon_M)$; here the right hand side is the cochain complex attached to $M\widehat \otimes_{A^0,q_0}A^{\bullet}$.
    \end{enumerate}
\end{thm}

\begin{proof}
  We first prove Item (1) here; Item (2) will be proved in \S \ref{axiomcoho}.
   Let $(M,\varepsilon_M=\sum_{n\geq 0}\phi_nX_1^{[n]})$ be a stratification in $\Strat(A^{\bullet})^{\wedge}_p$. As $\sigma_0^*(\varepsilon_M) = \id_M$, we have $\phi_0=\id_M$. On the other hand, we have
  \begin{equation*}
      \begin{split}
          p_2^*(\varepsilon_M)\circ p_0^*(\varepsilon_M)&=p_2^*(\varepsilon_M)(\sum_{n\geq 0}\phi_np_0(X_1^{[n]}))\\
          &=\sum_{m\geq 0,n\geq 0}\phi_m\circ\phi_np_2(X_1^{[m]})(1+aX_1)^{-n}(X_2-X_1)^{[n]}\\
          &=\sum_{l,m,n\geq 0}\phi_m\circ\phi_{l+n}(1+aX_1)^{-l-n}(-1)^lX_1^{[l]}X_1^{[m]}X_2^{[n]},
      \end{split}
  \end{equation*}
  and 
  \begin{equation*}
      p_1^*(\varepsilon_M) = \sum_{n\geq 0}\phi_np_1(X_1^{[n]}) = \sum_{n\geq 0}\phi_nX_2^{[n]}.
  \end{equation*}
  As a consequence, we conclude that $(M,\varepsilon_M)$ satisfying the cocycle condition if (and only if) $\phi_0 = \id_M$ and for any $n\geq 0$,
  \begin{equation}\label{eqnewrecu}
  \phi_n = \sum_{l,m\geq 0}\phi_m\circ\phi_{l+n}(1+aX_1)^{-l-n}(-1)^lX_1^{[l]}X_1^{[m]}.
  \end{equation} 
  We shall prove in the following Proposition \ref{Prop-SolveCocycle}, that the recurrence condition \eqref{eqnewrecu} is equivalent to the condition that for any $n\geq 1$, 
  \[ \phi_n = \prod_{i=0}^{n-1}(\phi_1-ia).\]
 Thus in particular the condition \eqref{Equ-Nilpotent} has to hold, and thus the functor as described in Item (1) is defined; the fact that Eqn. \eqref{eqinverse33} defines a quasi-inverse also follow from above discussions.  A standard argument shows that the above equivalence preserves tensor products and duals.  Bi-exactness follow obviously from constructions.
  \end{proof}

  The following key result is used in proof of  Theorem \ref{Thm-AxiomHT}(1); it gives another characterisation of   condition \eqref{eqnewrecu} in the proof.
  
  \begin{prop}\label{Prop-SolveCocycle}
  Let $M$ be a $p$-complete $R$-module.
      Let $\{\phi_n\}_{n\geq 0}\subset\End_R(M)$ with $\phi_0 = \id_M$. Then the following conditions are equivalent:
      \begin{enumerate}
          \item[(1)] For any $n\geq 0$, $\phi_n = \sum_{l,m\geq 0}\phi_m\circ\phi_{l+n}(1+aX)^{-l-n}(-1)^lX^{[l]}X^{[m]}$.

          \item[(2)] For any $n\geq 1$, $\phi_n = \prod_{i=0}^{n-1}(\phi_1-ia)$.
      \end{enumerate}
  \end{prop}
  \begin{proof}
      $(1)\Rightarrow(2):$ Assume we have for any $n\geq 0$,
      \begin{equation}\label{Equ-Key-I}
      \phi_n = \sum_{l,m\geq 0}\phi_m\circ\phi_{l+n}(1+aX)^{-l-n}(-1)^lX^{[l]}X^{[m]}.
  \end{equation}
  Comparing the coefficients of $X$ on both sides of (\ref{Equ-Key-I}), we have
  \[\phi_1\circ\phi_n-\phi_{1+n}-an\phi_n = 0.\]
  In other words, for any $n\geq 0$, we have 
  \begin{equation}\label{Equ-Key-II}
      \phi_{n+1} = (\phi_1-an)\circ\phi_n.
  \end{equation}
  Then we obtain Item (2) by iterations.

  $(2)\Rightarrow(1):$ Note that under the assumption, all $\phi_n$'s commute with each other such that (\ref{Equ-Key-II}) holds true. We first show (\ref{Equ-Key-I}) is true for $n=0$. That is, we want to show that
  \begin{equation}\label{Equ-Key-III}
      \id_M = \sum_{l,m\geq 0}\phi_m\circ\phi_{l}(1+aX)^{-l}(-1)^lX^{[l]}X^{[m]}.
  \end{equation}
  Recall in $\bQ[[a^{\pm 1},X,Y]]$ (where we regard $a,X,Y$ as free variables), we always have 
  \[\begin{split}
      (1+aX)^{\frac{Y}{a}} & = \sum_{n\geq 0}\frac{\frac{Y}{a}(\frac{Y}{a}-1)\cdots(\frac{Y}{a}-n+1)}{n!}(aX)^n\\
      & = \sum_{n\geq 0}\prod_{i=0}^{n-1}(Y-ia)X^{[n]}.
  \end{split}\]
  Replacing $X$ by $-(1+aX)^{-1}X$ above, we see that
  \[(1+aX)^{-\frac{Y}{a}} = (1-a(1+aX)^{-1}X)^{\frac{Y}{a}} = \sum_{n\geq 0}\prod_{i=0}^{n-1}(Y-ia)(-1)^n(1+aX)^{-n}X^{[n]}.\]
  As a consequence, we always have 
  \[1 = \sum_{l,m\geq 0}\left(\prod_{i=0}^{m-1}(Y-ia)\right)\left(\prod_{j=0}^{l-1}(Y-ja)\right)(-1)^l(1+aX)^{-l}X^{[l]}X^{[m]}.\]
  Now, by letting $Y = \phi_1$ above and comparing the coefficients of $X^{[n]}$'s, we see that (\ref{Equ-Key-III}) holds true.

  Now, assume we have proven (\ref{Equ-Key-I}) for some $n\geq 0$. Then we have
    \begin{equation*}
      \begin{split}
          \phi_{n+1} 
          = & (\phi_1-an)\circ\phi_n\\
          = & \sum_{l,m\geq 0}\phi_m\circ(\phi_1\circ\phi_{l+n}-an\phi_{l+n})(1+aX)^{-l-n}(-1)^lX^{[l]}X^{[m]}\\
          = & \sum_{l,m\geq 0}\phi_m\circ(\phi_{1+l+n}+al\phi_{l+n})(1+aX)^{-l-n}(-1)^lX^{[l]}X^{[m]}\\
          =& \sum_{l,m\geq 0}\phi_m\circ\phi_{1+l+n}(1+aX)^{-l-n}(-1)^lX^{[l]}X^{[m]}+ \sum_{m\geq 0,l\geq 1}\phi_m\circ\phi_{l+n}aX(1+aX)^{-l-n}(-1)^lX^{[l-1]}X^{[m]}\\
          =& \sum_{l,m\geq 0}\phi_m\circ\phi_{1+l+n}(1+aX)^{-l-n}(-1)^lX^{[l]}X^{[m]}+ \sum_{l,m\geq 0}\phi_m\circ\phi_{l+1+n}aX(1+aX)^{-1-l-n}(-1)^{l+1}X^{[l]}X^{[m]}\\
          =&(1+aX-aX)\sum_{l,m\geq 0}\phi_m\circ\phi_{1+l+n}(1+aX)^{-l-n-1}(-1)^lX^{[l]}X^{[m]}\\
          =&\sum_{l,m\geq 0}\phi_m\circ\phi_{1+l+n}(1+aX)^{-l-n-1}(-1)^lX^{[l]}X^{[m]}.
      \end{split}
  \end{equation*}
  Now, we can conclude Item (1) by induction on $n$.
  \end{proof}

  \subsection{Computation of cohomology}\label{axiomcoho}
This subsection is devoted to the proof of Theorem \ref{Thm-AxiomHT}(2). We first construct the morphism $\rho_M:[M\xrightarrow{\phi_M}M]\to\rC(M,\varepsilon_M)$ as mentioned in Theorem \ref{Thm-AxiomHT}(2). For simplicity, we put 
  \[\rC^n:=M\widehat \otimes_{R,q_0}A^n = M\widehat \otimes_{R,q_0}R\{X_1,\dots,X_n\}^{\wedge} = M\{X_1,\dots,X_n\}^{\wedge},\]
  and denote the $i$-th face map $p_{i,M}$ by $p_i$.
  By Notation \ref{cosimplicial ring structure}(a), for any $mX_1^{[i_1]}\cdots X_r^{[i_n]}\in \rC^n$, we have 
  \begin{equation}\label{Equ-FaceMod}
      p_j(mX_1^{[i_1]}\cdots X_r^{[i_n]}) = \left\{
      \begin{array}{cc}
          \varepsilon_M(m)(1+aX_1)^{-i_1-\cdots-i_n}(X_2-X_1)^{[i_1]}\cdots(X_{n+1}-X_1)^{[i_n]}, & j=0\\
          mX_1^{[i_1]}\dots X_{j-1}^{[i_{j-1}]}X_{j+1}^{[i_{j+1}]}\cdots X_{n+1}^{i_{n+1}}, & j\geq 1.
      \end{array}
      \right.
  \end{equation}
  The differential $\rd^n:\rC^n\to\rC^{n+1}$ is then induced by $\rd^n = \sum_{i=0}^{n+1}(-1)^ip_i$. In particular, we see that for any $m\in M$,
  \begin{equation}\label{Equ-d0}
      \rd^0(m) = \varepsilon_M(m)-m = \sum_{n\geq 1}\prod_{i=0}^{n-1}(\phi_M-ia)(m)X_1^{[n]} = \sum_{n\geq 1}\prod_{i=1}^{n-1}(\phi_M-ia)(\phi_M(m))X_1^{[n]}.
  \end{equation}
  
  Now, we are able to construct the morphism $\rho_M$ in Theorem \ref{Thm-AxiomHT}.
  
  \begin{construction}\label{Construction-Rho}
      Define a formal series \[\rho(Y,X):=\frac{(1+aX)^{\frac{Y}{a}}-1}{Y} = \sum_{n\geq 1} \left(\prod_{i=1}^{n-1}(Y-ia)\right) X^{[n]}.\]
      Then we have
      \[\rho(\phi_M,X_1)\circ\phi_M = \varepsilon_M-\id_M=\rd^0.\]
      Equivalently, the following diagram
      \[\begin{tikzcd}
M \arrow[d, "="] \arrow[rr, "\phi_M"] &  & M \arrow[d, "{\rho(\phi_M,X_1)}"] \\
 \rC^0 \arrow[rr, "\rd^0"]            &  & \rC^1                            
\end{tikzcd}\]
      commutes and induces a morphism of complexes of $R$-modules 
      \[\rho_M:[M\xrightarrow{\phi_M}M]\to \rC(M,\varepsilon_M).\]
  \end{construction}
  
  \begin{proof}[Proof of Theorem \ref{Thm-AxiomHT}(2)]
  Using the construction of $\rho_M$, Theorem \ref{Thm-AxiomHT}(2) then follows from Proposition \ref{Prop-Cohomology} below. 
  \end{proof}
  \begin{prop}\label{Prop-Cohomology}
      \begin{enumerate}
          \item[(1)] $\Ker(\phi_M)=\Ker(\rd^0)$.

          \item[(2)] We have $\Ker(\rd^1) = \rho(\phi_M,X_1)M$. As a result, $\rho_M$ induces an isomorphism 
          \[M/\phi_M(M)\cong\rH^1(\rC(M,\varepsilon_M)) = \rho(\phi_M,X_1)(M)/\rho(\phi_M,X_1)(\phi_M(M)).\]

          \item[(3)] For any $i\geq 2$, $\rH^i(\rC(M,\varepsilon_M)) = 0$.
      \end{enumerate}
  \end{prop}
  \begin{proof} 
  For Item (1):  The commutative diagram implies $\Ker(\phi_M)\subset\Ker(\rd^0)$.
      On the other hand, suppose $\rd^0(m) = 0$, then the coefficient of $X_1$ in $\epsilon_M(m)$, namely, $\phi_M(m)$ vanishes. So we have $\Ker(\rd^0)\subset\Ker(\phi_M)$ and we can conclude.

      For Item (2): Assume $\sum_{i\geq 0}m_iX_1^{[i]}\in\Ker(\rd^1)$. By (\ref{Equ-FaceMod}), we have
      \[
      \begin{split}\rd^1(\sum_{i\geq 0}m_iX_1^{[i]}) =& \sum_{i\geq 0}(\varepsilon_M(m_i)(1+aX_1)^{-i}(X_2-X_1)^{[i]}-m_iX_2^{[i]}+m_iX_1^{[i]})\\
      &=\sum_{i,j\geq 0}\varepsilon_M(m_{i+j})(1+aX_1)^{-i-j}(-1)^iX_1^{[i]}X_2^{[j]}-\sum_{j\geq 0}m_jX_2^{[j]}+\sum_{i\geq 0}m_iX_1^{[i]}\\
      &=\sum_{i,j,k\geq 0}\prod_{l=0}^{k-1}(\phi_M-la)(m_{i+j})X_1^{[k]}(1+aX_1)^{-i-j}(-1)^iX_1^{[i]}X_2^{[j]}-\sum_{j\geq 0}m_jX_2^{[j]}+\sum_{i\geq 0}m_iX_1^{[i]}\\
      &=0.
      \end{split}\]
      By letting $X_1=X_2=0$, we see that \[m_0 = 0.\]
      On the other hand, for any $j\geq 1$, comparing the coefficients of $X_2^{[j]}$, we have
      \[m_j = \sum_{i\geq 0}\varepsilon_M(m_{i+j})(1+aX_1)^{-i-j}(-1)^iX_1^{[i]}=\sum_{i,k\geq 0}\prod_{l=0}^{k-1}(\phi_M-la)(m_{i+j})X_1^{[k]}(1+aX_1)^{-i-j}(-1)^iX_1^{[i]}.\]
      Comparing the coefficients of $X_1$, we see that for any $j\geq 1$,
      \[\phi_M(m_j)-m_{1+j}-jam_j = 0\]
      and hence \[
         m_{1+j} = (\phi_M-aj)(m_j).      \]
        By iteration, we conclude that 
         \begin{equation} \label{eq1stdetall}
         m_j = \prod_{i=1}^{j-1}(\phi_M-ia)(m_1).
         \end{equation} Therefore, we have
      \[\sum_{i\geq 0}m_iX_1^{[i]} = m_0+\rho(\phi_M,X_1)(m_1) =0+\rho(\phi_M,X_1)(m_1)= \rho(\phi_M,X_1)(m_1).\]
      Therefore, we have $\Ker(\rd^1)\subset\rho(\phi_M,X_1)M$. 

      It remains to prove $\rho(\phi_M,X_1)M\subset\Ker(\rd^1)$. 
      By the definition of $\rho$, we see that in $\bQ[[a^{\pm 1},X_1,X_2,Y]]$,
      \begin{equation*}
          \begin{split}
              (1+aX_1)^{\frac{Y}{a}}\rho(Y,(1+aX_1)^{-1}(X_2-X_1))
              =&(1+aX_1)^{\frac{Y}{a}}\frac{(1+a(1+aX_1)^{-1}(X_2-X_1))^{\frac{Y}{a}}-1}{Y}\\
              =&(1+aX_1)^{\frac{Y}{a}}\frac{((1+aX_2)(1+aX_1)^{-1})^{\frac{Y}{a}}-1}{Y}\\
              =&\frac{(1+aX_2)^{\frac{Y}{a}}-(1+aX_1)^{\frac{Y}{a}}}{Y}\\
              =&\rho(Y,X_2)-\rho(Y,X_1).
          \end{split}
      \end{equation*}
      By letting $Y = \phi_M$, we then have
      \[\rd^1\circ\rho(\phi_M,X_1)=\varepsilon_M\circ\rho(\phi_M,(1+aX_1)^{-1}(X_2-X_1))-\rho(\phi_M,X_2)+\rho(\phi_M,X_1)=0,\]
      which implies that $\rho(\phi_M,X_1)M\subset\Ker(\rd^1)$ as desired.

      Writing $\rho(\phi_M,X_1)$ as ${\rm id}\cdot X_1+\sum_{n\geq 2} \left(\prod_{i=1}^{n-1}(\phi_M-ia)\right) X_1^{[n]}$, we see that $\rho(\phi_M,X_1)$ is injective, from which we conclude
      \[M/\phi_M(M)\cong\rH^1(\rC(M,\varepsilon_M)) = \rho(\phi_M,X_1)(M)/\rho(\phi_M,X_1)(\phi_M(M)).\]

For Item (3), we need to prove $\Ker(\rd^{s+1})\subset\Ima(\rd^s)$. 
The proof is combinatorial,  and is a bit long; we leave it to Section \ref{higher vanishing}. The main idea is to show that for an element in the kernel \[ g=\sum_{I\in\bN^{s+1}}m_I\underline X^{[I]}\in\Ker(\rd^{s+1}),\] 
the coefficients  $m_I$ (for \emph{all} the indices $I$) are actually determined by those $m_J$ with $J$ in an explicit (and much smaller) subset of indices $\Lambda^{s+1} \subset \bN^{s+1}$. See also Construction \ref{constr3stepssec9} for a slightly expanded explanation of strategy.
  \end{proof}




Finally, we point out that Theorem \ref{Thm-AxiomHT} also holds true after inverting $p$. 
  \begin{cor}\label{rationalstrat}
    \begin{enumerate}
    \item The rule $(M,\varepsilon_M)\mapsto (M,\phi_M)$ induces an equivalence from the category $\Strat(A^{\bullet}[\frac{1}{p}])$ to the category of pairs $(M,\phi_M)$ consisting of a finite projective $R[\frac{1}{p}]$-module $M$ and a $\phi_M\in\End_R(M)$ satisfying the condition
    \[
        \lim_{n\to+\infty}\prod_{i=0}^{n-1}(\phi_M-ia) = 0.
    \]
    The quasi-inverse of the above functor is induced by sending $(M,\phi_M)$ to $(M,\varepsilon_M)$
    with
    \[\varepsilon_M=(1+aX_1)^{\frac{\phi_M}{a}}:=\sum_{n\geq 0}(\prod_{i=0}^{n-1}(\phi_M-ia))X_1^{[n]}.\]
    Moreover, the equivalence is bi-exact, and preserves ranks, tensor products and duals, 
    \item Let $(M,\phi_M)$ correspond to $(M,\varepsilon_M)$. Then there exists an $R[1/p]$-linear quasi-isomorphism
    \[\rho_M:[M\xrightarrow{\phi_M}M]\xrightarrow{\simeq}\rC(M,\varepsilon_M)\]
    which is functorial in $(M,\varepsilon_M)$.
    \end{enumerate}
  \end{cor}
  \begin{proof}
      All constructions and calculations for the proof of Theorem \ref{Thm-AxiomHT} still hold true after inverting $p$.  
  \end{proof}

\section{Hodge--Tate crystals and Sen modules}\label{sectionHT}
In this section, we will give a linear-algebraic classification of Hodge--Tate crystals and discuss their cohomology using results in Section \ref{Strat}.

\subsection{Classification of Hodge--Tate crystals}

\begin{defn}
 Let $*\in\{\emptyset,\log\}$. 
Let \begin{equation*}
a=
\begin{cases}
  -E'(\pi), &  \text{if } \ast=\emptyset \\
 -\pi E'(\pi), &  \text{if } \ast=\log
\end{cases}
\end{equation*}  
Let  $\End_{\calO_K}^{\ast-\nht}$ (resp. $\End_{K}^{\ast-\nht}$)  be the category consisting of pairs $(M,\phi_M)$ which we call a \emph{module equipped with $a$-small endomorphism}, where
    \begin{enumerate}
        \item[(1)] $M$ is a finite free $\calO_K$-module (resp. $K$-vector space), and
        \item[(2)] $\phi_M$ is an $\calO_K$-linear (resp. $K$-linear) endomorphism of $M$ such that 
        \begin{equation}
            \lim_{n\to+\infty}\prod_{i=0}^{n-1}(\phi_M-ai) = 0.
        \end{equation} 
            \end{enumerate}
\end{defn}


We have the following theorem.
\begin{thm}\label{Thm-HTCrystal}
There are two commutative diagrams of functors
\begin{equation}\label{Diag-HTCrystals}
   \begin{tikzcd}
{\Vect((\calO_K)_{\Prism},\overline \calO_{\Prism})} \arrow[d, "\simeq"] \arrow[rr, hook] &  & {\Vect((\calO_K)_{\Prism,\log},\overline \calO_{\Prism})} \arrow[d, "\simeq"] \\
\End_{\ok}^{\nht}  \arrow[rr, "{}", hook]       &  & \End_{\calO_K}^{\log-\nht}                                                   
\end{tikzcd}
   \end{equation}
   and
   \[
   \begin{tikzcd}
{\Vect((\calO_K)_{\Prism},\overline \calO_{\Prism}[1/p])} \arrow[d, "\simeq"] \arrow[rr, hook] &  & {\Vect((\ok)_{\Prism,\log},\overline \calO_{\Prism}[1/p])} \arrow[d, "\simeq"] \\
\End_{K}^{\nht}  \arrow[rr, "{}", hook]       &  & \End_{K}^{\log-\nht}                                                   
\end{tikzcd}
\]
Here both the top horizontal rows are induced by the  forgetful functor $(\calO_K)_{\Prism,\log}\to(\calO_K)_{\Prism}$; both the bottom horizonal rows are induced by the assignment
   \[(M,\phi_M)\mapsto(M,\pi\phi_M).\]
   The vertical equivalences are induced by evaluation at $(\frakS,(E),*)$.
     Moreover, all the  vertical equivalences  are bi-exact, preserve  ranks, tensor products and duals. 
\end{thm}
\begin{proof}
    The vertical equivalences follow from Proposition \ref{Prop-dRStratification} and Theorem \ref{Thm-AxiomHT}(1). It suffices to check   commutativity of the diagrams.
 For any $\bM\in\Vect((\calO_K)_{\Prism,*},\overline \calO_{\Prism})$ with induced pair $(M,\phi_M)\in\End_{\calO_K}^{\ast-\nht}$, the stratification $(M,\varepsilon)$ with respect to $\frakS^{\bullet}_*/(E)\cong\calO_K\{X_1,\dots,X_{\bullet}\}^{\wedge}_{\pd}$ (cf. Proposition \ref{pd polynomial}) is given by the formula in Theorem \ref{Thm-AxiomHT}(1):
    \[\varepsilon_M=(1+aX_1)^{\frac{\phi_M}{a}} = \ \sum_{n\geq 0}\left(\prod_{i=0}^{n-1}(\phi_M-ia)\right) X_1^{[n]}\] 
Let $\bM\in\Vect((\calO_K)_{\Prism},\overline \calO_{\Prism})$ (on the prismatic site) with induced pair $(M,\phi_M)\in\End_{\calO_K}^{\nht}$. 
Now, consider $\bM$ as an object  in $\Vect((\calO_K)_{\Prism,\log},\overline \calO_{\Prism})$; its induced stratification with respect to $\frakS^{\bullet}_{\log}/(E)$ is given by base change along the map  \eqref{eqmoreover211} in Proposition \ref{pd polynomial},  which is
\[\varepsilon_M= \sum_{n\geq 0}\prod_{i=0}^{n-1}(\phi_M+iE'(\pi))(\pi X_1)^{[n]} = \sum_{n\geq 0}\prod_{i=0}^{n-1}(\pi\phi_M+i\pi E'(\pi))X_1^{[n]}\]
Thus its induced pair in $\End_{\calO_K}^{\log-\nht}$ is $(M,\pi\phi_M)$. This implies the commutativity of (\ref{Diag-HTCrystals}).  
\end{proof}

\begin{remark}
It is obvious that $\End_{K}^{\ast-\nht}$ is the isogeny category of $\End_{\calO_K}^{\ast-\nht}$. However, \emph{a priori}, it is not clear if $\Vect((\ok)_{\Prism,\ast},\overline \calO_{\Prism}[1/p])$ should be the isogeny category of $\Vect((\ok)_{\Prism,\ast},\overline \calO_{\Prism})$.
\end{remark}

\subsection{Cohomology of Hodge--Tate crystals: comparison with Sen complex}
Let $\bM$ be a Hodge--Tate crystal. In this subsection, we  investigate its cohomology via the corresponding pair $(M,\phi_M)$ given by Theorem \ref{Thm-HTCrystal}.

\begin{defn}
Let $\bm \in \vect(\okprisast, \baropris)$ resp. $\vect(\okprisast, \baropris[1/p])$. Let $(M,\phi_M)$ be the associated object in  $\End_{\calO_K}^{\ast-\nht}$ resp. $\End_{K}^{\ast-\nht}$.
\begin{enumerate}
\item Call $M$ (or, abusively, the pair $(M,\phi_M)$) the \emph{Sen module} associated to $\bm$ (with respect to $(\gs, (E), \ast)$, or equivalently, with respect to $\pi$).
\item Call the complex  
\[ [M\xrightarrow{\phi_M}M] \]
as the \emph{Sen complex}.
\end{enumerate}
The choice of terminology here comes from their  close relationship with Sen theory, cf. \S \ref{subsecHTsen}.
\end{defn}

 The following is the main theorem in this subsection.

 \begin{theorem}\label{HTcohom}
 Let $\bM$ be an object in $\Vect((\calO_K)_{\Prism,*},\overline \calO_{\Prism})$ resp.  $\Vect((\calO_K)_{\Prism,*},\overline \calO_{\Prism}[1/p])$. Let $(M,\phi_M)$ be the corresponding object in $\End_{\calO_K}^{\ast-\nht}$ resp.  $\End_{K}^{\ast-\nht}$.
 There is a natural quasi-isomorphism
\[[M\xrightarrow{\phi_M}M]\simeq \rR\Gamma((\calO_K)_{\Prism,*},\bM).\]
 \end{theorem}

Before we prove Theorem \ref{HTcohom}, we record a direct corollary.
\begin{cor} \label{corprislogpris}
Let $\bM\in \Vect((\calO_K)_{\Prism},\overline \calO_{\Prism}[\frac{1}{p}])$. We have a quasi-isomorphism
    \[\RGamma((\calO_K)_{\Prism},\bM)\simeq \RGamma((\calO_K)_{\Prism,\log},\bM),\]
    where on the right hand side, we regard $\bm$ as a sheaf on $\okprislog$ via pull-back.
\end{cor}
\begin{proof}
By Theorem \ref{HTcohom}, it reduces to note there is a natural quasi-isomorphism
    \[[M\xrightarrow{\phi_M}M]\simeq[M\xrightarrow{\pi \phi_M}M]\]
    since $\pi$ is invertible in $K$.
\end{proof}

To prove Theorem \ref{HTcohom}, we need to interpret $\rR\Gamma((\calO_K)_{\Prism,*},\bM)$ as some explicit complex.

 \begin{lem}\label{vansihing lemma}
   Let $\bM \in \Vect((\calO_K)_{\Prism,*},\overline \calO_{\Prism})$. Then for any $\frakB\in(\calO_K)_{\Prism,*}$ and any $i\geq 1$, we have
   \[\rH^i(\frakB,\bM) = 0.\]
 \end{lem}
 \begin{proof}
   This follows from the fpqc descent or the same argument in the proof of \cite[Lemma 3.12]{Tia23}.
 \end{proof}
 \begin{cor}[\v Cech-Alexander complex]
   Let $\bM\in \Vect((\calO_K)_{\Prism,*},\overline \calO_{\Prism})$. The prismatic cohomology $\RGamma((\calO_K)_{\Prism,*},\bM)$ can be computed by the \v Cech-Alexander complex
   \begin{equation}\label{Equ-Cech-Alexander complex}
       \bM(\frakS,(E))\xrightarrow{d^0}\bM(\frakS_*^1,(E))\xrightarrow{d^1} \bM(\frakS_*^2,(E))\to\cdots,
   \end{equation}
   where for any $n\geq 0$, $d^n=\sum_{i=0}^{n+1}(-1)^ip_i$ is induced by $p_i:\frakS^n\to\frakS^{n+1}$ defined in Lemma \ref{pd polynomial}.
 \end{cor}
 \begin{proof}
   Since $(\frakS,(E),*)$ is a cover of the final object of the topos ${\rm Shv}((\calO_K)_{\Prism,*})$, the result follows from Lemma \ref{vansihing lemma} combined with the \v Cech-to-derived spectral sequence.
 \end{proof}
\begin{proof}[Proof of Theorem \ref{HTcohom}]
By the identification $\bM(\frakS_*^{\bullet})\cong M\otimes_{\frakS/(E),q_0}\frakS^{\bullet}_*/(E)$ and Proposition \ref{Prop-dRStratification}, the \v Cech-Alexander complex $\bM(\frakS^{\bullet}_*,(E))$ is exactly the complex $\rC(M,\varepsilon_M)$ in Theorem \ref{Thm-AxiomHT}(2). So one can conclude by applying Theorem \ref{Thm-AxiomHT}(2). 
 \end{proof}

  \begin{rmk} 
     In the log-prismatic case for any $\calO_K$ and in the prismatic case for ramified $\calO_K$ (or equivalently, $a\equiv 0\mod \pi$), the comparison in Theorem \ref{HTcohom} can be deduced from the work of \cite{Tia23} on \emph{relative} Hodge--Tate crystals (that is, without using the computations in our axiomatic Theorem \ref{Thm-AxiomHT}(2)).
  This presents a quite interesting phenomenon in this setting: ``unramified" is more difficult than ``ramified". We only give a sketch here, and leave the details to interested readers.
  
Recall we constructed $\rho_M:[M\xrightarrow{\phi_M}M]\to \rC(M,\varepsilon_M)$ in Construction \ref{Construction-Rho}. To prove it is a quasi-isomorphism, it suffices to prove that the induced map \[ \left[M\otimes_{\calO_K}\calO_K\za T\ya\xrightarrow{\phi_M}M\otimes_{\calO_K}\calO_K\za T\ya \right]\to \rC(M\otimes_{\calO_K}\calO_K\za T\ya,\varepsilon_M)\] 
is a quasi-isomorphism.
By derived Nakayama's Lemma, it suffice to prove   
     \begin{equation}
         \label{eqdagger}  \left[ M/\pi[T]\xrightarrow{\phi_M}M/\pi[T] \right] \to \rC(M/\pi[T],\varepsilon_M).
     \end{equation}
     is a quasi-isomorphism.
     As $\phi_M$ is nilpotent modulo $\pi$, we know that $(M/\pi,\phi_M)$ defines a nilpotent Higgs module over $k[T]$.
  Apply  \cite[Theorem 5.12]{Tia23} to the prismatic lifting $(\rW(k)\za T\ya,(p))$ of $k[T]$, this nilpotent Higgs module  induces a Hodge--Tate crystal $\calM$ on the \emph{relative} prismatic site \[(k[T]/(\rW(k),(p)))_{\Prism}.\]
   As $\varepsilon_M \equiv\exp(\phi_MX_1)\mod \pi$ (as $a\equiv 0\mod \pi$), we can check that $\rC(M/\pi,\epsilon_M)$ is the \v Cech-Alexander complex in the sense of \cite[Theorem 5.14]{Tia23}, and hence obtain the quasi-isomorphism \eqref{eqdagger} as desired. 
     \end{rmk}

 \subsection{Cohomology of (log-) prismatic crystals}
 As a direct corollary of Theorem \ref{HTcohom}, we can get the vanishing of higher cohomology of (log-) prismatic crystals by d\'evissage.

 
\begin{thm}\label{non-reduced coh}
 Let $\bm \in \Vect(\okprisast, \opris)$. Then the cohomology
  \[ \RGamma(\okprisast,\bM) \]
   is concentrated in degree $[0,1]$. 
\end{thm}

We begin with two lemmas.
 \begin{lem}
 For any $j>1$ and $n\geq 1$, the cohomology group \[\rH^j((\calO_K)_{\Prism,*},\bM/\calI_{\Prism}^n)=0.\]
 \end{lem}

 \begin{proof}
We consider the short exact sequences of abelian sheaves on $(\calO_K)_{\Prism,*}$
  \[
  0\to \calI_{\Prism}\bM/\calI_{\Prism}^{n+1}\to \bM/\calI_{\Prism}^{n+1}\to \bM/\calI_{\Prism}\to 0
  \]
  for any $n\geq 1$. Then we get the exact triangles
  \[
  \rR\Gamma((\calO_K)_{\Prism,*},\calI\bM/\calI_{\Prism}^{n+1})\to \rR\Gamma((\calO_K)_{\Prism,*},\bM/\calI_{\Prism}^{n+1})\to \rR\Gamma((\calO_K)_{\Prism,*},\bM/\calI_{\Prism}).
  \]
Note that $\calI_{\Prism}\bM/\calI_{\Prism}^{n+1}$ is a crystal of vector bundle over $\calO_{\Prism}/\calI_{\Prism}^n$. Then by Theorem \ref{Thm-HTCrystal} and induction on $n$, we see that $\tau^{\leq 1}\rR\Gamma((\calO_K)_{\Prism,*},\bM/\calI_{\Prism}^n)\simeq \rR\Gamma((\calO_K)_{\Prism,*},\bM/\calI_{\Prism}^n)$ for any $n\geq 1$.

 \end{proof}

 \begin{lem}
 For any crystal $\bM\in \Vect((\calO_K)_{\Prism,*},\calO_{\Prism})$, we have $\bM\simeq\Rlim \bM/\calI_{\Prism}^n$.
 \end{lem}

 \begin{proof}
 Since an inductive limit of faithfully flat maps of prisms is a faithfully flat map of prisms (cf. \cite[Remark 2.4]{BS23}), the topos ${\rm Shv}((\calO_K)_{\Prism,*})$ is replete. Then we have $\lim_n\bM/\calI_{\Prism}^n\simeq\Rlim_n\bM/\calI_{\Prism}^n$ by \cite[Proposition 3.1.10]{BS22}. Then this lemma follows from the fact that $\bM$ is $\calI_{\Prism}$-adically complete.
 \end{proof}

 \begin{proof}[\bf Proof of Theorem \ref{non-reduced coh}]
 Since $\rR\Gamma$ commutes with derived inverse limit, we have
 \[
 \rR\Gamma((\calO_K)_{\Prism,*},\bM)\simeq \Rlim_n \rR\Gamma((\calO_K)_{\Prism,*},\bM/\calI_{\Prism}^n).
 \]
 Now we consider the exact triangle of abelian sheaves on $(\calO_K)_{\Prism,*}$
 \[
 \calI^n_{\Prism}\bM/\calI_{\Prism}^{n+1}\to \bM/\calI_{\Prism}^{n+1}\to \bM/\calI_{\Prism}^n.
 \]
 By taking the cohomology, we see that the natural map \[\rH^1((\calO_K)_{\Prism,*},\bM/\calI_{\Prism}^{n+1})\to \rH^1((\calO_K)_{\Prism,*},\bM/\calI_{\Prism}^n)\]
 is surjective. Then by  \cite[\href{https://stacks.math.columbia.edu/tag/07KY}{Tag 07KY}]{Sta}, we have
 \[\rH^k((\calO_K)_{\Prism,*},\bM)=0\]
 for any $k\geq 2$. This finishes the proof of Theorem \ref{non-reduced coh}.
 \end{proof}

\section{Hodge--Tate crystals on the perfect prismatic site} \label{secHTperf}

In \S \ref{subsecperfHT}, we  classify (rational) Hodge--Tate crystals on the perfect (log-) prismatic site by semi-linear $C$-representations of $G_K$.
In \S \ref{subsec_fullcrystal_rep}, we start with a Hodge--Tate crystals on the full (log-) prismatic site, and compute explicitly the $\gk$-action on the associated $C$-representation.

 \subsection{Hodge--Tate crystals on the perfect site}\label{subsecperfHT}

\begin{prop} \label{same perfect site}
   The forgetful morphism of sites
\[ (\calO_K)_{\Prism,\log}^{\perf}\to (\calO_K)_{\Prism}^{\perf}\]
is an equivalence. 
\end{prop}
\begin{proof}
This is proved in \cite[Proposition 2.18]{MW22log} (cf. Rem. \ref{rem52}). Indeed, this proposition says that given any perfect prism $(A,I) \in (\calO_K)_{\Prism}^{\perf}$, there exists a \emph{unique} way to equip a log structure. We sketch the idea of the proof here.
For uniqueness, suppose we have two objects $(A,I,M_i)\in (\calO_K)_{\Prism,\log}^{\perf}$ for $i\in\{1,2\}$, then one can check that $M_i$ is induced by a pre-log structure $\bN\xrightarrow{} A$ where $1\mapsto [a_i]$  for some $a_i\in A/p$ such that the reduction of $[a_i]$ modulo $I$ belongs to $\pi\cdot (A/I)^{\times}$. Thus, $[a_1]\in [a_2]\cdot A^{\times}$, and the two log structures are isomorphic. 
For existence, one can check there exists a unit $u\in (A/I)^{\times}$ such that $u\pi$ is the image of $[c]$ modulo $I$ for some $c\in A/p$. The log structure $N$ on $A$ induced by $\bN\xrightarrow{}B$ where $1\mapsto [c]$ gives rise to a log prism $(A, I, N)$.
\end{proof}

\begin{remark} \label{rem52}
 Prop. \ref{same perfect site} is the only place in this paper where we cite  \cite{MW22log}. 
 The paper \cite{MW22log} 
is indeed a sequel of the current paper, which treats log-prismatic Hodge--Tate crystals in the relative case. 
Since Prop. \ref{same perfect site} is a standard result (which also works in the general relative case) in log-prismatic site and its proof is rather long, we choose to leave the development in \cite{MW22log}. 
\end{remark}
 
 By Prop. \ref{same perfect site}, in this section, it suffices to treat the site $(\calO_K)^{\perf}_{\Prism}$.
Similar to Lemma \ref{Lem-BKCover}, we have
\begin{lem}\label{Lem-BKCover-perf}
    The prism $(\ainf,(\xi))$ is a cover of the final object in $\Shv((\calO_K)^{\perf}_{\Prism})$.
\end{lem}
\begin{proof} 
    This follows from the proof of \cite[Lemma 3.5]{Wu21} by using \cite[Proposition 7.11(2)]{BS22} instead of \cite[Proposition 7.11(1)]{BS22}.
\end{proof}
Let $(A_{\perf}^{\bullet},(\xi))$ be the \v Cech nerve associated to $(\ainf,(\xi))\in (\calO_K)_{\Prism}^{\perf}$. Then we have the following analogue of Proposition \ref{Prop-dRStratification}:
\begin{lem}\label{crystal is stratification-II}
Evaluations on $(A_{\perf}^{\bullet},(\xi))$ induce a bi-exact equivalence  of categories \[ \Vect((\calO_K)^{\perf}_{\Prism},\baropris[1/p]) \simeq \Strat(\overline \calO_{\Prism}[ {1}/{p}](A_{\perf}^{\bullet},(\xi)))\]
\end{lem}
\begin{proof}
  This follows from   Lemma  \ref{Lem-BKCover-perf}.
\end{proof}

  By Lemma \ref{crystal is stratification-II}, in order to give an explicit description of rational Hodge--Tate crystals on $(\calO_K)_{\Prism}^{\perf}$, we need to describe the cosimplicial ring $\overline \calO_{\Prism}[\frac{1}{p}](A_{\perf}^{\bullet},(\xi))$.
  
\begin{lemma}\label{galois descent}
  There is a canonical isomorphism of cosimplicial rings
  \[\overline \calO_{\Prism}[\frac{1}{p}](A_{\perf}^{\bullet})\cong \rC(G_K^{\bullet},C).\] 
  where for a topological ring $A$, $\rC(G_K^{\bullet},A)$ denotes the cosimplicial ring of continuous functions from $G_K^{\bullet}$ to $A$.
\end{lemma}
\begin{proof}
  This follows from similar arguments used in the proof of \cite[Lemma 5.3]{Wu21}, so we just give an outline here. For any $i\geq 0$, let $S^i$ be the perfectoidization of the $p$-complete tensor product of $(i+1)$-copies of $\calO_{C}$ over $\calO_K$, and then by \cite[Theorem 3.10]{BS22}, $A_{\perf}^{\bullet}/\xi = S^{\bullet}$. It is enough to show $S^{\bullet}[{1}/{p}] = C(G_K^{\bullet},C)$. Let $R^{i}$ be the $p$-complete normalization of $S^i$ in $S^i[{1}/{p}]$. Then $X^i: = \Spa(S^i[1/p],R^i)$ is the initial object in the category of perfectoid spaces over $\Spa(K,\calO_K)$ to which there are $(i+1)$-arrows from $X = X^0 = \Spa(C,\calO_{C})$. So we see that
  \[X^i \cong X\times_K\times\dots\times_KX\]
  is the fibre product of $(i+1)$-copies of $X$ over $K$, which turns out to be
  \[\Spa(C(G_K^i,C),C( G_K^i,\calO_{C})).\]
  This implies $S^{\bullet}[{1}/{p}] = C(G_K^{\bullet},C)$ as desired.
\end{proof}

\begin{rmk}
    In the proof of Proposition \ref{galois descent}, we have shown $ R^{\bullet} = \rC(G_K^{\bullet},\calO_C)$. The natural map $S^{\bullet}\to R^{\bullet}$ then induces a natural morphism of comsimplicial prisms $(A_{\perf}^{\bullet},(\xi))\to(\rC(G_K^{\bullet},\ainf),(\xi))$, by \cite[Theorem 3.10]{BS22}.
\end{rmk}

\begin{thm}\label{rational crystal as representation}
Let $\ast \in \{ \emptyset, \log\}$.
  Evaluation at $(A_{\inf},(\xi), \ast)$ induces a  bi-exact equivalence of categories
\[ \Vect((\calO_K)^{\perf}_{\Prism, \ast},\baropris[1/p])  \simeq \Rep_{G_K}(C)\]
\end{thm}
\begin{proof}
By Prop. \ref{same perfect site}, it suffices to treat the case $\ast=\emptyset$. 
Apply Lemmas \ref{crystal is stratification-II} 
 and \ref{galois descent}, and apply Galois descent.
\end{proof}

 \subsection{Representations from crystals on the full (log-) prismatic site} \label{subsec_fullcrystal_rep}

Let $*\in\{\emptyset,\log\}$. Recall as in \S \ref{sectionHT}, let $a=-E'(\pi)$ for $*=\emptyset$ and $a = -\pi E'(\pi)$ for $* = \log$.

\begin{construction} \label{const_function_gk}
Let $\bm \in \Vect((\calO_K)_{\Prism,*},\overline \calO_{\Prism}[1/p])$. It restricts to a crystal $\bm^\perf$ on the perfect site, and hence corresponds to a representation $W \in \rep_\gk(C)$. We would like to describe 
the Galois action on $W$ using the stratification data corresponding to $\bm$.
Recall using the Breuil-Kisin cosimplicial (log-) prism, one obtains a stratification $(M, \varepsilon)$ and then a pair $(M,\phi_M)$ by Theorem \ref{Thm-HTCrystal} .
In particular, the stratification
 \[\varepsilon: M\otimes_{\ok,p_0} \frakS_{\ast}^{1}[1/p]/E\to M\otimes_{\ok,p_1}\frakS_{\ast}^{1}[1/p]/E\]
is given by sending any $x\in M$ to
\begin{equation}\label{eqsec7vare}
\varepsilon(x) = \sum_{n\geq 0} \left( \prod_{i=0}^{n-1} (\phi_M-ia)(x)\right) \cdot X_1^{[n]}.
\end{equation}
The stratification of  $\bm^{\mathrm{perf}}$ is precisely the base change of $(M, \varepsilon)$ along  \[ \frakS_{\ast}^{\bullet}[1/p]/E   \to    A_{\perf}^{\bullet}[1/p]/E.\] 
To understand this new stratification and hence the $G_K$-action on $W$, it suffices to determine the image of $X_1=X \in\calO_K\{X\}^{\wedge}_{\pd}\cong \frakS_*^1/(E)$ (cf. Proposition \ref{pd polynomial})  under the composite
\begin{equation} \label{eqcom59}
    \frakS_{\ast}^{1}[1/p]/E   \to   A_{\perf}^{1}[1/p]/E \simeq   \rC(G_K, C).
\end{equation}
 \end{construction}
  
 \begin{lemma}\label{image of X}
  Consider $X\in\calO_K\{X\}^{\wedge}_{\pd}\cong\frakS_*^1/(E)$ as an element in $\rC(G_K, C)$ via \eqref  {eqcom59}. For any $g\in G_K$, we have
\[X(g) = c(g)\frac{\pi E'(\pi)}{a}\theta(\frac{\xi}{E([\pi^{\flat}])})(\mu_1-1),\]
where $c(g)\in\bZ_p$ is determined by $g(\pi^{\flat}) = \epsilon^{c(g)}\pi^{\flat}$, and $\mu_1$ is the primitive $p$-th root of unity as in Notation \ref{exampleAinfprism}.
\end{lemma}
\begin{proof}
  We only treat the $* = \emptyset$ case since the $*=\log$ case is similar. 
  In this case, $X$ is the image of $\frac{u_0-u_1}{E(u_0)}$ modulo $E$. By the proof of Lemma \ref{galois descent},  as functions in $\rC(G_K,\ainf)$, $u_0(g) = [\pi^{\flat}]$ and $u_1(g) = g([\pi^{\flat}])$ for any $g\in G_K$. Therefore, we have
  \[X(g) = \frac{[\pi^{\flat}]-[\epsilon]^{c(g)}[\pi^{\flat}]}{E([\pi^{\flat}])} = \frac{1-[\epsilon]^{c(g)}}{[\epsilon]-1}[\pi^{\flat}] \frac{\xi}{E([\pi^{\flat}])}\varphi^{-1}([\epsilon]-1).\]
  Modulo $\xi$, we see $X(g) = -c(g)\pi\theta(\frac{\xi}{E([\pi^{\flat}])})(\mu_1-1)$ as desired.
  Note here $\theta(\frac{\xi}{E([\pi^{\flat}])}) \in \o_C^\times$ is a unit.
\end{proof}

  \begin{prop}\label{Prop-MatrixCocycle}
    Use notations in Construction \ref{const_function_gk}.
The morphism \[(\gs, (E), \ast) \to (\ainf, (\xi), \ast)\] induces an identification
\[ M\otimes_{K}C \xrightarrow{\simeq} W.\]
    For $g\in G_K$, its semi-linear action on $W$ is   determined such that for any $x\in M$,    \begin{equation}\label{Equ-MatrixCocycle}
        g(x) = \sum_{n\geq 0} \left( \left(\prod_{i=0}^{n-1}(\phi_M- ia)(x)\right) \cdot \left( c(g) \frac{\pi E'(\pi)}{a}\theta(\frac{\xi}{E([\pi^{\flat}])})(\mu_1-1)  \right)^{[n]}\right).
    \end{equation}
  \end{prop}
  \begin{proof} One simply plugs Lemma \ref{image of X} into \eqref{eqsec7vare}. 
  \end{proof}

\section{Locally analytic vectors: a review}\label{seclav}
 
 In this section, we review some results from \cite{GP21, Gao23}. We recall the notion of locally analytic vectors, and define some differential operators; these operators will be  \emph{specialized} to the Sen theory setting in \S \ref{seckummersen} and will be crucially used there.
 In \S \ref{subsecreviewOC}, we include a brief review of \emph{overconvergent $(\varphi, \tau)$-modules} which serve as motivation for the discussions in this section and \S \ref{seckummersen}.

\subsection{Locally analytic vectors}

Let us very quickly recall the theory of locally analytic vectors, see \cite[\S 2.1]{BC16} and \cite[\S 2]{Ber16} for more details.  
Recall the multi-index notations: if $\cbf = (c_1, \hdots,c_d)$ and $\kbf = (k_1,\hdots,k_d) \in \mathbb{N}^d$ (here $\mathbb{N}=\mathbb{Z}^{\geq 0}$), then we let $\cbf^\kbf = c_1^{k_1} \cdot \ldots \cdot c_d^{k_d}$. Recall that a $\Qp$-Banach space $W$ is a $\Qp$-vector space with a complete non-Archimedean  norm $\|\cdot\|$ such that $\|aw\|=\|a\|_p\|w\|, \forall a \in \Qp, w \in W$, where $\|a\|_p$ is the   $p$-adic norm on $\Qp$.

\begin{defn}\label{defLAV}
\begin{enumerate}
\item 
Let $G$ be a $p$-adic Lie group, and let $(W, \|\cdot \|)$ be a $\Qp$-Banach representation of $G$.
Let $H$ be an open subgroup of $G$ such that there exist coordinates $c_1,\hdots,c_d : H \to \Zp$ giving rise to an analytic bijection $\cbf : H \to \Zp^d$.
 We say that an element $w \in W$ is an $H$-analytic vector if there exists a sequence $\{w_\kbf\}_{\kbf \in \mathbb{N}^d}$ with $w_\kbf \to 0$ in $W$, such that $$g(w) = \sum_{\kbf \in \mathbb{N}^d} \cbf(g)^\kbf w_\kbf, \quad \forall g \in H.$$
Let $W^{H\dan}$ denote the space of $H$-analytic vectors.

\item $W^{H\dan}$ injects into $\mathcal{C}^{\an}(H, W)$ (the space of analytic functions on $H$ valued in $W$), and we endow it with the induced norm, which we denote as $\|\cdot\|_H$. We have $\|w\|_H=\sup_{\kbf \in \mathbb{N}^d}\|w_{\kbf}\|$, and $W^{H\dan}$ is a Banach space.

\item 
We say that a vector $w \in W$ is \emph{locally analytic} if there exists an open subgroup $H$ as above such that $w \in W^{H\dan}$. Let $W^{G\dla}$ denote the space of such vectors. We have $W^{G\dla} = \cup_{H} W^{H\dan}$ where $H$ runs through  open subgroups of $G$. We can endow $W^{\la}$ with the inductive limit topology, so that $W^{\la}$ is an LB space.

\item We can naturally extend these definitions to the case when $W$ is a Fr\'echet- or LF- representation of $G$, and use $W^{G\dpa}$ to denote the \emph{pro-analytic} vectors, cf. \cite[\S 2]{Ber16}.
\end{enumerate}
\end{defn}

\begin{Notation} \label{nota hatG}
We defined $\hat{G}=\gal(L/K)$ in Notation \ref{notafields}.
Note $\hat{G}=\gal(L/K)$ is a $p$-adic Lie group of dimension 2. We recall the structure of this group in the following.
\begin{enumerate}
\item Recall that:
\begin{itemize}
\item if $K_{\infty} \cap K_{p^\infty}=K$ (always valid when $p>2$, cf. \cite[Lem. 5.1.2]{Liu08}), then $\gal(L/K_{p^\infty})$ and $\gal(L/K_{\infty})$ generate $\hat{G}$;
\item if $K_{\infty} \cap K_{p^\infty} \neq K$, then necessarily $p=2$, and $K_{\infty} \cap K_{p^\infty}=K(\pi_1)$ (cf. \cite[Prop. 4.1.5]{Liu10}) and $\pm i \notin K(\pi_1)$, and hence $\gal(L/K_{p^\infty})$ and $\gal(L/K_{\infty})$   generate an open subgroup  of $\hat{G}$ of index $2$.
\end{itemize}

\item Note that:
\begin{itemize}
\item $\gal(L/K_{p^\infty}) \simeq \Zp$, and let
$\tau \in \gal(L/K_{p^\infty})$ be \emph{the} topological generator such that
\begin{equation} \label{eq1tau}
\begin{cases}
\tau(\pi_i)=\pi_i\mu_i, \forall i \geq 1, &  \text{if }  \Kinfty \cap \Kpinfty=K; \\
\tau(\pi_i)=\pi_i\mu_{i-1}=\pi_i\mu_i^2, \forall i \geq 2, & \text{if }  \Kinfty \cap \Kpinfty=K(\pi_1).
\end{cases}
\end{equation}

\item $\gal(L/K_{\infty})$ ($\subset \gal(K_{p^\infty}/K) \subset \Zp^\times$) is not necessarily pro-cyclic when $p=2$; however, this issue will \emph{never} bother us.
\end{itemize}
\item If we let $\Delta \subset \gal(L/K_{\infty})$ be the torsion subgroup, then $\gal(L/K_{\infty})/\Delta$  is pro-cyclic; choose $\gamma' \in \gal(L/K_{\infty})$ such that its image in $\gal(L/K_{\infty})/\Delta$ is a topological generator.
Let $\hat{G}_n \subset \hat{G}$ be the subgroup topologically generated by $\tau^{p^n}$ and $(\gamma')^{p^n}$. 
\end{enumerate}
\end{Notation}

\begin{Notation} \label{notataula}
We set up some notations with respect to representations of $\hat{G}$.
\begin{enumerate}
\item Given a $\hat{G}$-representation $W$, we use
$$W^{\tau=1}, \quad W^{\gamma=1}$$
to mean $$ W^{\gal(L/K_{p^\infty})=1}, \quad
W^{\gal(L/K_{\infty})=1}.$$
And we use
$$
W^{\tau\dla},   \quad  W^{\gamma\dla} $$
to mean
$$
W^{\gal(L/K_{p^\infty})\dla}, \quad  
W^{\gal(L/K_{\infty})\dla}.  $$

\item  Let
$W^{\tau\dla, \gamma=1}:= W^{\tau\dla} \cap W^{\gamma=1},$
then by \cite[Lem. 3.2.4]{GP21}
$$ W^{\tau\dla, \gamma=1} \subset  W^{\hat{G}\dla}. $$
\end{enumerate}
\end{Notation}

\begin{rem}
Note that we never define $\gamma$ to be an element of $\gal(L/K_\infty)$; although when $p>2$ (or in general, when $\gal(L/K_\infty)$ is pro-cyclic), we could have defined it as a topological generator of $\gal(L/K_\infty)$. In particular, although ``$\gamma=1$" might be slightly ambiguous (but only when $p=2$), we use the notation for brevity.
\end{rem}

\begin{notation} \label{notalieg}
For $g\in \hat{G}$, let $\log g$ denote  the (formally written) series $(-1)\cdot \sum_{k \geq 1} (1-g)^k/k$. Given a $\hat{G}$-locally analytic representation $W$, the following two Lie-algebra operators (acting on $W$) are well defined:
\begin{itemize}
\item  for $g\in \gal(L/\kinfty)$ enough close to 1, one can define $\nabla_\gamma := \frac{\log g}{\log(\chi_p(g))}$;
\item for $n \gg 0$ hence $\tau^{p^n}$ enough close to 1, one can define $\nabla_\tau :=\frac{\log(\tau^{p^n})}{p^n}$.
\end{itemize}
Clearly, these two Lie-algebra operators form a $\qp$-basis of $\Lie(\hat{G}$).
\end{notation}

\subsection{Review of overconvergent $(\varphi, \tau)$-modules}\label{subsecreviewOC}
In this subsection, we  quickly recall some ideas and results in \cite{GL20, GP21} on \emph{overconvergent $(\varphi, \tau)$-modules}, which are certain ``decompleted" modules defined using the Kummer tower $\kinfty$. 
We emphasize that these constructions   \emph{motivate} much of the definitions and computations of Sen theory over the Kummer tower in \S \ref{seckummersen}.
Indeed, one can compare the formulae in Construction \ref{consreviewphitau} with those in Thm. \ref{thm331kummersenmod}.

Recall in (algebraic) $p$-adic Hodge theory, we study representations of the Galois group $G_K$. A key idea   is to first restrict the representations to some subgroups of $G_K$. A most convenient subgroup is $G_\kpinfty$ since it is normal and the quotient group $\Gamma_K$ is a $1$-dimensional \emph{$p$-adic Lie group}, to which many techniques in $p$-adic analysis --- e.g.  locally analytic vectors --- can be applied. We briefly review the following:

\begin{construction}
Let $V\in \rep_\gk(\qp)$, the famous theorem of Cherbonnier-Colmez \cite{CC98}  says that one can attach an \emph{overconvergent $(\varphi, \Gamma)$-module} to it. To save space, we recall the following (loose) formula which constructs the \emph{rigid-overconvergent} $(\varphi, \Gamma)$-module (cf. \S \ref{subsecwtbi} for Berger's well-known ring $\wtb_{\rig}^\dagger$):
\begin{equation}\label{eqkpinftyd}
D_{\rig, \kpinfty}^\dagger(V)  \approx \left((V\otimes_\qp \wtb_{\rig}^\dagger)^{G_\kpinfty}\right)^{\Gamma_K\dla}.
\end{equation}
 (For the interested reader, the correct formula  replaces left hand side by the union  $\cup_n \varphi^{-n}(D_{\rig, \kpinfty}^\dagger(V) )$, cf. \cite[Thm. 8.1]{Ber16};  applying Kedlaya's slope filtration theorem \cite{Ked05} to these objects then recovers Cherbonnier-Colmez's theorem).
 \end{construction}


\begin{construction}\label{consreviewphitau}
Consider the subgroup $G_\kinfty$ of $G_K$ constructed using the Kummer tower.  
The field $\kinfty$ (similar to $\kpinfty$) is still an APF extension (cf. \cite{Win83}), hence Caruso \cite{Car13} shows one can   construct the \'etale $(\varphi, \tau)$-modules (similar to  the \'etale $(\varphi, \Gamma)$-modules) to classify Galois representations.
In two papers joint with Liu and Poyeton respectively, \cite{GL20, GP21}, we prove a conjecture of Caruso, which says that similar to Cherbonnier-Colmez's theorem \cite{CC98}, one can also attach \emph{overconvergent $(\varphi, \tau)$-modules} to $V\in \rep_\gk(\qp)$. The proof in \cite{GP21} uses crucially the idea of locally analytic vectors, indeed, the main formula (again loosely written) is the following.
\begin{equation}\label{eqkinftyd}
D_{\rig, \kinfty}^\dagger(V) \approx \left((V\otimes_\qp \wtb_{\rig}^\dagger)^{G_L}\right)^{\gamma=1,\tau\dla}.
\end{equation}
(Similar comments below \eqref{eqkpinftyd} still apply, cf. \cite{GP21} for full details.)  
A key difference between \eqref{eqkinftyd} and \eqref{eqkpinftyd} is that we \emph{can not} take $G_\kinfty$-invariants first, because $G_\kinfty \subset G_K$ is not normal. This fact complicates the computation for \eqref{eqkinftyd}. Nonetheless, we still manage to show it is a free module of full rank, and furthermore, there is a ``comparison" isomorphism with  overconvergent $(\varphi, \Gamma)$-modules:
\begin{equation}
D_{\rig, \kinfty}^\dagger(V) \otimes \left((  \wtb_{\rig}^\dagger)^{G_L}\right)^{\hat{G}\dla}
\simeq D_{\rig, \kpinfty}^\dagger(V) \otimes \left((  \wtb_{\rig}^\dagger)^{G_L}\right)^{\hat{G}\dla}.
\end{equation} 
 \end{construction}


We shall see many formulae in \S \ref{seckummersen} resemble those discussed above.

\subsection{Locally analytic vectors in $\wtb^I$}\label{subsecwtbi}

We recall some rings and operators from the  study of $(\varphi,\tau)$-modules in \cite{GP21, Gao23}; they will be used in  \S \ref{seckummersen}.

\begin{notation}
We briefly recall the rings $\wt{\mathbf{A}}^{I}$ and  $\wt{\mathbf{B}}^{I}$, see \cite[\S 2]{GP21} for detailed discussions, also see \cite[\S 2]{Gao23} for a faster summary. Recall   we have an embedding $\gs \into \ainf$;  henceforth we simply identify $u$ with its image $[\pi^\flat]$.
\begin{enumerate}
   \item For $n \geq 0$, let $r_n: =(p-1)p^{n-1}$.
Let $\wt{\mathbf{A}}^{[r_\ell, r_k]}$ be the $p$-adic completion of $ \wt{\mathbf{A}}^+ [\frac{p}{u^{ep^\ell}} , \frac{u^{ep^k}}{p}]$, and let 
$$\wt{\mathbf{B}}^{[r_\ell, r_k]}: =\wt{\mathbf{A}}^{[r_\ell, r_k]}[1/p].$$
These spaces are equipped with $p$-adic topology.
  \item  When $I  \subset J$ are two closed intervals as above, then by \cite[Lem. 2.5]{Ber02}, there exists a natural (continuous) embedding
$\wt{\mathbf{B}}^{J}   \hookrightarrow  \wt{\mathbf{B}}^{I}$. Hence we can define the nested intersection
$$\wt{\mathbf{B}}^{[r_\ell, +\infty)}: = \bigcap_{k \geq \ell} \wt{\mathbf{B}}^{[r_\ell, r_k]},$$
and equip it with the natural Fr\'echet  topology. 
  \item  Finally, let 
$$\wt{\mathbf{B}}_{  \rig}^{\dagger}: = \bigcup_{n \geq 0} \wt{\mathbf{B}}^{[r_n, +\infty)},$$
 which is a LF space.

\end{enumerate} 
\end{notation}

\begin{convention}
When $Y$ is a ring with a $G_K$-action, $X \subset \overline{K}$ is a subfield, we use $Y_X$ to denote the $\gal(\overline{K}/X)$-invariants of  $Y$.   Examples include when $Y=\wt{\mathbf{A}}^{I}, \wt{\mathbf{B}}^{I}$ and $X=L, K_\infty$. This ``style of notation" imitates that of \cite{Ber02}, which uses the subscript $\ast_{K}$ to denote $G_{\kpinfty}$-invariants.
\end{convention}

\begin{defn} \label{defnfkt}
(cf. \cite[\S 5.1]{GP21} for full details).
Recall we have an element $[\epsilon] \in \ainf$  in Notation \ref{exampleAinfprism}.
Let $t=\log([\epsilon])  \in \bcrisplus$ be the usual element, where $\bcrisplus$ is the usual crystalline period ring.
Define the element
\[
\lambda :=\prod_{n \geq 0} (\varphi^n(\frac{E(u)}{E(0)}))  \in \bcrisplus.\]
Define
$$ \mathfrak{t} = \frac{t}{p\lambda},$$
then it turns out $\mathfrak{t} \in \ainf$.
  \end{defn}

\begin{lemma} \label{lem b}
\cite[Lem. 5.1.1]{GP21}
There exists some $n=n(\fkt) \geq 0$, such that $\mathfrak{t}, 1/\mathfrak{t} \in 
  \wt{\mathbf{B}}^{[r_n, +\infty)}$. In addition, $\mathfrak{t}, 1/\mathfrak{t} \in
 (\wt{\mathbf{B}}^{[r_n, +\infty)}_{ L})^{\hat{G}\dpa}$.
\end{lemma}

Let us caution that $\mathfrak t$ is an element of (the Banach space) $(\ainf)^{G_L}$, but it is \emph{not} a  {locally analytic vector} inside it; roughly speaking, we need the bigger spaces  $\wt{\mathbf{B}}^{I}_L$  (which contains $t$, the $p$-adic $2\pi i$) to take ``derivatives".

\begin{defn} \label{defndiffwtb}
We define two differential operators on the ring $(\wt{\mathbf{B}}_{  \rig, L}^{\dagger})^{\hat{G}\dpa}$, which are ``normalized" operators of those in Notation \ref{notalieg}. 
\begin{enumerate}
\item (cf. \cite[\S 4]{Gao23}). Define
$$N_\nabla: (\wt{\mathbf{B}}_{  \rig, L}^{\dagger})^{\hat{G}\dpa} \to (\wt{\mathbf{B}}_{  \rig, L}^{\dagger})^{\hat{G}\dpa}$$ 
by setting
\begin{equation}\label{eqnnring}
{N_\nabla:=}
\begin{cases} 
\frac{1}{p\mathfrak{t}}\cdot \nabla_\tau, &  \text{if }  \Kinfty \cap \Kpinfty=K; \\
& \\
\frac{1}{p^2\mathfrak{t}}\cdot \nabla_\tau=\frac{1}{4\mathfrak{t}}\cdot \nabla_\tau, & \text{if }  \Kinfty \cap \Kpinfty=K(\pi_1), \text{ cf. Notation \ref{nota hatG}. }
\end{cases}
\end{equation}
 Note that $1/\mathfrak t$ is in $ (\wt{\mathbf{B}}^\dagger_{\rig, L})^{\hat{G}\dpa}$ by Lem \ref{lem b}, hence  division by $\fkt$ is allowed.
 A convenient and useful fact is that $N_\nabla$ commutes with $\gal(L/\kinfty)$, i.e., $gN_\nabla=N_\nabla g, \forall g\in \gal(L/\kinfty)$, cf. \cite[Eqn. (4.2.5)]{Gao23}.
 
\item (cf.   \cite[5.3.4]{GP21}.)
Define 
$$\partial_{\gamma}: (\wt{\mathbf{B}}_{  \rig, L}^{\dagger})^{\hat{G}\dpa} \to (\wt{\mathbf{B}}_{  \rig, L}^{\dagger})^{\hat{G}\dpa}$$ 
via
$$\partial_{\gamma}:=\frac{1}{\mathfrak t}\nabla_{\gamma}.$$
Since $\gamma(\mathfrak t) = \chi(\gamma) \cdot \mathfrak t$, we have $\nabla_{\gamma}(\mathfrak t) =\mathfrak t $ and hence
\[
\partial_{\gamma}(\mathfrak t)  = 1.
\]
\end{enumerate}
\end{defn}

\begin{rem}\label{remcompaKis}
We mention some remarks about  $N_\nabla$ that are not used in the sequel.
\begin{enumerate}
\item The   $p$ (resp. $p^2$) in the denominator of \eqref{eqnnring} makes  our monodromy operator compatible with earlier theory of Kisin in \cite{Kis06}, but \emph{up to a minus sign}. See also \cite[1.4.6]{Gao23} for general convention of minus signs in that paper.

\item The operator $N_\nabla$ in fact restricts to an operator
$$
N_\nabla: \B_{\rig, \kinfty}^\dagger \to \B_{\rig, \kinfty}^\dagger, 
$$
where $\B_{\rig, \kinfty}^\dagger$ is the Robba ring in the $(\varphi, \tau)$-module setting, cf. \cite[\S 4]{Gao23}.
 
\end{enumerate}
\end{rem}

\section{Sen theory over Kummer tower}
\label{seckummersen}

In this section, we  construct Sen theory over the non-Galois Kummer tower $\kinfty$, via the theory of locally analytic vectors reviewed in \S \ref{seclav}.
We first review classical (cyclotomic) Sen theory. We then compute the set of locally analytic vectors $(\hat{L})^{\hat{G}\dla}$; this is used as a \emph{bridge} transporting the $\kpinfty$-Sen theory to the $\kinfty$-Sen theory.
To be more precise, in \S \ref{subsecKS}, we define the $\kinfty$-Sen module $D_{\Sen, \kinfty}(W)$, and define the $\kinfty$-Sen operator $\frac{1}{\theta(u\lambda')}\cdot N_\nabla$. This operator, when linearly extended over $C$ (in fact, $(\hat L)^{\hat G \dla}$ is enough), becomes the \emph{same} as the classical Sen operator.

 \subsection{Cyclotomic tower and Sen theory}\label{subseccycSen}
 
 We recall the usual category of semi-linear representations.
 \begin{defn}\label{defsemilinrep}
 Suppose $\mathcal G$ is a topological group that acts continuously on a topological ring $R$. We use $\rep_{\mathcal G}(R)$ to denote the category where an object is a finite free $R$-module $M$ (topologized via the topology on $R$) equipped with a continuous and \emph{semi-linear} $\mathcal G$-action in the usual sense that
$$g(rx)=g(r)g(x), \forall g\in \mathcal G, r \in R, x\in M.$$
\end{defn}
 
 Recall $\kpinfty$ is defined in Notation \ref{notafields}.
 Let $\hatkpinfty$ be its $p$-adic completion.


 \begin{theorem}
 Base change functors induce  bi-exact  equivalences of categories (cf. Def. \ref{defsemilinrep})
 \begin{equation*}
 \rep_\gammak(\kpinfty)\simeq \rep_\gammak(\hatkpinfty) \simeq \rep_\gk(C).
 \end{equation*}
 Here, given $W\in \rep_\gk(C)$, the corresponding object in $\rep_\gammak(\hatkpinfty)$ is $W^{G_\kpinfty}$, and the corresponding object in $\rep_\gammak(\kpinfty)$ is 
\begin{equation}\label{senlav}
D_{\Sen, \kpinfty}(W): =(W^{G_\kpinfty})^{\gammak\dla};
\end{equation} 
in addition, the natural map
\[ D_{\Sen, \kpinfty}(W)\otimes_\kpinfty C \to W\]
is an isomorphism.
 \end{theorem}
 \begin{proof}
 This is proved in \cite{Sen80}, except the last formula \eqref{senlav}. In \cite{Sen80}, $D_{\Sen, \kpinfty}(W)$ is recovered as the \emph{``$K$-finite vectors"}; it turns out they coincide with the \emph{locally analytic vectors}, by \cite[Thm. 3.2]{BC16}.
 \end{proof}
 
 \begin{notation}\label{notaSenop}
 Let $W\in \rep_\gk(C)$.
 \begin{enumerate}
 \item Since \eqref{senlav} implies $\Gamma_K$-action on $D_{\Sen, \kpinfty}(W)$ is locally analytic, thus the operator $\nabla_\gamma$ in Notation \ref{notalieg}   induces a operator
\begin{equation}\label{eqsenclassical}
\nabla_\gamma: D_{\Sen, \kpinfty}(W) \to D_{\Sen, \kpinfty}(W).
\end{equation}
This is called the \emph{Sen operator}: it is $\kpinfty$-linear because $\nabla_\gamma$ kills $\kpinfty$.

\item  We can \emph{$C$-linearly extend} $\nabla_\gamma$ to $D_{\Sen, \kpinfty}(W) \otimes_\kpinfty C=W$. That is, we obtain a $C$-linear operator
\begin{equation}\label{eqsenextendtoc}
\nabla_\gamma: W \to W;
\end{equation}
we still call it the \emph{Sen operator}.
 \end{enumerate}
 \end{notation}

\subsection{Locally analytic vectors in $\hat{L}$} \label{seclavL}

Let $\bdrplus$  denote the usual de Rham period ring. Let $\theta: \bdrplus \to C$ be the usual map which extends $\theta: \ainf \to \O_C$.
Recall that as in \cite[\S 2.2]{Ber02}, when $r_n \in I$, there exists a continuous embedding $\iota_n: \wt{\mathbf{B}}^{I} \hookrightarrow \bdrplus$.

\begin{lemma}
Consider the image of $\mathfrak{t}$  via the map $\theta: \ainf\to \oc$, then $0 \neq \theta(\mathfrak t) \in (\hat{L})^{\hat{G}\dla}$.  In addition, $1/\theta(\mathfrak t) \in (\hat{L})^{\hat{G}\dla}$.
\end{lemma}
\begin{proof}
We first check $\theta(\fkt)\neq 0$. Recall $\fkt=\frac{t}{p\lambda}$. Note $\theta(\frac{t}{[\epsilon]-1})=1$ using the expansion $t=\log([\varepsilon])$. Hence it suffices to show $\theta(\frac{[\epsilon]-1}{E(u)})\neq 0$: this holds because both $E$ and $\frac{[\epsilon]-1}{\varphi^{-1}([\epsilon]-1)}$ generate the principal ideal $\ker\theta$.

The proof of analyticity for $\theta(\mathfrak t)$ and $1/\theta(\mathfrak t)$ are the same; alternatively, we can use the fact that  $(\hat{L})^{\hat{G}\dla}$ is a field \cite[Lem. 2.5]{BC16}. We treat $\theta(\mathfrak t)$  in the following.
Choose $n \geq n(\fkt)$ as in Lem. \ref{lem b}  so that $\fkt \in (\wtb^{[r_n, r_n]}_L)^{\hat{G}\dla}$. 
Consider the image of $\fkt$ under the following composite map
\begin{equation}
\label{eqcompftk}
\wtb^{[r_n, r_n]} \xrightarrow{\iota_n} \bdrplus \xrightarrow{\theta} C;
\end{equation}
since both maps are continuous, hence the image is an element in $(\hat{L})^{\hat{G}\dla}$. Unfortunately, the map $\iota_n$   factors as
$$\wtb^{[r_n, r_n]} \xrightarrow{\varphi^{-n}} \wtb^{[r_0, r_0]}\xrightarrow{\iota_0} \bdrplus,$$
hence the   image of $\fkt$ under \eqref{eqcompftk} is only $\theta(\varphi^{-n}(\fkt))$. That is, we obtained
$$\theta(\varphi^{-n}(\fkt)) \in (\hat{L})^{\hat{G}\dla}.$$
Nonetheless, we have $\varphi(\fkt) =\frac{pE(u)}{E(0)} \fkt$.  One can deduce that
$$\fkt = \varphi^{-n}(\fkt) \cdot \prod_{i=1}^n \varphi^{-i}(\frac{pE(u)}{E(0)}),$$
which holds as an equality inside $\ainf$. To see that $\theta(\fkt) \in (\hat{L})^{\hat{G}\dla}$, it then suffices to see that each $\theta(\varphi^{-i}(\frac{pE(u)}{E(0)}) )$ is   locally analytic: but each of these is an element of $\kinfty$ hence is locally analytic (indeed, locally trivial). 
\end{proof}

 \begin{defn}\label{defnablaforL}
 Over the field $(\hat{L})^{\hat{G}\dla}$, we can define two differential operators, which are precisely \emph{``$\theta$-specializations"} of those in Def. \ref{defndiffwtb}.
 \begin{enumerate}
 \item Define $$N_\nabla: (\hat{L})^{\hat{G}\dla} \to (\hat{L})^{\hat{G}\dla}$$ 
by setting
\begin{equation}\label{eqnnring:reduced level}
{N_\nabla:=}
\begin{cases} 
\frac{1}{p\theta(\mathfrak{t})}\cdot \nabla_\tau, &  \text{if }  \Kinfty \cap \Kpinfty=K; \\
& \\
\frac{1}{p^2\theta(\mathfrak{t})}\cdot \nabla_\tau=\frac{1}{4\theta(\mathfrak{t})}\cdot \nabla_\tau, & \text{if }  \Kinfty \cap \Kpinfty=K(\pi_1), \text{ cf. Notation \ref{nota hatG}. }
\end{cases}
\end{equation}
Similar as in Def. \ref{defndiffwtb}, we have $gN_\nabla=N_\nabla g, \forall g\in \gal(L/\kinfty)$.

\item Define 
$$\partial_{\gamma}: (\hat{L})^{\hat{G}\dla} \to (\hat{L})^{\hat{G}\dla}$$ 
via
$$\partial_{\gamma}:=\frac{1}{\theta(\mathfrak t)}\nabla_{\gamma}.$$

\item Both these (normalized) differential operators are well-defined for a $(\hat{L})^{\hat{G}\dla}$-vector space  equipped with semi-linear and locally analytic action  by $\hat{G}$.
 \end{enumerate}
 \end{defn}

The following theorem of Berger-Colmez is crucial for the discussion in the following.

\begin{theorem}\label{thmBC61}
Let $\wt{K}/K$ be a Galois extension contained in $\overline{K}$ whose Galois group is a $p$-adic Lie group with    Lie algebra $\mathfrak g$, and let $\wh{\wt K}$ be the $p$-adic completion.
There exists some $m \in \mathbb N$, a non-zero element $\mathfrak a \in \O_{\wh{\wt K}(\mu_m)} \otimes_\zp \mathfrak g$ such that $\mathfrak a=0$ over $(\wh{\wt K})^{\gal(\wt{K}/K)\dla}$.
\end{theorem}
\begin{proof}
This follows from \cite[Thm. 6.1, Prop. 6.3]{BC16}. Note in \emph{loc. cit.}, one can make $\mathfrak a$ ``primitive" (defined above \cite[Thm. 6.1]{BC16}. In addition, this $\mathfrak a$ can be chosen as a certain ``Sen operator" (see \cite[Prop. 6.3]{BC16}).
\end{proof}

\begin{cor}\label{corkill}
Up to   nonzero scalars, the combination $
\theta(u\lambda') \nabla_\gamma +N_\nabla$ (from Def. \ref{defnablaforL}) is the \emph{unique} non-zero operator in $\hat{L} \otimes_\qp \Lie(\hat G)$ that kills all of $(\hat{L})^{\hat{G}\dla}$. Here  $\lambda$ is defined in Def. \ref{defnfkt}, and $\lambda'$ is the $u$-derivative, that is $\lambda'=\frac{d}{du}(\lambda)$.
\end{cor}
\begin{proof}
The existence of a linear combination  $a\nabla_\gamma +b\nabla_\tau$ that kills $(\hat{L})^{\hat{G}\dla}$ is guaranteed by Thm. \ref{thmBC61}. In addition, neither $\nabla_\gamma$ nor $\nabla_\tau$ alone can kill all of $(\hat{L})^{\hat{G}\dla}$; hence the combination has to be unique up to  non-zero scalars. 
It hence suffices to compute this operator against the element $\theta(\mathfrak t)\in  (\hat{L})^{\hat{G}\dla}$.
Indeed, we can even make the computation inside $(\wt{\mathbf{B}}_{  \rig, L}^{\dagger})^{\hat{G}\dpa}$. 
It is easy to see  $\nabla_\gamma(\mathfrak t)=\mathfrak t$. 
Using the formula in \cite[Lem. 4.1.2]{Gao23} (which holds uniformly even when $\Kinfty \cap \Kpinfty=K(\pi_1)$), one computes
$$ N_\nabla(\mathfrak t) = N_\nabla(\frac{t}{p\lambda}) = \frac{t}{p}\cdot (-\frac{1}{\lambda^2}) N_\nabla(\lambda) = \frac{-t}{p\lambda^2}\cdot   \lambda u\cdot \frac{d}{du}(\lambda) = -\mathfrak t u\lambda'.$$
Hence we can conclude. 
We  remark that it is more convenient to use $N_\nabla$ instead of $\nabla_\tau$ in the formula as it already subsumes the   possible normalization issues  when $p=2$.
\end{proof}


 We now determine the structure of $(\hat{L})^{\hat{G}\dla}$. We first review a description Prop. \ref{loc ana in L} by Berger-Colmez.
  We then obtain an alternative description Prop. \ref{loc ana in L new} which is more convenient for us.

\begin{construction}\label{consalphan}
\begin{enumerate}
\item As in \cite[\S 4.4]{BC16},
consider the 2-dimensional $\Qp$-representation of $G_K$ (associated to our choice of $\{\pi_n\}_{n \geq 0}$) such that $g \mapsto \smat{\chi(g) & c(g) \\ 0 & 1}$ where $\chi$ is the $p$-adic cyclotomic character. Since the co-cycle $c(g)$ becomes trivial over $C_p$, there exists $\alpha \in C_p$ (indeed, $\alpha \in \hat{L}$) such that $c(g) = g(\alpha)\chi(g)-\alpha$.
This implies  $g(\alpha) = \alpha/\chi(g) + c(g)/\chi(g)$ and so $\alpha \in \hat{L}^{\hat{G}\dla}$.

\item 
Now similarly as in the beginning of \cite[\S 4.2]{BC16}, let $\alpha_n \in L$ such that $\|\alpha-\alpha_n\|_p \leq p^{-n}$. Then there exists $r(n) \gg0$ such that if $m \geq r(n)$, then $\|\alpha-\alpha_n\|_{\hat{G}_m}= \|\alpha-\alpha_n\|_p$ and $\alpha-\alpha_n \in \hat{L}^{\hat{G}_m\dan}$ (see Notation \ref{nota hatG} for $\hat{G}_m$, and see Def. \ref{defLAV} for  $\|\cdot \|_{\hat{G}_m}$ ). We can furthermore suppose that $\{r(n)\}_n$ is an increasing sequence.
\end{enumerate}
\end{construction}

\begin{defn}
Let $(H, \|\cdot \|)$ be a $\Qp$-Banach algebra such that $\|\cdot \|$ is sub-multiplicative, and let $W \subset H$ be a $\Qp$-subalgebra. Let $T$ be a variable, and let  $W \dacc{T}_n$ be the vector space consisting of $\sum_{k \geq 0} a_k T^k$ with $a_k \in W$, and $p^{nk} a_k \to 0$ when $k \to +\infty$. For $h \in H$ such that $\|h \|\leq p^{-n}$, denote $W \dacc{h}_n$ the image of the evaluation map $W \dacc{T}_n \to H$ where $T \mapsto h$.
\end{defn}

\begin{prop} 
\label{loc ana in L}
\begin{enumerate}
\item $\hat{L}^{\hat{G}\dla} =\cup_{n \geq 1} K({\mu_{r(n)}, \pi_{r(n)}})\dacc{ \alpha-\alpha_n }_n. $
\item $\hat{L}^{\hat{G}\dla, \nabla_\gamma=0} = L.$
\item $\hat{L}^{\tau\dla, \gamma=1} = {K_{\infty}}.$
\end{enumerate}
\end{prop}
\begin{proof}
Item (1) is \cite[Prop. 4.12]{BC16}, the rest follow easily, cf. \cite[Prop. 3.3.2]{GP21}.
We quickly recall the proof of Item (1) here. Suppose $x\in \hat{L}^{\hat{G}_n\dan}$. For $i \geq 0$, let
$$y_i = \sum_{k \geq 0} (-1)^k (\alpha - \alpha_n)^k \nabla_\tau^{k+i}(x) \binom{k+i}{k},$$
then there exists $m\geq n$ such that $y_i \in \hat{L}^{\hat{G}_m\dan}$, and
\begin{equation} \label{eqxalpha}
x = \sum_{i \geq 0} y_i (\alpha - \alpha_n)^i \in \hat{L}^{\hat{G}_m\dan}
\end{equation}
Note  roughly speaking, \eqref{eqxalpha} is the ``Taylor expansion" of $x$ with respect to the ``variable" $\alpha - \alpha_n$. The equality  \eqref{eqxalpha} holds precisely because 
\begin{equation} \label{eqnablaalpha}
\nabla_\tau (\alpha - \alpha_n) =\nabla_\tau(\alpha)=1.
\end{equation}
Finally, the fact $\nabla_\tau(y_i)=0$ will imply that $y_i\in K(\mu_m, \pi_m)$, concluding (1).
\end{proof}

\begin{prop} \label{loc ana in L new}
 Denote $\beta = \theta(\mathfrak t)$. Apply the same procedure as in Item (2) of Construction \ref{consalphan}, choose the analogous elements $\beta_n \in L$.
Then
$$\hat{L}^{\hat{G}\dla} =\cup_{n \geq 1} K({\mu_{r(n)}, \pi_{r(n)}})\dacc{ \beta-\beta_n }_n. $$
\end{prop}
\begin{proof}
Recall   in Def. \ref{defnablaforL}, we defined $\partial_\gamma: =\frac{\nabla_\gamma}{\beta}$. Since $\nabla_\gamma(\beta)=\beta$, we have 
\begin{equation} \label{eqnablabeta}
\partial_\gamma (\beta - \beta_n) =\partial_\gamma (\beta)=1.
\end{equation}
This is the key analogue of Eqn. \eqref{eqnablaalpha}.
Now similar to the proof in Prop. \ref{loc ana in L}, suppose $x\in \hat{L}^{\hat{G}_n\dan}$, we can define
$$ z_i = \sum_{k \geq 0} (-1)^k (\beta - \beta_n)^k \partial_\gamma ^{k+i}(x) \binom{k+i}{k},$$
then there exists $m\geq n$ such that $z_i \in \hat{L}^{\hat{G}_m\dan}$, and
\begin{equation} \label{eqxbeta}
x = \sum_{i \geq 0} z_i (\beta - \beta_n)^i \in \hat{L}^{\hat{G}_m\dan}
\end{equation}
Finally, $\partial_\gamma(z_i)=0$ implies that $\nabla_\gamma(z_i)=0$ and hence $z_i\in K(\mu_m, \pi_m)$.
\end{proof}


\subsection{Kummer tower and Sen theory} \label{subsecKS}

 \begin{theorem}\label{thm331kummersenmod}
 Given  $W\in \rep_\gk(C)$ of dimension $d$, define
 \begin{equation*}
 D_{\Sen, \kinfty}(W):= (W^{G_L})^{\tau\dla, \gamma=1}.
 \end{equation*}
 Then this is a $\kinfty$-vector space of dimension $d$, and the natural maps induce isomorphisms
 \begin{equation}\label{eqkpk}
 D_{\Sen, \kinfty}(W) \otimes_\kinfty (\hat{L})^{\hat{G}\dla} \simeq (W^{G_L})^{\hat{G}\dla} \simeq  D_{\Sen, \kpinfty}(W) \otimes_\kpinfty (\hat{L})^{\hat{G}\dla}.
 \end{equation}
 \end{theorem}
 \begin{proof}
We study $D_{\Sen, \kinfty}(W)= (W^{G_L})^{\tau\dla, \gamma=1}$ in two steps. In Step 1, we show $(W^{G_L})^{\hat{G}\dla, \nabla_\gamma=0}$ is of dimension $d$ over $L$, this is achieved via a monodromy descent. In Step 2, via an (easy) \'etale descent, we further show the ($\gamma=1$)-invariant $D_{\Sen, \kinfty}(W)$ has dimension $d$. 
 
Step 1 (monodromy descent). We claim the following:
\begin{itemize}
\item Let $M$ be a $\hat{L}^{\hat{G}\dla}$-vector space of dimension $d$ with a semi-linear and locally analytic $\hat{G}$-action. Then the subspace $M^{\nabla_\gamma=0}$ is a   $L$-vector space of dimension $d$ such that
\begin{equation}
M^{\nabla_\gamma=0} \otimes_L \hat{L}^{\hat{G}\dla} =M
\end{equation}
\end{itemize}
 This is a ``$\theta$-specialization" of the argument \cite[Rem. 6.1.7]{GP21}, hence the proof is practically verbatim. 
 To proceed, as in Prop. \ref{loc ana in L new}, we denote $\beta=\theta(\mathfrak{t})$. There, we also made use of the operator
  $\partial_\gamma =\frac{1}{\beta}  \nabla_\gamma$, which is precisely $\theta$-specialization  of the operator (with same notation) in \cite[5.3.4]{GP21}.
 Choose a basis of $M$, and let  $D_\gamma=\Mat(\partial_\gamma)$, then it suffices to show that there exists $H\in \GL_d(\hat{L}^{\hat{G}\dla})$ such that
 \begin{equation}\label{eqdgamma}
 \partial_{\gamma}(H)+D_{\gamma}H = 0
 \end{equation} 
For $k \in \mathbb N$, let $D_k = \Mat(\partial_{\gamma}^k)$. For $n$ large enough, the series given by
$$H = \sum_{k \geq 0}(-1)^kD_k\frac{(\beta-\beta_n)^k}{k!}$$
converges   to the desired  solution of \eqref{eqdgamma}. Here, $\beta-\beta_n$ is used as a ``variable" just as in the proof of Prop. \ref{loc ana in L new}.
  
  Step 2 (etale descent). 
  By \cite[Prop. 3.1.6]{GP21}, we know  
  \begin{equation}\label{eq316first}
  (W^{G_L})^{\hat{G}\dla} =D_{\Sen, \kpinfty}(W) \otimes_\kpinfty (\hat{L})^{\hat{G}\dla}.
  \end{equation}
 Apply Step 1 to the above vector space, and  so 
 $$X:=(W^{G_L})^{\hat{G}\dla, \nabla_\gamma=0}$$
  is a vector space over $L$ of  dimension $d$. In addition, $X$ is stable under $\gal(L/\kinfty)$-action as this action commutes with $\nabla_\gamma$. (Note however $\tau$-action does not commute with $\nabla_\gamma$, not even on the ring level: for example $\tau \nabla_\gamma(u)=0 \neq \nabla_\gamma \tau(u)$.) 
  In summary, $X$ is a $L$-vector space with a $\gal(L/\kinfty)$-action. Note $L=\cup_n K(\pi_n, \mu_n)$ and $\gal(L/\kinfty)$ is topologically finitely generated. Thus, for $n\gg 0$, $X$ descends to some $\gal(L/\kinfty)$-stable  vector space  $X_n$ over $K(\pi_n, \mu_n)$. 
 By Galois descent,   $X_n^{\gamma=1}$ is a $K(\pi_n)$-vector space of dimension $d$, and hence $X_n^{\gamma=1}\otimes_{K(\pi_n)} \kinfty$ is precisely the  desired $D_{\Sen, \kinfty}(W)$. Finally, apply \cite[Prop. 3.1.6]{GP21} again, then we have
 \begin{equation}
 (W^{G_L})^{\hat{G}\dla} =D_{\Sen, \kinfty}(W) \otimes_\kinfty (\hat{L})^{\hat{G}\dla},
 \end{equation}
 which together with \eqref{eq316first} proves \eqref{eqkpk}. 
 \end{proof}

Let $W \in \rep_\gk(C)$. Since $D_{\Sen, \kinfty}(W) $ are locally analytic vectors, we can define (cf. Def. \ref{defnablaforL})
\begin{equation}\label{eq322tau}
N_\nabla: D_{\Sen, \kinfty}(W)  \to (W^{G_L})^{\hat{G}\dla}
\end{equation}

\begin{theorem}\label{thmkummersenop}
Let $W \in \rep_\gk(C)$, then Eqn. \eqref{eq322tau}, after linear scaling,  induces a $\kinfty$-linear operator, which we call the \emph{Sen operator over the Kummer tower}
\begin{equation}\label{eqnnablanorm}
\frac{1}{\theta(u\lambda')}\cdot N_\nabla: D_{\Sen, \kinfty}(W) \to D_{\Sen, \kinfty}(W).
\end{equation}
(We also sometimes use the simplified terminology ``$\kinfty$-Sen operator").
Extend it $C$-linearly to a $C$-linear operator on $ D_{\Sen, \kinfty}(W) \otimes_\kinfty C=W$, and denote it by the same notation:
\begin{equation}
\frac{1}{\theta(u\lambda')}\cdot N_\nabla:  W \to W
\end{equation}
Then this is precisely the (uniquely defined) \emph{Sen operator} in Eqn. \eqref{eqsenextendtoc}.
\end{theorem} 

 \begin{proof}
 Note we have the relation  $gN_\nabla=N_\nabla g, \forall g\in \gal(L/\kinfty)$, hence $N_\nabla$ stabilizes $D_{\Sen, \kinfty}(W)$. 
 Furthermore, $\theta(u\lambda') \in \kinfty$, hence $\frac{1}{\theta(u\lambda')}\cdot N_\nabla$ still stabilizes $D_{\Sen, \kinfty}(W)$. Thus \eqref{eqnnablanorm} is well-defined.
 This is a $\kinfty$-linear operator  because $\nabla_\tau$ hence $\frac{1}{\theta(u\lambda')}\cdot N_\nabla$ kills $\kinfty$.
 
By Cor. \ref{corkill}, 
$\mathfrak{a}= \nabla_\gamma +\frac{1}{\theta(u\lambda')} N_\nabla$ is an $(\hat{L})^{\hat{G}\dla}$-\emph{linear} operator  on both sides of the following:
  \begin{equation*}
 D_{\Sen, \kinfty}(W) \otimes_\kinfty (\hat{L})^{\hat{G}\dla} =D_{\Sen, \kpinfty}(W) \otimes_\kpinfty (\hat{L})^{\hat{G}\dla}.
 \end{equation*}
 On the right hand side,  $N_\nabla$  kills $ D_{\Sen, \kpinfty}(W) $ and hence $\mathfrak{a} =\nabla_\gamma$ is precisely the $(\hat{L})^{\hat{G}\dla}$-linear extension of the Sen operator in Eqn. \eqref{eqsenclassical}. Similarly, on the left hand side, $\nabla_\gamma$ kills $ D_{\Sen, \kinfty}(W)$, and hence $\mathfrak{a}$ is the same as the $(\hat{L})^{\hat{G}\dla}$-linear extension of \eqref{eqnnablanorm}. We can  conclude by further extending $C$-linearly.
 \end{proof}

\begin{cor}
The two operators $\frac{1}{\theta(u\lambda')}\cdot N_\nabla$ in Eqn. \eqref{eqnnablanorm} and $\nabla_\gamma$ in Eqn. \eqref{eqsenclassical} have the same eigenvalues, and have the same semi-simplicity property.
\end{cor}

\begin{defn}
Let $F$ be a field. Write $\mathrm{End}_F$ for the category consisting of $(M, f)$ where $M$ is a
   finite dimensional $F$-vector space and $f: M\to M$ is an $F$-linear endomorphism.
\end{defn}  
  
  \begin{prop} \label{propisomobj}
Let $W_1, W_2 \in \rep_\gk(C)$. 
By abuse of notation, we use $\phi_\mathrm{Sen}$ to denote the Sen operators on $D_{\Sen, \kpinfty}(W_i)$ as well as on $W_i$, cf. Notation \ref{notaSenop}.
Similarly, we use   $\phi_{\kinfty-\mathrm{Sen}}$ to denote the $\kinfty$-Sen operators on $D_{\Sen, \kinfty}(W_i)$, cf. Thm. \ref{thmkummersenop}.
The following are equivalent.
\begin{enumerate}
\item $W_1$ and $W_2$ are isomorphic as objects in $\rep_\gk(C)$.
\item   $(D_{\Sen, \kpinfty}(W_1), \phi_\mathrm{Sen})$ and $(D_{\Sen, \kpinfty}(W_2), \phi_\mathrm{Sen})$ are isomorphic as objects in $\mathrm{End}_\kpinfty$.
\item   $(W_1,  \phi_\mathrm{Sen})$ and $(W_2,  \phi_\mathrm{Sen})$ are isomorphic as objects in $\mathrm{End}_C$.
\item    $(D_{\Sen, \kinfty}(W_1), \phi_{\kinfty-\mathrm{Sen}})$ and $(D_{\Sen, \kinfty}(W_2), \phi_{\kinfty-\mathrm{Sen}})$  are isomorphic as objects in $\mathrm{End}_\kinfty$.
\end{enumerate}
\end{prop}
\begin{proof}
Clearly, (1) implies (2)-(4). (2) obviously implies (3); (4) implies (3) by  Thm. \ref{thmkummersenop}.
The implication from (3) to (1) is proved in \cite[p. 101, Thm. 7]{Sen80}.
\end{proof}

 We record a cohomology comparison theorem in Sen theory.
Recall for $W \in \rep_\gk(C)$, we have the classical Sen module $D_{\Sen, \kpinfty}(W)$ together with Sen operator $\nabla_\gamma$ in Notation \ref{notaSenop}.
We also have $\kinfty$-Sen module $D_{\Sen, \kinfty}(W)$ together with $\kinfty$-Sen operator $\phi_{\sen, \kinfty}=\frac{1}{\theta(u\lambda')}\cdot N_\nabla$.

\begin{theorem}\label{thm-Sen coho Galois coho}
Let $W \in \rep_\gk(C)$. We have  a diagram of  quasi-isomorphisms between (continuous) group cohomologies and Lie algebra cohomologies:
\begin{equation}\label{diagqisom}
\begin{tikzcd}
{\rgamma(G_K, W)} \arrow[rr, "\simeq"] &  & {\rgamma(\hat{G}, W^{G_L})} \arrow[rr, "\simeq"] \arrow[d, "\simeq"]                &  & {\rgamma(\gammak, W^{G_\kpinfty})} \arrow[d, "\simeq"]           \\
                                       &  & {\rgamma(\hat{G}, (W^{G_L})^{\hat{G}\dla})} \arrow[rr, "\simeq"] \arrow[d, "\simeq"] &  & {\rgamma(\gammak,  D_{\sen, \kpinfty}(W))} \arrow[d, "\simeq"] \\
                                       &  & {(\rgamma(\Lie \hat{G}, (W^{G_L})^{\hat{G}\dla}))^{\hat G}} \arrow[rr, "\simeq"]     &  & {(\rgamma(\Lie \gammak, D_{\sen, \kpinfty}(W)))^{\gammak}}
\end{tikzcd}
\end{equation}
In addition, we have quasi-isomorphism
\begin{align}
\label{qisopinfty}\rgamma(G_K, W)\otimes_K \kpinfty & \simeq
[D_{\sen, \kpinfty}(W) \xrightarrow{\nabla_\gamma} D_{\sen, \kpinfty}(W)] \\
\label{qisoL} \rgamma(G_K, W)\otimes_K L &\simeq \rgamma(\Lie \hat{G}, (W^{G_L})^{\hat{G}\dla})\\
\label{qisoinfty}\rgamma(G_K, W)\otimes_K \kinfty &\simeq
[D_{\sen, \kinfty}(W) \xrightarrow{\phi_{\sen, \kinfty}} D_{\sen, \kinfty}(W)]
\end{align}
\end{theorem}
\begin{proof}
All the quasi-isomorphism in  diagram \eqref{diagqisom} (as well as \eqref{qisopinfty}) should be well-known to experts; but we cannot locate an explicit reference. Fortunately, these all follow from some recent general formalism developed by \cite{Poratlav} and \cite{RJRC}.

Note objects in $\rep_\gk(C)$  satisfy the axioms (TS1), (TS2), (TS3), (TS4) in \cite[\S 5.1]{Poratlav}, cf. the examples below  \cite[Cor. 5.4]{Poratlav}; here (TS1)-(TS3) are the usual Tate-Sen axioms introduced by Berger-Colmez, and the extra (TS4) ensures that $C$-representations have no ``\emph{higher locally analytic vectors}", cf. \emph{loc. cit.} for detailed discussions. Thus all the quasi-isomorphisms in diagram \eqref{diagqisom} follow from   \cite[Thm. 1.5, Thm. 1.7]{RJRC} and \cite[Thm. 5.1]{Poratlav} (as already observed in  \cite[Cor. 5.4]{Poratlav}).
(We are also informed by  Rodr\'{\i}guez Camargo that results and arguments in this paragraph are established in much broader generality in \cite{RC22}).

Finally, \eqref{qisopinfty} and \eqref{qisoL} follow from standard Galois descent; and \eqref{qisoinfty} follows by taking $\gal(L/\kinfty)$-invariant of \eqref{qisoL}.
\end{proof}

 \section{Hodge--Tate   crystals and (log-) nearly Hodge--Tate representations}\label{subsecHTsen}

In this section,  we classify rational Hodge--Tate   crystals by  (log-) nearly Hodge--Tate representations, and compare their cohomologies.
The main idea is to show the ``small endomorphism" corresponding to a Hodge--Tate crystal, after some normalization, matches with the $\kinfty$-Sen operator constructed in  \S \ref{seckummersen}.

\begin{notation} \label{notasecgalcoho}
Let $\ast \in \{ \emptyset, \log \}$.
Let $\bm \in \Vect((\calO_K)_{\Prism, \ast},\baropris[1/p]).$
Consider the evaluations
\[ M=\bM((\gs, E, \ast)) \]
In addition, one can associate a small endomorphism $(M, \phi_M)$ by Thm. \ref{Thm-HTCrystal}.
  Consider also
  \[ W= \bM((\ainf, (\xi), \ast)). \]
 We have $W \in \rep_\gk(C)$. By Thm. \ref{thmkummersenop}, we can associate \[ (D_{\Sen, \kinfty}(W), \phi_\kinfty) \in \End_\kinfty.\]
We shall consider $M$ as a subspace of $W$ via the identification \[M\otimes_{K } C \xrightarrow{\simeq} W\]
induced by the morphism
\[ (\gs, E, \ast) \to (\ainf, (\xi), \ast)\]
Finally recall we use
\begin{equation*}
a=
\begin{cases}
  -E'(\pi), &  \text{if } \ast=\emptyset \\
 -\pi E'(\pi), &  \text{if } \ast=\log
\end{cases}
\end{equation*} 
\end{notation}

\begin{thm}\label{thmMWSen}
Use Notation  \ref{notasecgalcoho}.  
\begin{enumerate}
\item Via the identification $M\otimes_{K } C = W$, we have $M \subset D_{\Sen, \kinfty}(W)$; furthermore, this inclusion induces an isomorphism
 \begin{equation}\label{eqMDsen}
  M\otimes_K  \kinfty \simeq D_{\Sen, \kinfty}(W) 
  \end{equation}
  
 \item  Scaling $\phi_M: M \to M$ by $\frac{1}{a}$ and extending $\kinfty$-linearly gives rise to
$$\frac{\phi_M}{a} : M\otimes_K  \kinfty \to M\otimes_K  \kinfty;$$
then this is the \emph{same} (via \eqref{eqMDsen}) as the $\kinfty$-Sen operator 
\begin{equation*} 
\phi_\kinfty=\frac{1}{\theta(u\lambda')}\cdot N_\nabla: D_{\Sen, \kinfty}(W) \to D_{\Sen, \kinfty}(W)
\end{equation*}
 defined in Thm. \ref{thmkummersenop}.
 \end{enumerate}
 \end{thm}
  
\begin{proof}
Item (1).  Let $(\underline{e})$ be a basis of $M$.
The formula \eqref{Equ-MatrixCocycle} in Prop. \ref{Prop-MatrixCocycle} implies that   $\gal(L/\kinfty)$ acts trivially on $M$. In addition, that formula implies that \begin{equation}\label{new82}
        \tau^i(\underline{e}) = (\underline{e})  \sum_{n\geq 0} \left( \left(\prod_{i=0}^{n-1}( A_1-ia)\right) \cdot \left( c(\tau^i) \frac{\pi E'(\pi)}{a} \theta(\frac{\xi}{E})(\mu_1-1)  \right)^{[n]}\right).
    \end{equation}
    where $A_1$ is the matrix of $\phi_M$ with respect to $\underline{e}$, and
$$c(\tau^i)=i, \text{ resp. } 2i, \text{ when } \Kinfty \cap \Kpinfty=K, \text{ resp. }  K(\pi_1).$$
This implies that elements of $M$ are $\tau$-locally analytic. 
In summary, 
$$M \subset W^{\gamma=1, \tau\dla}$$
  Thus \cite[Prop. 3.1.6]{GP21} implies that
  \begin{equation*}
  D_{\Sen, \kinfty} =M\otimes_K  \kinfty.
  \end{equation*}
 
Item (2).  The formula \eqref{new82} can be re-written as
\begin{equation}
    \tau^i(\underline{e})=(\underline{e})\left(1+c(\tau^i) \pi E'(\pi)   \theta(\frac{\xi}{E})(\mu_1-1)\right)^{\frac{A_1}{a}}
\end{equation}
This implies that
\begin{equation*} \label{mwtau}
\nabla_\tau (\underline{e}) = (\underline{e})\frac{\pi E'(\pi)}{a} \theta(\frac{\xi}{E}) ( \mu_1-1) A_1, \quad \text{ resp. }  (\underline{e})2\frac{\pi E'(\pi)}{a} \theta(\frac{\xi}{E}) (\mu_1-1) A_1, 
\end{equation*}  
when $\Kinfty \cap \Kpinfty=K, \text{ resp. }  K(\pi_1)$.
This implies our $\kinfty$-Sen operator acts via
\begin{equation}
\frac{1}{\theta(u\lambda')}\cdot N_\nabla(\underline{e}) 
= (\underline{e})\frac{ \pi E'(\pi) \theta(\frac{\xi}{E}) ( \mu_1-1)A_1  }{pa\theta(u\lambda' \fkt)}
= (\underline{e}) \frac{ \pi E'(\pi) \theta(\frac{\xi}{E})( \mu_1-1) }{p\theta(u\lambda' \fkt)}\cdot \frac{A_1}{a}
\end{equation}
To finish the proof of this theorem, it suffices to check the ``scalar term" is $1$, namely:
\begin{equation*}
\frac{ \pi E'(\pi) \theta(\frac{\xi}{E})( \mu_1-1) }{p\theta(u\lambda' \fkt)}=1
\end{equation*}
Using $\theta(\frac{\xi}{E}) =\theta(\frac{\mu}{\varphi^{-1}(\mu)E} )$ where $\mu=[\epsilon]-1$,
and $\theta(\varphi^{-1}(\mu))=\mu_1-1$, the above identity simplifies as 
\begin{equation*}
\theta( \frac{\mu E'}{p\lambda' \fkt E}  )=1
\end{equation*}
Note $\theta(\frac{t}{\mu})=1$ using $t=\log(\mu+1)$, it suffices to check
\begin{equation}\label{eqwitht}
\theta( \frac{t E'}{p\lambda' \fkt E}  )=1
\end{equation}
 Apply multiplication rule to $\lambda'$, one sees that
 \begin{equation*}
 \theta(\lambda') =\frac{E'(\pi)}{E(0)} \theta(\varphi(\lambda))
 \end{equation*}
Hence \eqref{eqwitht} becomes
\begin{equation*}
\theta( \frac{t }{p\frac{E}{E(0)}  \varphi(\lambda) \fkt}  )=1
\end{equation*}
 The denominator is $p\frac{E}{E(0)}  \varphi(\lambda) \fkt=p\lambda \fkt =t$, thus we can conclude.
\end{proof}

The following theorem (together with Thm. \ref{rational crystal as representation}) completes the proof of all the equivalences in Thm. \ref{thmintroHT}(1).  Recall $\rep_\gk^{\ast-\mathrm{nHT}}(C)$ is defined in Def. \ref{defnnht}.

\begin{theorem} \label{thmnht}
The evaluation functor $\bM \mapsto W:=\bM((\ainf, (\xi), \ast))$  induces a bi-exact equivalence of categories:
$$\Vect((\calO_K)_{\Prism, \ast},\overline \calO_{\Prism}[ {1}/{p}]) \simeq \rep_\gk^{\ast-\mathrm{nHT}}(C).$$
\end{theorem} 
\begin{proof}
Let $\bM \in \Vect((\calO_K)_{\Prism, \ast},\overline \calO_{\Prism}[ {1}/{p}])$, we first show $W(\bM)$ is $\ast$-nearly Hodge--Tate.
Let $(M, \phi_M)$ be the associated Sen module, then Thm. \ref{thmMWSen} implies $\frac{\phi_M}{a}$ is the Sen operator; hence  the condition $\prod_{i=0}^{n-1}(\phi_M-ai)$ converges to zero translates to the fact that the eigenvalues of $\frac{\phi_M}{a}$ ---i.e., the Sen weights of $W(\bm)$--- are in the range as in Def. \ref{defnnht}. This shows $W(\bM)$ is $\ast$-nearly Hodge--Tate, and gives our desired functor
\begin{equation}\label{434functor}
\Vect((\calO_K)_{\Prism, \ast},\overline \calO_{\Prism}[ {1}/{p}]) \to \rep_\gk^{\ast-\mathrm{nHT}}(C).
\end{equation}

We now prove the functor is essentially surjective. 
For \emph{any}  $W \in \rep_\gk(C)$, let $\phi_W$ be the (classical) Sen operator on $D_{\Sen, \kpinfty}(W)$. 
By \cite[p.100, Thm. 5]{Sen80}, there exists a sub-$K$-vector space of full dimension 
\[N \subset D_{\Sen, \kpinfty}(W)\]
 which is stable under $\phi_W$. \footnote{Note this sub-$K$-vector space is in general not unique. For example, consider the trivial $C$-representation whose Sen operator is the zero map, then \emph{any} sub-$K$-vector space is stable under the zero map.}
If $W$ is in addition $\ast$- nearly Hodge--Tate, then the eigenvalues of $\phi_W|_N$ has to satisfy the condition in   Def. \ref{defnnht}, and hence  $(N, f=a \phi_W)$ is an object in $\mathrm{End}_K^{\ast-\nht}$. Via Thm. \ref{Thm-HTCrystal}, this gives rise to a rational Hodge--Tate crystal $\bm$.  Let $W' =\bm((\ainf, (\xi), \ast)) \in \rep_\gk(C)$ be the associated representation, one needs to check 
\[W'\simeq W.\]
 By Thm. \ref{thmMWSen},
$D_{\Sen, \kinfty}(W') =N\otimes_K \kinfty$, and its $\kinfty$-Sen operator is precisely $\frac{f}{a} = \phi_W|_N\otimes 1$; hence in particular, by Thm. \ref{thmkummersenop}, the $C$-linear  Sen   operator on $W'=N\otimes_K C$ is the \emph{same} as that on $W=N\otimes_K C$. Thus $W'$ is isomorphic to $W$ by Prop. \ref{propisomobj}. (We point to Rem. \ref{remconfusion} for a possible confusion in the construction above.)

We now prove the functor is fully faithful.
Let $\bm_1, \bm_2 \in \Vect((\calO_K)_{\Prism, \ast},\overline \calO_{\Prism}[{1}/{p}])$, 
let $(M_1, \phi_1), (M_2, \phi_2)$ be the corresponding Sen modules,
and let $W_1, W_2 \in \rep_\gk^{\ast-\mathrm{nHT}}(C)$ be the corresponding representations.
It suffices to show
$$\mathrm{Morph}((M_1, \phi_1), (M_2, \phi_2)) \to \mathrm{Morph}(W_1, W_2) $$
is a bijection. It is injective because $M_i$ contains a basis of $W_i$; it hence suffices to show both sides have the same $K$-dimension. 
Note that \eqref{434functor} is a tensor functor  and note $f_i$'s are (scaled) Sen operators. 
By considering the $C$-representation $\mathrm{Hom}_C(W_1, W_2)$, it suffices to show the following statement: given $\bm \in \Vect((\calO_K)_{\Prism, \ast},\overline \calO_{\Prism}[1/p])$ with corresponding $(M, \phi)$ and $W$, then we have
\begin{equation*}
M^{\phi=0} =  W^{G_K};
\end{equation*}
but this follows from \cite[p.100, Thm. 6]{Sen80} (which implies the above two spaces are isomorphic after tensoring with $C$).
\end{proof}
 
 
\begin{remark}\label{remconfusion}
 Let us point out a possible confusion in the proof of essential surjectivity above.
 In the proof, we obtained a $C$-linear isomorphim $W \to W'$ (of two vectors spaces \emph{without} considering Galois actions), which is further compatible with the $C$-linear Sen operators. Indeed, it comes from the following composite:
\begin{equation}\label{eqlinsen}
W \simeq N\otimes_K C \simeq \bm((\gs, E))\otimes_K C   \simeq W'
\end{equation}
   where the first isomorphism follows from Sen's original construction, the second comes from construction of $\bm$, and the third comes from definition of $W'$. Note however, the ``identity" map  $N\otimes_K C \simeq \bm((\gs, E))\otimes_K C $ is \emph{not} ``$G_K$-equivariant"! (Hence, in particular, the composite \eqref{eqlinsen} is \emph{not} $G_K$-equivariant!)
   Indeed, the $\tau$-action on $N$ is trivial under classical Sen theory; whereas the $\tau$-action on $\bm((\gs, E))$ is in general not trivial by Prop. \ref{Prop-MatrixCocycle}.
   The powerful aspect of Sen's theorem \cite[p. 101, Thm. 7]{Sen80}, is that the Sen-operator-equivariant map \eqref{eqlinsen} already guarantees existence of a $G_K$-equivariant isomorphism between $W$ and $W'$, although in an \emph{implicit} way!
\end{remark} 
 
 We prove the cohomology comparison in Thm. \ref{thmintroHT}(2).
 
 \begin{theorem} \label{thm-coho-pris-proetale}
 Let $\ast \in \{\emptyset, \log\}$. Let $\bm \in \Vect( \okprisast,\overline \calO_{\Prism}[1/p] )$ be a Hodge--Tate crystal, and let  $W \in \rep_\gk^{\ast-\nht}(C)$ be the associated representation. There exists a natural quasi-isomorphism 
\[   \rg( \okprisast, \bm) \simeq \rg(\gk, W) \]
where the right hand side is   Galois cohomology.
 \end{theorem}
 \begin{proof}
Restriction to the perfect site induces natural morphisms
\[  \rg( \okprisast, \bm) \to \rg(\okprisperfast, \bm) \simeq  \rg(\gk, W) \]
where the second quasi-isomorphism follows from Lem. \ref{galois descent}. It suffices to show the composite is also a quasi-isomorphism.
 Let $(M, \phi_M)$ be the corresponding Sen module. By Thm. \ref{HTcohom}, we have
\[   \rR\Gamma((\calO_K)_{\Prism,*},\bM) \simeq [M\xrightarrow{\phi_M}M].\]
Apply Thm. \ref{thmMWSen} and Thm. \ref{thm-Sen coho Galois coho} respectively, we have
\[ [M\xrightarrow{\phi_M}M]  \otimes_K \kinfty \simeq  [D_{\sen, \kinfty}(W) \xrightarrow{\phi_{\sen, \kinfty}} D_{\sen, \kinfty}(W)] 
\simeq \rgamma(G_K, W)\otimes_K \kinfty.\]
We can conclude.
 \end{proof}
 
\subsection{Comparison with a theorem of Sen} \label{subsec44}
In this subsection, we recall a (rather peculiar) theorem of Sen which also concerns about ``$K$-rationality" of the Sen operator, and make some  comparisons   with our result.

\begin{theorem} \label{thmsenthm10}
\cite[p.110, Thm. 10]{Sen80}
Suppose the residue field $k$ is \emph{algebraically closed}. Then there exists a  (\emph{non-canonical}) equivalence  between $\rep_\gk(C)$ and the category of $K$-vector spaces equipped with a linear operator. This non-canonical equivalence, which depend on choices for each simple object in $\rep_\gk(C)$, is \emph{never} compatible with tensor products.
\end{theorem}


\begin{construction}\label{constSen}
We give a very brief sketch of the rather technical proof of Thm. \ref{thmsenthm10}, pointing out where \emph{non-canonicity} comes from.
\begin{enumerate}
\item 
Recall, as we already used in the proof of our Thm. \ref{thmnht}, Sen shows in  \cite[p.100, Thm. 5]{Sen80} that for any $C$-representation, the Sen operator is   ``$K$-rational": that is, it is stable on a sub-$K$-vector space.

\item 
 Conversely, when the residue field $k$ is algebraically closed, \cite[p.104, Thm. 9]{Sen80} shows that any $K$-matrix can be recovered as a Sen operator for some $C$-representation.
Note however, as stated in  \cite[p.105, Thm. 9']{Sen80}, this only constructs a   \emph{set bijection} between   $H^1(G_K, \GL_n(C))$ and the set of similarity classes of $n\times n$ matrices over $K$. One still needs to build some \emph{functoriality} into this bijection. 

\item 
Sen then \emph{chooses}, for each \emph{simple} $C$-representation $U$, a $K$-vector space $U_0$ stable under Sen operator; this builds a functor for \emph{semi-simple} objects. To get the complete functor, one first map a general $C$-representation $W$ to its semi-simple part $W_S$, then one notes $(W_S)_0$ is still stable under the Sen operator of $W$ (not just that of $W_S$!), since Sen operator is Galois equivariant.
\end{enumerate}
\end{construction}

\begin{remark}\label{remfon214}
In the remark following \cite[Thm. 2.14]{Fon04}, Fontaine mentions that Thm. \ref{thmsenthm10} can also be obtained using his \emph{classification} of $C$-representations: this  also points to the  non-canonicity of Thm. \ref{thmsenthm10}.
\end{remark}

\begin{remark}
We give some yet further remarks about the non-canonicity of  Thm. \ref{thmsenthm10}, and compare it with our Thm. \ref{thmnht}.
In $p$-adic Hodge theory, many typical  theorems say that there is an equivalence of categories:
\[\mathrm{MOD} \simeq \mathrm{REP}\]
where $\mathrm{MOD}$ is certain ``module category" and $\mathrm{REP}$ is certain ``representation category". We invite the readers to consider the examples such as results of Bhatt-Scholze \cite{BS23}, of Cherbonnier-Colmez \cite{CC98}, or Thm. \ref{thmnht}.
In proofs of these results, it is usually easy to construct a  \emph{functor} from  $\mathrm{MOD}$ to $\mathrm{REP}$ (sometimes perhaps initially to a bigger representation category, such as in our  Thm. \ref{thmnht}). The most difficult part almost always is to show the functor is essentially surjective (or, to determine its essential image).

What happens with Sen's result Thm. \ref{thmsenthm10}, is that there is not even an obvious \emph{map} from $\mathrm{MOD}$ to $\mathrm{REP}$; 
Sen can still construct a such map, by the \emph{set bijection} mentioned in  \ref{constSen}(2): but then this is hopeless to be a \emph{functor}. This forces Sen to go the other way around, by constructing a \emph{functor} from  $\mathrm{REP}$ to $\mathrm{MOD}$, which then unfortunately relies on many choices, and cannot be canonical. 
We regard this strong contrast with our Thm. \ref{thmnht} as another hint (in addition to  Rem. \ref{remintro115}) that the (log)- nearly Hodge--Tate representations deserve to be further studied.
\end{remark}

\section{Cohomology vanishing: proof of Proposition \ref{Prop-Cohomology}(3)}\label{higher vanishing}

In this section, we prove Proposition \ref{Prop-Cohomology}(3). The main strategy is explained in Construction \ref{constr3stepssec9}; the two steps there are then carried out in the two subsections.

We start by introducing some notations.

 
\begin{notation} \label{nota9index}
   First we set up some notation on indices.
Let $s\geq 1$ (throughout this section).  For  $I=(i_1,\dots,i_s)\in\bN^s$, let $|I|:=i_1+i_2+\cdots+i_s.$ Let $E_q = (0,\dots,0,1,0,\dots,0)\in\bN^s$ with $1$   at the $q$-th component. Thus,
    \[I = i_1E_1+\cdots+i_sE_s.\]
For   $I=(i_1,\dots,i_s)\in\bN^s$ and   $J=(j_1,\dots,j_t)\in\bN^t$, we define $$(I,J):=(i_1,\dots,i_s,j_1,\dots,j_t)\in\bN^{s+t}.$$    
\end{notation}

\begin{notation}    
 From now on, we use notations in \S\ref{axiomcoho}. Consider two elements 
 \[f := \sum_{I=(i_1,\dots,i_s)\in\bN^s}m_I\underline X^{[I]}\in \rC^s\] and 
 \[g:=\sum_{J=(j_1,\dots,j_{s+1})\in\bN^{s+1}}n_J\underline X^{[J]} \in \rC^{s+1}.\] 
 If $\rd^s(f) = g$, then by
   (\ref{Equ-FaceMod}), we can write the equation  as follows
  \begin{equation}\label{Equ-DifferentialMod=0}
      \begin{split}
           \sum_{J\in\bN^{s+1}}n_J\underline X^{[J]}
          =&\sum_{I\in\bN^s}(1+aX_1)^{\frac{\phi_M}{a}-|I|}(m_I)(X_2-X_1)^{[i_1]}\cdots(X_{s+1}-X_1)^{[i_s]}\\
          &+\sum_{l=1}^{s+1}(-1)^l\sum_{I\in\bN^s}m_IX_1^{[i_1]}\dots X_{l-1}^{[i_{l-1}]}X_{l+1}^{[i_{l}]}\cdots X_{s+1}^{[i_{s}]}.
      \end{split}
  \end{equation}
\end{notation}

\begin{construction} \label{constr3stepssec9}
    Now given $g \in \Ker (\rd^{s+1})$, to prove Proposition \ref{Prop-Cohomology}(3),  we need to find an $f$ such that $\rd^s(f) = g$. The  proof will be roughly divided into the following  steps. 
\begin{enumerate}[leftmargin=0cm]
    \item[] \textbf{Step 1}:     In \S \ref{subsecstep1} (precisely, Prop. \ref{constructio of Lambda}),   we prove that if $g=\sum_{K}n_K\underline X^{[J]}\in \Ker(d^{s+1})$, then the coefficients $n_K$ of $g$ are determined by those whose indices are in $\Lambda^{s+1}$, a specific subset of $\bN^{s+1}$ (cf. Proposition \ref{constructio of Lambda}). To construct the subset $\Lambda^{s+1}$, we need to investigate the coefficients in the Equation (\ref{Equ-DifferentialMod=0}). The whole process will be divided into several parts with respect to the leading value $k_1$ of $K=(k_1,k_2,\cdots,k_{s+1})\in \bN^{s+1}$.

 (We warn that in \S \ref{subsecstep1}, we will prove much more than claimed in Step 1 above, cf. Construction \ref{consstep1}. We further warn there will be a (necessary) switch of superscripts between $s$ and $s+1$ there.)
    
    \item[] \textbf{Step 2}: cf. \S \ref{subsecstep2}. Based on the subset $\Lambda^{s+1}$, we will give a concrete construction of a candidate $f$.
   It remains to verify $\rd^s(f)=g$. Note both $\rd^s(f), g \in \mathrm{Ker}(\rd^{s+1})$, thus by Step 1, it suffices to compare their coefficients over indices in $\Lambda^{s+1}$.
    
\end{enumerate}
\end{construction}

We remark that although the construction of $f$ and the verification of $d^s(f)=g$ look quite complicated, it is indeed a direct generalisation of the pattern we found in the case of $s=1$. (Namely, we encourage the interested reader to try the computation for $s=1$.)

\subsection{Step 1: construction of $\Lambda^{s}$ (and $\Lambda^{s+1}$)  } \label{subsecstep1}

\begin{construction} \label{consstep1}
    \begin{enumerate}
        \item We shall construct some subset $\Lambda^s\subset\bN^s$ by solving the equation $\rd^s(f) = g$ (\ref{Equ-DifferentialMod=0}), and show that the coefficients of $f$---i.e., the $\{m_K\}_{K\in \bN^s}$'s in (\ref{Equ-DifferentialMod=0})---are uniquely determined by $\{m_I\}_{I\in\Lambda^s}$ and $\{n_J\}_{J\in \bN^{s+1}}$. We shall construct $\Lambda^s$ as a union of 3 parts, each part enough to \emph{determine some $m_K$}, depending on the  leading value $k_1$ of $K=(k_1,k_2,\cdots,k_{s})$; see Def. \ref{defnlambdat} for a summary of these parts.

        \item Note the above item achieves more than as stated in \textbf{Step 1} in Construction \ref{constr3stepssec9} above; cf. next item.

        \item Indeed, by letting $g=0$, then $\rd^s(f)= 0$ implies the coefficients of $f$ are uniquely determined by $\{m_I\}_{I\in\Lambda^s}$. (Note there is shift from $s+1$ to $s$ in this discussion; cf Prop. \ref{constructio of Lambda} and caution above Def. \ref{defnlambdat}.) 
    \end{enumerate}
\end{construction}

\begin{observation}\label{observation} The basic idea to achieve   Construction \ref{consstep1} is simple.     
    Let $$F = \sum_{I\in\bN^s}m_I\underline X^{[I]}, \quad G = \sum_{I\in\bN^s}n_I\underline X^{[I]}.$$
    \begin{enumerate}
        \item If we want to compare the coefficients $m_I$ with $n_I$ for $I = (i_1,i_2,\dots,i_s)$ such that $i_{t_1}=\cdots=i_{t_j}=0$, then it suffices to let $X_{i_{t_1}}=\cdots=X_{i_{t_j}}= 0$, and consider the coefficients of $Y:=\prod_{1\leq h\leq s, h\neq t_1,\dots,t_j}X_h^{[i_h]}$. 

        \item Similarly, if we want to compare $m_I$ with $n_I$ for $ I = (1,*)$ with the leading term being $1$, it suffices to consider the coefficients (including $X_2,\dots,X_s$ as well) of $X_1$ in the expansions.
    \end{enumerate}
   For example, in Lemma \ref{Lem-Set1}, we simply let $X_1 = 0$ in (\ref{Equ-DifferentialMod=0}); for another example, in Lemma \ref{Lem-Set2}, we compare coefficients of $X_1$. 
\end{observation}

We embark on the mission sketched in Construction \ref{consstep1}(1).
The next lemma deals with the case where the index $K = (k_1,\dots,k_s)$ satisfies $k_1 = 0$. So we  let $X_1=0$ in (\ref{Equ-DifferentialMod=0}) and then compare   coefficients on both sides (cf. Observation \ref{observation}).
 
  \begin{lem}\label{Lem-Set1}
   Consider the following subset $\Lambda_1^s \subset \bN^s$ of indices:
   \begin{itemize}
       \item If $s \in \{1,2\}$, then $\Lambda_1^s = \{(0,\dots,0)\}$.

       \item If $s\geq 3$, then  
       $\Lambda^s_1:=\{(0,\dots,0,i_{2j+1},\dots,i_s)\in\bN^s\mid j\geq 1,i_{2j+1}\geq 1\}$.
   \end{itemize}
Let $K=(0,k_2,\dots,k_s)\in\bN^s$ be an index with the first term being $0$.
There are integers $$\{z^{(s)}_{K,I}\}_{I\in \Lambda^s_1} \text{ and } \epsilon^{(s)}_{K,K},  \text{ all in } \{\pm 1,0\}$$ 
depending only on $s $ and $ K$, such that if we have $\rd^s(f) = g$, then
\[m_K=\sum_{I\in\Lambda^s_1}z^{(s)}_{K,I}m_I+\epsilon^{(s)}_{K,K}n_{(0,K)}\] 
(Thus in particular, these $m_K$'s are determined by $m_I$'s with $I \in \Lambda_1^s$ and $n_{(0, K)}$).
Furthermore, for any $I\in\Lambda^s_1$, we have
      \begin{enumerate}
          \item[(1)] if $K\in \Lambda^s_1$, then $\epsilon^{(s)}_{K,K}=0$, and $z^{(s)}_{K,I} = 1$ for $K=I$ while $z^{(s)}_{K,I} = 0$ for $K\neq I$, and that
          \item[(2)] if $K = (0,\dots,0,k_{2i},\dots,k_s)$ for some $i\geq 1$ such that $k_{2i}\cdots k_s\neq 0$, then $\epsilon^{(s)}_{K,K}=1$ and $z^{(s)}_{K,I} = 0$, and
          \item[(3)] if $K = (0,\dots,0,k_{2i},\dots,k_s)$ for some $i\geq 1$ such that $k_{2i}\neq 0$ and that $k_{2i}\cdots k_s=0$, then $\epsilon^{(s)}_{K,K}=1$, and that
          
          \item[(4)] if $|K|\neq |I|$, then $z_{K,I}^{(s)} = 0$.
      \end{enumerate}
    \end{lem}
    \begin{proof}
        As results hold trivially for $s=1$, now assume $s\geq 2$.
        
        If $K\in \Lambda_1^s$, then we define $z_{K,I}^{(s)}$ and $\epsilon_{K,K}^{(s)}$ as stated in Item (1).

        Now, we are going to construct $z_{K,I}^{(s)}$ and $\epsilon_{K,K}^{(s)}$ for $K\notin \Lambda_1^s$. In this case, we must have $K=(0,\dots,0,k_{2i},\dots,k_s)$ for some $i\geq 1$ with $k_{2i}\neq 0$.  
        
        By letting $X_1 = 0$ in (\ref{Equ-DifferentialMod=0}), we obtain that
        \begin{equation}\label{Equ-Set1-I}
            \begin{split}
                \sum_{I=(i_1,\dots,i_s)\in \bN^s}n_{(0,I)}X_2^{[i_1]}\cdots X_{s+1}^{[i_s]}=&\sum_{l=2}^{s+1}(-1)^l\sum_{I=(0,i_2,\dots,i_s)\in\bN^s}m_IX_2^{[i_2]}\dots X_{l-1}^{[i_{l-1}]}X_{l+1}^{[i_{l}]}\cdots X_{s+1}^{[i_{s}]}\\
                =&
                \sum_{l=2}^{s+1}(-1)^l\sum_{(0,i_1,\dots i_{l-2},i_l\dots,i_s)\in\bN^s}m_{(0,i_1,\dots i_{l-2},i_l\dots,i_s)}X_2^{[i_1]}\cdots X_{l-1}^{[i_{l-2}]}X_{l+1}^{[i_l]}\cdots X_{s+1}^{[i_s]}.
            \end{split}
        \end{equation}
        For any $2\leq t_1<t_2<\dots<t_r\leq s+1$ (with $1\leq r\leq s$), letting $X_{t_1} = \dots =X_{t_r} = 0$ in (\ref{Equ-Set1-I}) and considering the coefficients of $\prod_{2\leq t\leq s+1,t\neq t_1,\dots,t_r}X_{t}^{i_{t-1}}$ with all $i_{t-1}\geq 1$ (so we only need to consider $l\in\{t_1,\dots,t_r\}$ in the above summation for $l$), we have 
        \begin{equation}\label{Equ-Set1-II}
            \sum_{k=1}^r(-1)^{t_k}m_{(\sum_{2\leq q\leq t_{k}-1,q\neq t_1,\dots,t_{k-1}}i_{q-1}E_q+\sum_{t_k+1\leq q'\leq s+1,q'\neq t_{k+1},\dots,t_r}i_{q'-1}E_{q'-1})} = n_{(\sum_{2\leq h\leq s+1,h\neq t_1,\dots,t_r}i_{h-1}E_h)}.
        \end{equation}

    Now, assume $k_{2i}\cdots k_s\neq 0$.
    In (\ref{Equ-Set1-II}), let $r=2i-1$ with $i\leq 1$ and let $t_1=2,t_2=3,\dots,t_{2i-1}=2i\leq s+1$, and $i_{h-1}=k_{h-1}$ for all $2i+1\leq h\leq s+1$. Then we get 
    \[n_{(0,0,\dots,0,k_{2i},\dots,k_s)} = \sum_{k=2}^{2i}(-1)^km_{(\sum_{2i+1\leq q'\leq s+1}k_{q'-1}E_{q'-1})} = m_{(0,\dots,0,k_{2i},\dots,k_s)},\]
    as claimed in Item (2).

    Assume $k_{2i}\cdots k_s = 0$. In the spirit of Observation \ref{observation}, in the equation (\ref{Equ-Set1-II}) above, we need to choose $r\geq 2i$ with $i\geq 1$ and $t_1 = 2,\dots,t_{2i-1}=2i, 2i+2\leq t_{2i}<\dots<t_r\leq s+1$ such that for any $2i+2\leq h\leq s$, $k_h=0$ if and only if $h+1\in\{t_{2i},\dots,t_r\}$. Then we have 
    \[K=k_{2i}E_{2i}+\sum_{2i+2\leq q'\leq s+1,q'\neq t_{2i},\dots,t_r}k_{q'-1}E_{q'-1}.\]
    Let $i_{h-1}=k_{h-1}$ for all $h\in\{2i+1,\dots,s+1\}\setminus\{t_{2i},\dots,t_r\}$ in (\ref{Equ-Set1-II}), and then we get
    \[\begin{split}
        &n_{(k_{2i}E_{2i+1}+\sum_{2i+2\leq h\leq s+1,h\neq t_{2i},\dots,t_r}k_{h-1}E_h)} \\= & \sum_{l=2}^{2i}(-1)^lm_{(k_{2i}E_{2i}+\sum_{2i+2\leq q'\leq s+1,q'\neq t_{2i},\dots,t_r}k_{q'-1}E_{q'-1})}\\
        &+\sum_{h=2i}^{r}(-1)^{t_h}m_{(\sum_{2i+1\leq q\leq t_h-1,q\neq t_{2i},\dots,t_{h-1}}k_{q-1}E_q+\sum_{t_h+1\leq q'\leq s+1,q'\neq t_{h+1},\dots,t_r}k_{q'-1}E_{q'-1})}\\
        = & m_{(k_{2i}E_{2i}+\sum_{2i+2\leq q'\leq s+1,q'\neq t_{2i},\dots,t_r}k_{q'-1}E_{q'-1})}\\
        &+\sum_{h=2i}^{r}(-1)^{t_h}m_{(k_{2i}E_{2i+1}+\sum_{2i+2\leq q\leq t_h-1,q\neq t_{2i},\dots,t_{h-1}}k_{q-1}E_q+\sum_{t_h+1\leq q'\leq s+1,q'\neq t_{h+1},\dots,t_r}k_{q'-1}E_{q'-1})},
    \end{split}\]
    which implies that
    \begin{equation}\label{Equ-Set1-III}
    \begin{split}
        &m_{(k_{2i}E_{2i}+\sum_{2i+2\leq q'\leq s+1,q'\neq t_{2i},\dots,t_r}k_{q'-1}E_{q'-1})} \\
        =& \sum_{h=2i}^{r}(-1)^{t_h+1}m_{(k_{2i}E_{2i+1}+\sum_{2i+2\leq q\leq t_h-1,q\neq t_{2i},\dots,t_{h-1}}k_{q-1}E_q+\sum_{t_h+1\leq q'\leq s+1,q'\neq t_{h+1},\dots,t_r}k_{q'-1}E_{q'-1})}\\
        &+n_{(k_{2i}E_{2i+1}+\sum_{2i+2\leq q'\leq s+1,q'\neq t_{2i},\dots,t_r}k_{q'-1}E_{q'})}.
    \end{split}
    \end{equation}
    For any $2i\leq h\leq r$, define
    \[J_h:=k_{2i}E_{2i+1}+\sum_{2i+2\leq q\leq t_h-1,q\neq t_{2i},\dots,t_{h-1}}k_{q-1}E_q+\sum_{t_h+1\leq q'\leq s+1,q'\neq t_{h+1},\dots,t_r}k_{q'-1}E_{q'-1},\]
    and then $J_h\in\Lambda_1^s$ (as $k_{2i}\neq 0$) satisfying
    \[
    |K|
    =k_{2i}+\sum_{2i+1\leq q'\leq s,q'+1\neq t_{2i},\dots,t_r}k_{q'}
    =|J_h|.\]
    Using this, we see that (\ref{Equ-Set1-III}) amounts to
    \begin{equation}\label{Equ-Set1-IV}
        m_{K}=\sum_{h=2i}^r(-1)^{t_h+1}m_{J_h}+n_{(0,K)}.
    \end{equation}
    So one can construct $z_{K,I}^{(s)}$'s in the obvious way by applying (\ref{Equ-Set1-IV}) and obtain the Item (3).

    Finally, Item (4) follows from the above constructions directly.
 \end{proof}

 Now we consider the case where the index $K = (k_1,\dots,k_s)$ of $m_K$ satisfies $k_1 \geq 2$. It turns out that the $m_{k_1,\ast}$'s with $k_1\geq 2$ are uniquely determined by $m_{1,*}$'s and $n_{1,k_1-1,*}$. We will show this by comparing the coefficient of $X_1$ on both sides of (\ref{Equ-DifferentialMod=0}).

 \begin{lem}\label{Lem-Set2}
For any $s\geq 1$, consider the subset of indices with leading term being $1$, i.e.,  $$\Lambda_2^s=\{(1,i_2,\dots,i_s)\mid i_2,\dots,i_s\geq 0\}\subset\bN^s.$$
    Let $K=(k_1,\dots,k_s)\in\bN^s$ with $k_1\geq 2$.
For each $I\in\Lambda_2^{(s)}$, there exists
\[  z_{K,I}^{(s)} \in \{\pm 1,0\} \]
depending only on $s$ and $ K$, such that if $d^s(f)=g$, then 
\[m_{K}=\sum_{I\in \Lambda_2^s}z_{K,I}^{(s)}m_I+(\phi_M-a(|K|-1))m_{K-E_1}-\sum_{l=2}^sm_{K-E_1+E_l}-n_{(1,K-E_1)}.\]
    Furthermore, we have $z_{K,I}^{(s)}=0$ in the case either $k_2\cdots k_s\neq 0$ or $|K|\neq |I|$.
 \end{lem}
 \begin{proof}
     In (\ref{Equ-DifferentialMod=0}), comparing the coefficient of $X_1$, we see that
     \begin{equation}\label{Equ-Set2-I}
         \begin{split}
             &\sum_{I=(i_1,\dots,i_s)\in \bN^s}n_{(1,I)}X_2^{[i_1]}\cdots X_{s+1}^{[i_s]}\\
             =&\sum_{I=(i_1,\dots,i_s)\in\bN^s}((\phi_M-a|I|)(m_I)-\sum_{l=1}^sm_{I+E_l})X_2^{[i_1]}\cdots X_{s+1}^{[i_s]}+\sum_{l=2}^{s+1}(-1)^l\sum_{I=(1,i_2,\dots,i_s)}m_IX_2^{[i_2]}\cdots X_{l-1}^{[i_{l-1}]}X_{l+1}^{[i_{l}]}\cdots X_{s+1}^{[i_s]}.
         \end{split}
     \end{equation}
     
     Let $K=(k_1,k_2,\dots,k_s)$ satisfy $k_1\geq 2$.
     
     Assume $k_2\cdots k_s\neq 0$.
     In (\ref{Equ-Set2-I}), comparing the coefficients of $X_2^{[i_1]}\cdots X_{s+1}^{[i_{s}]}$ with $(i_1,\dots,i_s)=(k_1-1,\dots,k_s)$ (namely, $I=K-E_1$), we have 
     \[n_{(1,K-E_1)}=(\phi_M-a|K-E_1|)(m_{K-E_1})- \sum_{l=1}^sm_{K-E_1+E_l}.\]
     In other words, we conclude that
     \begin{equation}\label{Equ-Set2-II}
         m_{K} = (\phi_M-a(|K|-1))(m_{K-E_1})-\sum_{l=2}^sm_{K-E_1+E_l}-n_{(1,K-E_1)}.
     \end{equation}
     Now, we can define $z_{K,I}^{(s)}$'s in the obvious way by applying (\ref{Equ-Set2-II}), and conclude the ``furthermore'' part in this case.

     Now, assume $k_2\cdots k_s = 0$. Then there is an $r\geq 1$ and $3\leq t_1<\cdots<t_r\leq s+1$ such that $k_h = 0$ if and only if $h+1\in\{t_1,\dots,t_r\}$. In particular, we have
     \[K=E_1+\sum_{2\leq t\leq s+1,t\neq t_1,\dots,t_r}j_{t-1}E_{t-1}\]
     where $j_1=k_1-1$ and $j_{t-1} = k_{t-1}$ for any $t\notin\{2,t_1,\dots,t_r\}$. Thus, $j_{t-1}\geq 1$ for all $t$. 
     
     Letting $X_{t_1} = \cdots =X_{t_r} = 0$ in (\ref{Equ-Set2-I}) and considering the coefficient of $\prod_{2\leq t\leq s+1,t\neq t_1,\dots,t_r}X_{t}^{[j_{t-1}]}$ (cf. Observation \ref{observation}), we have
     \begin{equation}\label{Equ-Set2-III}
         \begin{split}
           &n_{(E_1+\sum_{2\leq h\leq s+1,h\neq t_1,\dots,t_r}j_{h-1}E_h)}\\
           =&(\phi_M-a(\sum_{2\leq t\leq s+1,t\neq t_1,\dots,t_r}j_{t-1}))(m_{(\sum_{2\leq t\leq s+1,t\neq t_1,\dots,t_r}j_{t-1}E_{t-1})})-\sum_{l=1}^sm_{(E_l+\sum_{2\leq t\leq s+1,t\neq t_1,\dots,t_r}j_{t-1}E_{t-1})}\\
           &+\sum_{k=1}^r(-1)^{t_k}m_{(E_1+\sum_{2\leq q\leq t_k-1,q\neq t_1,\dots,t_{k-1}}j_{q-1}E_q+\sum_{t_k+1\leq q'\leq s+1,q'\neq t_{k+1},\dots,t_r}j_{q'-1}E_{q'-1})},
         \end{split}
     \end{equation}
     which implies that
     \begin{equation}\label{Equ-Set2-III'}
         \begin{split}
             &m_{(E_1+\sum_{2\leq t\leq s+1,t\neq t_1,\dots,t_r}j_{t-1}E_{t-1})}\\
             =&(\phi_M-a(\sum_{2\leq t\leq s+1,t\neq t_1,\dots,t_r}j_{t-1}))(m_{(\sum_{2\leq t\leq s+1,t\neq t_1,\dots,t_r}j_{t-1}E_{t-1})})-\sum_{l=2}^sm_{(E_l+\sum_{2\leq t\leq s+1,t\neq t_1,\dots,t_r}j_{t-1}E_{t-1})}\\
             &+\sum_{k=1}^r(-1)^{t_k}m_{(E_1+\sum_{2\leq q\leq t_k-1,q\neq t_1,\dots,t_{k-1}}j_{q-1}E_q+\sum_{t_k+1\leq q'\leq s+1,q'\neq t_{k+1},\dots,t_r}j_{q'-1}E_{q'-1})}-n_{(E_1+\sum_{2\leq h\leq s+1,h\neq t_1,\dots,t_r}j_{h-1}E_h)}.
         \end{split}
     \end{equation}
     For any $1\leq k\leq r$, put
     \[J_k:= E_1+\sum_{2\leq q\leq t_k-1,q\neq t_1,\dots,t_{k-1}}j_{q-1}E_q+\sum_{t_k+1\leq q'\leq s+1,q'\neq t_{k+1},\dots,t_r}j_{q'-1}E_{q'-1}.\]
     and then for any $1\leq k\leq r$, we have $J_k\in \Lambda_2^s$ and
     \[|J_k|=1+j_1+\sum_{3\leq q\leq s+1,q\neq t_1,\dots,t_r}j_{q-1} = k_1+\sum_{3\leq q\leq s+1}k_{q-1} = |K|.\]
     Using this, we see that (\ref{Equ-Set2-III'}) amounts to 
     \begin{equation}\label{Equ-Set2-IV}
         \begin{split}
             m_{K}=&(\phi_M-a|K-E_1|)(m_{K-E_1})-\sum_{l=2}^sm_{E_l+K-E_1}+\sum_{k=1}^r(-1)^{t_k}m_{J_k}-n_{(1,K-E_1)}\\
             =&\sum_{k=1}^r(-1)^{t_k}m_{J_k}+(\phi_M-a(|K|-1))(m_{K-E_1})-\sum_{l=2}^sm_{E_l+K-E_1}-n_{(1,K-E_1)}.
         \end{split}
     \end{equation}
     Now, we can construct $z_{K,I}^{(s)}$'s in the obvious way by using (\ref{Equ-Set2-IV}), and conclude the ``furthermore'' part in this case.
 \end{proof}

As a corollary of Lemma \ref{Lem-Set2}, we have
 \begin{cor}\label{Cor-Set2}
     Let $\Lambda_2^s$ be as above. Then for any $K=(k_1,\dots,k_s)\in\bN^s$ with $k_1\geq 1$, there are integers
     \begin{itemize}
         \item $z_{K,I}^{(s)}$ for all $I\in\Lambda_2^{s}$ with $|I|\leq |K|$ and
         \item $\epsilon_{K,J}^{(s)}$ for all $J\in\Lambda_2^s$ with $|J|<|K|$ 
     \end{itemize}
      such that if $\rd^s(f) = g$, we have 
     \[m_K=\sum_{I\in\Lambda_2^s,|I|\leq |K|}z_{K,I}^{(s)}\prod_{i=|I|}^{|K|-1}(\phi_M-ia)m_I+\sum_{J\in\Lambda_2^s,|J|< |K|}\epsilon_{K,J}^{(s)}\prod_{i=|J|+1}^{|K|-1}(\phi_M-ia)n_{(1,J)},\]
     where by convention, $\prod_{i=|I|}^{|I|-1}(\phi_M-ia):=\id_M$.


     
             
     
 \end{cor}
 \begin{proof}
     We have to construct integers $\{z^{(s)}_{K,I}\}_{I\in\Lambda_2^s,|I|\leq |K|}$ for any $K = (k_1,k_2,\dots,k_s)\in\bN^s$ with $k_1\geq 1$, which depends only on $K$. We do induction on $k_1$.

     When $k_1 = 1$, we put $\epsilon_{K,J}^{(s)} = 0$ for any $J\in\Lambda_s^2$ with $|J|<|K|$, and $z^{(s)}_{K,I} = 1$ if $I=K$ and otherwise $z^{(s)}_{K,I} = 0$. Then the result is obviously true in this case. 
     
     Now, assume $k_1\geq 2$, and then $K = E_1+L$ with $L=(l_1,k_2,\dots,k_s)$ such that $l_1\geq 1$. When $I\in\Lambda_2^s$ with $|I|=|K|$, let $x_{K,I}^{(s)}$ denote the corresponding $z_{K,I}^{(s)}$ in Lemma \ref{Lem-Set2}, and in particular, $x_{K,I}^{(s)} = 0$ as long as $|K|\neq |I|$. Then by Lemma \ref{Lem-Set2} again, we have
     \[\begin{split}
         m_K
         =&\sum_{J\in\Lambda_2^s,|J|=|K|}x_{K,J}^{(s)}m_J+(\phi_M-a|L|)m_L-\sum_{l=2}^sm_{L+E_l}-n_{(1,L)}\\
         =&\sum_{J\in\Lambda_2^s,|J|=|K|}x_{K,J}^{(s)}m_J+(\phi_M-a|L|)(\sum_{I\in\Lambda_2^s,|I|\leq |L|}z_{L,I}^{(s)}\prod_{h=|I|}^{|L|-1}(\phi_M-ah)m_I+\sum_{J\in\Lambda_2^s,|J|< |L|}\epsilon_{L,J}^{(s)}\prod_{i=|J|+1}^{|L|-1}(\phi_M-ia)n_{(1,J)})\\
         &-\sum_{l=2}^s(\sum_{I\in\Lambda_2^s,|I|\leq |L+E_l|=|K|}z_{L+E_l,I}^{(s)}\prod_{h=|I|}^{|K|-1}(\phi_M-ah)m_I+\sum_{J\in\Lambda_2^s,|J|< |K|}\epsilon_{L+E_l,J}^{(s)}\prod_{i=|J|+1}^{|K|-1}(\phi_M-ia)n_{(1,J)})\\
         &-n_{(1,L)}\\
         =&\sum_{J\in\Lambda_2^s,|J|=|K|}x_{K,J}^{(s)}m_J+\sum_{I\in\Lambda_2^s,|I|\leq |L|}z_{L,I}^{(s)}\prod_{h=|I|}^{|K|-1}(\phi_M-ah)m_I-\sum_{l=2}^s\sum_{I\in\Lambda_2^s,|I|\leq |K|}z_{L+E_l,I}^{(s)}\prod_{h=|I|}^{|K|-1}(\phi_M-ah)m_I\\
         &+\sum_{J\in\Lambda_2^s,|J|< |L|}\epsilon_{L,J}^{(s)}\prod_{h=|J|+1}^{|K|-1}(\phi_M-ah)m_{(1,J)}-\sum_{l=2}^s\sum_{J\in\Lambda_2^s,|J|<|K|}\epsilon_{L+E_l,J}^{(s)}\prod_{h=|J|+1}^{|K|-1}(\phi_M-ah)n_{(1,J)}-n_{(1,L)}\\
         =&\sum_{I\in\Lambda_2^s,|I|\leq |K|}z_{K,I}^{(s)}\prod_{h=|I|}^{|K|-1}(\phi_M-ah)m_I+\sum_{J\in\Lambda_2^s,|J|< |K|}\epsilon_{K,J}^{(s)}\prod_{h=|J|+1}^{|K|-1}(\phi_M-ah)n_{(1,J)},
     \end{split}\]
     where $z_{K,I}^{(s)}$'s and $\epsilon_{K,J}^{(s)}$'s are defined such that the last equality above holds true. The proof is complete.
 \end{proof}

  Thanks to Lemma \ref{Lem-Set2}, we are reduced to studying the $m_{K}$'s with $K=(1,k_2,*)$. The next lemma tells us if $k_2 = 0$, then $m_K$ is determined by $m_I$'s with $I\in\Lambda_3^s$ (which will be defined below) and $n_{1,*}$'s as long as $\rd^s(f) = g$.
 
 \begin{lem}\label{Lem-Set3}
     Consider the following subset  $\Lambda_3^s$ of indices: 
     \begin{itemize}
         \item If $1\leq s\leq 3$, let $\Lambda_3^s = \{(1,0,\dots,0)\}$.

         \item If $s\geq 4$, let $\Lambda_3^s:=\{(1,0,\dots,0,i_{2j+2},\dots,i_s)\in\bN^s\mid j\geq 1,i_{2j+2}\geq 1\}$.
     \end{itemize}
     Let $K=(1,0,\dots,0,k_{2j+1},\dots,k_s)$ with $j\geq 1$ and $k_{2j+1}\geq 1$. For each $I\in \Lambda_3^s$, there exists an integer $$z_{K,I}^{(s)}\in\{\pm 1,0\}$$ depending only on $s $ and $ K$ satisfying the following: If $\rd^s(f) = g$, then we have
     \[m_K=\sum_{I\in \Lambda_3^s}z_{K,I}^{(s)}m_I+(\phi_M-(|K|-1)a)m_{K-E_1}-\sum_{l=2}^sm_{K-E_1+E_l}-n_{(1,K-E_1)}.\]
     Furthermore, we have $z_{K,I}^{(s)} = 0$ in the case either $k_{2j+2}\cdots k_s\neq 0$ or $|K|\neq |I|$.
 \end{lem}
 \begin{proof}
     All results hold  trivially for $s=1, 2$. From now on, assume $s\geq 3$.
     
     Note that for any $2\leq t_1<t_2<\dots<t_r\leq s+1$ (with $1\leq r\leq s$), we still have (\ref{Equ-Set2-III}); that is, 
     \[\begin{split}
           &(\phi_M-a(\sum_{2\leq t\leq s+1,t\neq t_1,\dots,t_r}j_{t-1}))(m_{(\sum_{2\leq t\leq s+1,t\neq t_1,\dots,t_r}j_{t-1}E_{t-1})})-\sum_{l=1}^sm_{(E_l+\sum_{2\leq t\leq s+1,t\neq t_1,\dots,t_r}j_{t-1}E_{t-1})}\\
           +&\sum_{k=1}^r(-1)^{t_k}m_{(E_1+\sum_{2\leq q\leq t_k-1,q\neq t_1,\dots,t_{k-1}}j_{q-1}E_q+\sum_{t_k+1\leq q'\leq s+1,q'\neq t_{k+1},\dots,t_r}j_{q'-1}E_{q'-1})} = n_{(E_1+\sum_{2\leq h\leq s+1,h\neq t_1,\dots,t_r}i_{h-1}E_h)},
         \end{split}\]
     Assume $i\geq 1$. By letting $r\geq 2i$ and $t_1=2,\dots,t_{2i-1}=2i$, we get
     \begin{equation}\label{Lem-Set3-I}
         \begin{split}
            &n_{(E_1+\sum_{2i+1\leq h\leq s+1,h\neq t_{2i},\dots,t_r}j_{h-1}E_{h})}\\= &(\phi_M-a(\sum_{2i+1\leq t\leq s+1,t\neq t_{2i},\dots,t_r}j_{t-1}))(m_{(\sum_{2i+1\leq t\leq s+1,t\neq t_{2i},\dots,t_r}j_{t-1}E_{t-1})})\\
             &-\sum_{l=1}^sm_{(E_l+\sum_{2i+1\leq t\leq s+1,t\neq t_{2i},\dots,t_r}j_{t-1}E_{t-1})}
           +\sum_{l=2}^{2i}(-1)^{l}m_{(E_1+\sum_{2i+1\leq q'\leq s+1,q'\neq t_{2i},\dots,t_r}j_{q'-1}E_{q'-1})}\\
           &+\sum_{h=2i}^r(-1)^{t_h}m_{(E_1+\sum_{2i+1\leq q\leq t_h-1,q\neq t_{2i},\dots,t_{h-1}}j_{q-1}E_q+\sum_{t_h+1\leq q'\leq s+1,q'\neq t_{h+1},\dots,t_r}j_{q'-1}E_{q'-1})}\\
          =& (\phi_M-a(\sum_{2i+1\leq t\leq s+1,t\neq t_{2i},\dots,t_r}j_{t-1}))(m_{(\sum_{2i+1\leq t\leq s+1,t\neq t_{2i},\dots,t_r}j_{t-1}E_{t-1})})\\
             &-\sum_{l=1}^sm_{(E_l+\sum_{2i+1\leq t\leq s+1,t\neq t_{2i},\dots,t_r}j_{t-1}E_{t-1})}
           +m_{(E_1+\sum_{2i+1\leq q'\leq s+1,q'\neq t_{2i},\dots,t_r}j_{q'-1}E_{q'-1})}\\
           &+\sum_{h=2i}^r(-1)^{t_h}m_{(E_1+\sum_{2i+1\leq q\leq t_h-1,q\neq t_{2i},\dots,t_{h-1}}j_{q-1}E_q+\sum_{t_h+1\leq q'\leq s+1,q'\neq t_{h+1},\dots,t_r}j_{q'-1}E_{q'-1})}.
         \end{split}
     \end{equation}

     Let $K=(1,0,\dots,0,k_{2i+1},\dots,k_s)$ with $k_{2i+1}\neq 0$.

     Assume $k_{2i+1}\cdots k_s\neq 0$. Then let $r = 2i$ and $t_{2i} = 2i+1$ in (\ref{Lem-Set3-I}). For any $2i+2\leq h\leq s+1$, by letting $j_{h-1}=k_{h-1}$ for all $h$, we get
     \begin{equation*}
         \begin{split}
             (\phi_M-a(\sum_{t=2i+2}^sj_{t-1}))(m_{(\sum_{t=2i+2}^{s+1}j_{t-1}E_{t-1})})-\sum_{l=1}^sm_{(E_l+\sum_{t=2i+2}^{s+1}j_{t-1}E_{t-1})}=n_{(E_1+\sum_{h=2i+2}^{s+1}j_{h-1}E_{h})}.
         \end{split}
     \end{equation*}
     In other words, we have
     \begin{equation*}
         m_K = (\phi_M-(|K|-1)a)m_{K-E_1}-\sum_{l=2}^sm_{K-E_1+E_l}-n_{(1,K-E_1)}.
     \end{equation*}
     Now, one can construct $z_{K,I}^{(s)}$'s in the obvious way by applying the above equation and conclude the ``furthermore'' part in this case.

     Now, assume $k_{2i+2}\cdots k_s = 0$. Then in (\ref{Lem-Set3-I}), we put $r \geq 2i+1 $ and furthermore let $t_{2i} = 2i+1$ and $2i+3\leq t_{2i+1}<\cdots<t_r\leq s+1$ such that $k_h=0$ if and only if $h+1\in\{t_{2i+1},\dots,t_r\}$. Now, let $j_{h} = k_{h}$ for any $ h\in\{2i+1,\dots,s\}\setminus\{t_{2i+1}-1,\dots,t_{r}-1\}$ in (\ref{Lem-Set3-I}), and then we get
     \begin{equation*}
         \begin{split}
            &n_{(E_1+j_{2i+1}E_{2i+2}+\sum_{2i+3\leq t\leq s+1,t\neq t_{2i+1},\dots,t_r}j_{t-1}E_{t})}\\
            =& (\phi_M-a(\sum_{2i+2\leq t\leq s+1,t\neq t_{2i+1},\dots,t_r}j_{t-1}))(m_{(j_{2i+1}E_{2i+1}+\sum_{2i+3\leq t\leq s+1,t\neq t_{2i+1},\dots,t_r}j_{t-1}E_{t-1})})\\
            &-\sum_{l=1}^sm_{(E_l+j_{2i+1}E_{2i+1}+\sum_{2i+3\leq t\leq s+1,t\neq t_{2i+1},\dots,t_r}j_{t-1}E_{t-1})}\\
            &+\sum_{h=2i+1}^r(-1)^{t_h}m_{(E_1+j_{2i+1}E_{2i+2}+\sum_{2i+3\leq q\leq t_h-1,q\neq t_{2i+1},\dots,t_{h-1}}j_{q-1}E_q+\sum_{t_h+1\leq q'\leq s+1,q'\neq t_{h+1},\dots,t_r}j_{q'-1}E_{q'-1})}.
           \end{split}
     \end{equation*}
     In other words, we have
     \begin{equation}\label{Equ-Set3-II}
     \begin{split}
            &m_{(E_1+j_{2i+1}E_{2i+1}+\sum_{2i+3\leq t\leq s+1,t\neq t_{2i+1},\dots,t_r}j_{t-1}E_{t-1})}\\
            =& (\phi_M-a(\sum_{2i+2\leq t\leq s+1,t\neq t_{2i+1},\dots,t_r}j_{t-1}))(m_{(j_{2i+1}E_{2i+1}+\sum_{2i+3\leq t\leq s+1,t\neq t_{2i+1},\dots,t_r}j_{t-1}E_{t-1})})\\
            &-\sum_{l=2}^sm_{(E_l+j_{2i+1}E_{2i+1}+\sum_{2i+3\leq t\leq s+1,t\neq t_{2i+1},\dots,t_r}j_{t-1}E_{t-1})}\\
            &+\sum_{h=2i+1}^r(-1)^{t_h}m_{(E_1+j_{2i+1}E_{2i+2}+\sum_{2i+3\leq q\leq t_h-1,q\neq t_{2i+1},\dots,t_{h-1}}j_{q-1}E_q+\sum_{t_h+1\leq q'\leq s+1,q'\neq t_{h+1},\dots,t_r}j_{q'-1}E_{q'-1})}\\
            &-n_{(E_1+j_{2i+1}E_{2i+2}+\sum_{2i+3\leq t\leq s+1,t\neq t_{2i+1},\dots,t_r}j_{t-1}E_{t})}.
    \end{split}
    \end{equation}
           For any $h$, we define
           \[J_h:=E_1+j_{2i+1}E_{2i+2}+\sum_{2i+3\leq q\leq t_h-1,q\neq t_{2i+1},\dots,t_{h-1}}j_{q-1}E_q+\sum_{t_h+1\leq q'\leq s+1,q'\neq t_{h+1},\dots,t_r}j_{q'-1}E_{q'-1}.\]
           Then we have $J_h\in \Lambda_3^s$ (as $j_{2i+1}\neq 0$) satisfying
           \[|J_h| = 1+j_{2i+1}+\sum_{2i+3\leq t\leq s+1,t\neq t_{2i+1},\dots,t_r}j_{t-1}j_{t-1} = 1+\sum_{2i+1\leq h\leq s+1,h+1\neq t_{2i+1},\dots,t_r\}}k_h =|K|\]
           by noting that
           \[K=E_1+j_{2i+1}E_{2i+1}+\sum_{2i+3\leq t\leq s+1,t\neq t_{2i+1},\dots,t_r}j_{t-1}E_{t-1}.\] 
           Using this, (\ref{Equ-Set3-II}) amounts to that
           \begin{equation}\label{Equ-Set3-III}
               m_K=\sum_{h=2i+1}^r(-1)^{t_h}m_{J_h}+(\phi_M-a(|K|-1))m_{K-E_1}-\sum_{l=2}^sm_{K-E_1+E_l}-n_{(1,K-E_1)}.
           \end{equation}
        Now one can construct the desired $z_{K,I}^{(s)}$'s in the obvious way by applying (\ref{Equ-Set3-III}) and conclude the ``furthermore'' part in this case immediately.
 \end{proof}

 Now, we are able to construct the desired set $\Lambda^{s+1}$ as promised  in \textbf{Step 1} of Construction \ref{constr3stepssec9}. 
 To do so, we need to \emph{change the superscript from $s$ to $s+1$} (and let $n_J = 0$ for all $J$). To avoid confusion, we define these sets for any $t \geq 1$.

\begin{defn} \label{defnlambdat}
Let $t \geq 1$. 
    Let $\Lambda^{t}\subset \bN^t$ be the following union
    \[\Lambda^{t}:=\Lambda_1^{t}\cup\Lambda_3^{t}\cup\Lambda_4^{t}, \]
where $\Lambda_1^t$ and $\Lambda_3^t$ are defined in Lemmas \ref{Lem-Set1} and \ref{Lem-Set3} (for $s = t$) respectively, and $\Lambda_4^{t}$ is defined as follows:
\begin{itemize}
    \item If $t = 1$, let $\Lambda_4^t = \{1\}$.

    \item It $t\geq 2$, let $\Lambda_4^t = \{(1,i_2,\dots,i_t)\mid i_2\geq 1\}$.
\end{itemize}
We caution: the subset $\Lambda_2^t:=\{(1,i_2,\dots,i_t)\mid i_2,\dots,i_t\geq 0\}\subset\bN^t$ (cf. Lemma \ref{Lem-Set2}) is not related to the subset $\Lambda^t$; but it will be important for the construction of $f$ in Construction \ref{construction of f}.
\end{defn}

 The following proposition finishes Step 1 in Construction \ref{constr3stepssec9}, showing that an element $f$ in $\Ker(\rd^{s+1})$ is determined by its coefficients over indices in $\Lambda^{s+1}$.
 
 \begin{prop} \label{constructio of Lambda}
 Recall $s \geq 1$. Consider $\Lambda^{s+1}$ as in Def. \ref{defnlambdat}. 
     Then for any $I\in\bN^{s+1}$, there are unique integers $$\{z_{I,K}^{(s+1)}\}_{K\in\Lambda^{s+1},|K|\leq |I|}$$ depending only on $s$ and $I$ such that if $f=\sum_{I\in\bN^{s+1}}m_I\underline X^{[I]}\in\Ker(\rd^{s+1})$, i.e., 
     \[ \rd^{s+1}(f)=g=0, \]
     then 
     \[m_I=\sum_{K\in\Lambda^{s+1},|K|\leq |I|}z_{I,K}^{(s+1)}\prod_{i=|K|}^{|I|-1}(\phi_M-ia)m_K.\]
     Furthermore, when $I\in\Lambda^{s+1}$, we have 
     \begin{itemize}
         \item $z_{I,K}^{(s+1)} = 1$ if $K=I$, and

         \item $z_{I,K}^{(s+1)} = 0$ if $K\neq I$.
     \end{itemize}
 \end{prop}
 \begin{proof}
     Fix an $I=(i_1,\dots,i_{s+1})\in\bN^{s+1}$. If $I\in\Lambda^{s+1}$, then we define $z_{I,K}^{(s+1)}$ as stated. Now, we assume $I\notin\Lambda^{s+1}$.

     If $i_1=0$, we define $z_{I,K}^{(s+1)}$ by applying Lemma \ref{Lem-Set1} (after replacing $s$ by $s+1$ there) in the $g=0$ case for $K\notin\Lambda_1^{s+1}$ and $z_{I,K}^{(s+1)} = 0$ for $K\in\Lambda^{s+1}\setminus\Lambda_1^{s+1}$.

     If $i_1 = 1$, then there exists some $j\geq 1$ such that $i_2=\cdots=i_{2j}=0$ and $i_{2j+1}\neq 0$ (as $I\notin\Lambda^{s+1}$). Then we define $z_{I,K}^{(s+1)}$ by applying Lemma \ref{Lem-Set3} (after replacing $s$ by $s+1$ there) in the $g=0$ case for $K\in\Lambda_3^{s+1}$ and $z_{I,K}^{(s+1)} =0$ for $K\in\Lambda^{s+1}_4$. Then we have
     \[\begin{split}
     m_I=&\sum_{K\in\Lambda^{s+1}\setminus\Lambda^{s+1}_1,|K|\leq |I|}z_{I,K}^{(s+1)}\prod_{L=|K|}^{|I|-1}(\phi_M-la)(m_K)+(\phi_M-(|I|-1)a)m_{I-E_1}-\sum_{l=2}^{s+1}m_{I-E_1+E_l}\\
     =&\sum_{K\in\Lambda^{s+1}\setminus\Lambda^{s+1}_1,|K|\leq |I|}z_{I,K}^{(s+1)}\prod_{L=|K|}^{|I|-1}(\phi_M-la)(m_K)+(\phi_M-(|I|-1)a)\sum_{J\in\Lambda_1^{s+1},|J|\leq |I|-1}z_{I-E_1,J}^{(s+1)}\prod_{l=|J|}^{|I|-2}(\phi_M-la)m_{J}\\
     &-\sum_{l=2}^{s+1}\sum_{J_l\in\Lambda_1^{s+1},|J_l|\leq |I|}z_{I-E_1+E_l,J_l}^{(s+1)}\prod_{l=|J_l|}^{|I|-1}(\phi_M-la)m_{J_l}\\
     =&\sum_{K\in\Lambda^{s+1}\setminus\Lambda^{s+1}_1,|K|\leq |I|}z_{I,K}^{(s+1)}\prod_{L=|K|}^{|I|-1}(\phi_M-la)(m_K)+\sum_{J\in\Lambda_1^{s+1},|J|\leq |I|-1}z_{I-E_1,J}^{(s+1)}\prod_{l=|J|}^{|I|-1}(\phi_M-la)m_{J}\\
     &-\sum_{l=2}^{s+1}\sum_{J_l\in\Lambda_1^{s+1},|J_l|\leq |I|}z_{I-E_1+E_l,J_l}^{(s+1)}\prod_{l=|J_l|}^{|I|-1}(\phi_M-la)m_{J_l}\\
     =&\sum_{K\in\Lambda^{s+1}\setminus\Lambda^{s+1}_1,|K|\leq |I|}z_{I,K}^{(s+1)}\prod_{L=|K|}^{|I|-1}(\phi_M-la)(m_K)+\sum_{J\in\Lambda_1^{s+1},|J|\leq |I|}(z_{I-E_1,J}^{(s+1)}-\sum_{l=2}^{s+1}z_{I-E_1+E_l,J}^{(s+1)})\prod_{l=|J|}^{|I|-1}(\phi_M-la)m_{J},\end{split}\]
     where we put $z_{I-E_1,J}^{(s+1)} = 0$ when $|J|=|I|$.
     Put $z_{I,K}^{(s+1)}:=z_{I-E_1,K}^{(s+1)}-\sum_{l=2}^{s+1}z_{I-E_1+E_l,K}^{(s+1)}$ for $K\in\Lambda_{s+1}^1$ with $|K|\leq |I|$. Then the result holds true in this case.

     If $i_1\geq 2$, by Corollary \ref{Cor-Set2} in the $g=0$ case, there are integers $\{w^{(s)}_{I,J}\}_{J=(j_1,j_2,\dots,j_{s+1})\in\bN^{s+1},|J|\leq |I|,j_1=1}$ depending only on $s$ and $I$ such that if $\rd^{s+1}(f) = 0$, then
     \[m_I=\sum_{J=(1,j_2,\dots,j_{s+1})\in\bN^{s+1},|J|\leq |I|}w^{(s)}_{I,J}\prod_{l=|J|}^{|I|-1}(\phi_M-la)(m_J)
     .\]
     For any $J$ with $j_1 = 1$ and $|J|\leq |I|$, by what we have proved, we have
     \[m_J=\sum_{K\in\Lambda^{s+1},|K|\leq |J|}z_{J,K}^{(s+1)}\prod_{j=|K|}^{|J|-1}(\phi_M-ja)m_K.\]
     Now, one can construct the desired $z_{I,K}^{(s+1)}$'s in this case by proceeding as in the ``$i_1=1$'' case above.
 \end{proof}

\subsection{Step 2: vanishing of cocycles} \label{subsecstep2}
We proceed to carry out \textbf{Step 2} in Construction \ref{constr3stepssec9}. 
Here we will make use of both  $\Lambda^s$ and $\Lambda^{s+1}$.
Given $g\in\rC^{s+1}$ satisfying $\rd^{s+1}(g) = 0$,
\begin{itemize}
    \item we make use of $\Lambda^s$ to construct the desired $f\in\rC^s$, cf. Construction \ref{construction of f};
    \item  we then make use of  $\Lambda^{s+1}$ to check $\rd^s(f) = g$, cf. Construction \ref{constr final proof}.
\end{itemize}


 \begin{construction}\label{construction of f}
     Let $g = \sum_{I\in\bN^{s+1} }n_I\underline X^{[I]}\in\Ker(\rd^{s+1})$.
 We will define an element $f = \sum_{I\in\bN^s}m_I\underline X^{[I]}$ in the following; equivalently, we define elements $m_I$ for all $I=(i_1,\dots,i_s)$.
\begin{enumerate}
    \item[(a)] Suppose $i_1=0$. 
    \begin{itemize}
        \item If there exists some $j\geq 1$ such that $i_1=\cdots=i_{2j-1}=0$ and $i_{2j}\neq 0$ (equivalently, $(0,I)\in\Lambda^{s+1}_1$), then define $m_I = \epsilon_{I,I}^{(s)}n_{(0,I)}$ with $\epsilon_{I,I}^s$ defined in Lemma \ref{Lem-Set1} (which turns out to be $1$).
        \item  If there exists some $j\geq 1$ such that $i_1=\cdots=i_{2j}=0$ and $i_{2j+1}\neq 0$ (equivalently, $I\in\Lambda_1^s$), then define $m_I=0$.
    \end{itemize}

    \item[(b)] Suppose $i_1=1$. 
    \begin{itemize}
        \item If there exists some $j\geq 1$ such that $i_2=\cdots=i_{2j-1}=0$ and $i_{2j}\neq 0$, then define $m_I=0$. 
        \item If there exists some $j\geq 1$ such that $i_2=\cdots=i_{2j}=0$ and $i_{2j+1}\neq 0$ (equivalently, $(1,I-E_1)\in\Lambda^{s+1}_3$), then define 
    \[m_I=(\phi_M-(|I|-1)a)m_{I-E_1}-\left( \sum_{l=2}^sm_{I-E_1+E_l} \right)-n_{(1,I-E_1)}.\]
    Here, $m_{I-E_1}$ and $m_{I-E_1+E_l}$ are already defined in (a).
    \end{itemize}
    
    \item[(c)] Suppose $i_1\geq 2$ (equivalently, $(1,I-E_1)\in\Lambda_{4}^{s+1}$). 
    
    Suppose that we have constructed all $m_I$'s with $i_1\leq l$ for some $l\geq 1$. We are going to construct $m_I$'s with $i_1=l+1$. Note that we must have $I=E_1+J$ with $J=(j_1=l,j_2,\dots,j_s)$. Let $\{z_{I,K}^{(s)}\}_{K\in\Lambda_2^s}$ be the integers defined in Lemma \ref{Lem-Set2}. Then define 
    \[m_I= \left( \sum_{K\in\Lambda_3^s}z_{I,K}^{(s)}m_K \right) +(\phi_M-(|I|-1)a)m_{J}-\sum_{l=2}^sm_{J+E_l}-n_{(1,J)}.\]
    Here, $m_J$ and $m_{J+E_l}$ are defined by inductive hypothesis.
\end{enumerate}

 Now, we have constructed $m_I$ for all $I$, and hence obtain an element $f = \sum_{I\in\bN^s}m_I\underline X^{[I]}$. We have to show $f$ is a well-defined element in $\rC^s$; that is, $\lim_{|I|\to +\infty}m_I = 0$. 
 
 Let $I=(i_1,i_2,\dots,i_s)$. When $i_1=0, 1$, by Items (a) and (b) above, we can conclude by using $\lim_{|J|\to+\infty}n_{J} = 0$. Now, assume $i_1\geq 2$ and let $z_{I,J}^{(s)}$ and $\epsilon_{I,K}^{(s)}$ for $J,K\in \Lambda_2^s$ with $|J|,|K|+1\leq |I|$ be the integers defined in Corollary \ref{Cor-Set2}. Then the proof of Corollary \ref{Cor-Set2} together with Items (c) above implies that in this case, we have
    \[m_I=\sum_{J\in\Lambda_2^s,|J|\leq |I|}z_{I,J}^{(s)}\prod_{i=|J|}^{|I|-1}(\phi_M-ia)m_J+\sum_{K\in\Lambda_2^s,|K|< |I|}\epsilon_{I,K}^{(s)}\prod_{i=|K|+1}^{|I|-1}(\phi_M-ia)n_{(1,K)}.\]
As $\lim_{h\to+\infty}\prod_{i=m}^{h-1}(\phi_M-ia) = 0$ and $\lim_{|J|\to+\infty}n_J=0$, we see that $\lim_{|I|\to+\infty}m_I = 0$ as desired. So we see that $f$ is a well-defined element in $\rC^s$. 

 \end{construction}

\begin{construction}[Proof of Prop. \ref{Prop-Cohomology}(3)] \label{constr final proof}
    We now conclude the proof. Let $f$ be as in Construction \ref{construction of f}.
    Write $$g':=\rd^s(f) = \sum_{I\in\bN^{s+1}}n_I'\underline X^{[I]}.$$ It remains to prove $g'=g$.
    By Prop. \ref{constructio of Lambda}, it suffices to check
    \[ n_L =n_L', \forall L \in\Lambda^{s+1}. \]
    We will freely use the Items (a), (b), (c) in Construction \ref{construction of f} above.
\begin{itemize}[leftmargin=0cm]
    \item 
If $L=(0,I)\in\Lambda^{s+1}_1$ with $I=(i_1,\dots,i_s)$. Then there exists some $j\geq 1$ such that $i_1=\cdots=i_{2j-1} = 0$ and $i_{2j}\neq 0$. By Item (a) above, we have $n_L=\epsilon_{I,I}^{(s)}m_I$. On the other hand, by Lemma \ref{Lem-Set1}, we have $n'_L=\epsilon_{I,I}^{(s)}m_I$. Therefore, $n_L=n_L'$.
  \item 
If  $L=(1,I-E_1)\in \Lambda^{s+1}_3$ with $I=(1,i_2,\dots,i_s)$ such that $i_2=\cdots=i_{2j} =0$ and $i_{2j+1}\neq 0$ for some $j\geq 1$. By Lemma \ref{Lem-Set3} and Item (b) above, we have
\[n_L'=\sum_{J\in\Lambda_3^s}z_{I,J}^{(s)}m_J+(\phi_M-(|I|-1)a)m_{I-E_1}-\sum_{l=2}^sm_{I-E_1+E_l}-m_I=\sum_{J\in\Lambda_3^s}z_{I,J}^{(s)}m_J+n_L.\]
By Item (a)   above, we have $m_J = 0$, for any $J\in\Lambda_3^s$. Thus, we conclude that $n_L=n_L'$.
  \item 
If  $L=(1,I)\in\Lambda_4^{s+1}$ with $I=(i_1,\dots,i_s)$ such that $i_1\geq 1$, put $J=I+E_1$. Let $\{z_{J,K}^{(s)}\}_{K\in\Lambda_2^s,|K|=|J|}$ be the integers defined in Lemma \ref{Lem-Set2}. By Lemma \ref{Lem-Set2}, we have
\[n_L=\sum_{K\in\Lambda_2^s}z_{J,K}^{(s)}m_K+(\phi_M-(|J|-1)a)m_I-\sum_{l=1}^sm_{I+E_l}.\]
By Item (c), we see that for any $L\in\Lambda_4^{(s+1)}$, we have $n_L=n_L'$.
\end{itemize}

In conclusion, we have proved $\rd^s(f)=g$. This implies $\rH^{s+1}(\rC(M,\varepsilon_M)) = 0$ for all $s\geq 1$.
\end{construction}

  \bibliographystyle{alpha}

\end{document}